\documentclass[11pt]{ArtEnFancyClassPDC}   
\usepackage{AuxiliarPDC}

\newcommand{\muf}[1][\varphi]{\mathrm{m}_{\scriptstyle #1}}
\newcommand{\tmuf}{\tilde{\mathrm{m}}_{\scriptstyle\varphi}}

\newcommand{\mux}[1][x]{\mu^u_{#1}}
\newcommand{\mufx}{\mu^u_{fx}}

\newcommand{\msx}[1][x]{\mu^s_{#1}}

\newcommand{\mcux}[1][x]{\mu^{cu}_{#1}}

\newcommand{\mcsx}[1][x]{\mu^{cs}_{#1}}

\newcommand{\nux}[1][x]{\nu^u_{#1}}
\newcommand{\nsx}[1][x]{\nu^s_{#1}}

\newcommand{\Gcal}{\mathcal{G}}

\newcommand{\Jac}[1][x_0,y_0]{\mathrm{Jac}^{\Gcal}_{#1}}
\newcommand{\Jacs}[1][x_0,y_0]{\mathrm{Jac}^{s}_{#1}}
\newcommand{\Jacu}[1][x_0,y_0]{\mathrm{Jac}^{u}_{#1}}

\newcommand{\Pcs}[1]{\ensuremath{P^{cs}(#1)}}
\newcommand{\Pcu}[1]{\ensuremath{P^{cu}(#1)}}

\newcommand{\Benu}[1][x]{D^u(#1,\epsilon,n)}

\newcommand{\Eq}{\mathscr{Eq}}

\newcommand{\clps}{\mathrm{c}_{\scriptscriptstyle{lps}}}
\newcommand{\ccen}{\mathrm{c}_{\scriptscriptstyle{cen}}}
\newcommand{\ccov}{\mathrm{c}_{\scriptscriptstyle{cov}}}
\newcommand{\cpsm}{\mathrm{c}_{\scriptscriptstyle{psm}}}

\newenvironment{proofw}{\noindent\textit{Proof.}}{\hfill$\square$} 

\newcommand{\dermm}[1][x]{g_{f^{-1}\xi}(#1)}

\hypersetup{
	unicode=true,          
	pdftoolbar=true,        
	pdfmenubar=true,        
	pdffitwindow=false,     
	pdfstartview={FitH},    
	pdftitle={Equilibrium States for Center Isometries},    
	pdfauthor={Pablo D. Carrasco and Federico Rodriguez-Hertz},     
	pdfkeywords={Equilibrium States, Center Isometries, Conditional Measures, SRB measures}, 
	pdfnewwindow=true,      
	colorlinks=true,       
	linkcolor=arsenic,          
	citecolor=green,        
	filecolor=magenta,      
	urlcolor=cyan           
}

\begin{document}


	\title{Equilibrium States for Center Isometries}
	\author[1]{Pablo D. Carrasco \thanks{pdcarrasco@gmail.com; partially supported by FAPEMIG Universal APQ-02160-21}}
	\author[2]{Federico Rodriguez-Hertz \thanks{hertz@math.psu.edu; partially supported by NSF grant DMS-1900778}}
    \affil[1]{ICEx-UFMG, Avda. Presidente Antonio Carlos 6627, Belo Horizonte-MG, BR31270-90}
	\affil[2]{Penn State, 227 McAllister Building, University Park, State College, PA16802}

	\date{\today}

\maketitle

\begin{abstract}
We develop a geometric method to establish existence and uniqueness of equilibrium states associated to some H\"older potentials for center isometries (as are regular elements of Anosov actions), in particular the entropy maximizing measure and the SRB measure. It is also given a characterization of equilibrium states in terms of their disintegrations along stable and unstable foliations. Finally, we show that the resulting system is isomorphic to a Bernoulli scheme.
\end{abstract}
\tableofcontents

\section{Introduction and Main Results} 
\label{sec:introduction_and_main_results}

We will study some aspects of the thermodynamic formalism for maps related to higher rank Anosov actions. Let $M$ be a compact metric space and $f:M\rightarrow M$ be an homeomorphism. For $x\in M,\epsilon>0, n\in \Nat$ we denote by $D(x,\ep)$ the open disc of center $x$ and radius $\ep$, and by $\Ben$ the open $(\epsilon,n)$-Bowen ball centered at $x$,
\[
\Ben=\{y\in M:d(f^jx,f^jy)<\epsilon,\ j=0,\ldots,n-1 \}.
\]
Recall that the topological pressure of the (continuous) potential $\varphi:M\to \Real$ is given by
\[
P_{\mathit{top}}^f(\varphi)=\Ptop(\varphi):=\lim_{\epsilon\mapsto 0}\limsup_{n\mapsto\oo}\frac{\log S(\epsilon,n)}{n}
\]
where  
\[
S(\epsilon,n)=\inf\{\sum_{x\in E}e^{\SB(x)}:M=\cup_{x\in E}D(x,\epsilon,n)\}\quad \SB=\sum_{i=0}^{n-1}\varphi\circ f^i.
\]
Denote by $\PTM{f}{M}$	the set of $f$-invariant Borel probability measures on $M$. For $\nu\in \PTM{f}{M}$ the pressure of $\varphi$ with respect to $\nu$ is 
\[
P_{\nu}(\varphi):=\nu(\varphi)+h_{\nu}(f).
\]
\begin{theorem*}[Variational Principle]
	It holds
	\[
	\Ptop(\varphi)=\sup_{\mathclap{\nu\in \PTM{f}{M}}}\ P_{\nu}(\varphi)
	\]
\end{theorem*}
This by now classical theorem is due to Walters \cite{WaltersPres}.

\begin{definition}
	A measure $\mm\in \PTM{f}{M}$ is an equilibrium state for the potential $\varphi$ if $\Ptop(\varphi)=P_{\mm}(\varphi)$.
\end{definition}

The existence, uniqueness and general properties of equilibrium states is of utmost relevance in ergodic theory, in particular for the differentiable one. Indeed, one of the most important theorems in smooth ergodic theory is the following.

\begin{theorem*}
	Let $M$ be a closed (that is, compact and without a boundary) manifold and consider a $\mathcal{C}^1$ diffeomorphism $f:M\rightarrow M$ having a hyperbolic attractor $\Lambda$ (in particular, $f|\Lambda$ is transitive). Then for every H\"older potential $\varphi:\Lambda\rightarrow \Real$ there exists a unique equilibrium state $\muf$ for $f|\Lambda$. Furthermore, if $f|\Lambda$ is topologically mixing then the system $(f,\muf)$ is metrically isomorphic to a Bernoulli shift. 
\end{theorem*}
As stated, this theorem is a culmination of the work of several authors. The existence and uniqueness is essentially due to Sinai, Ruelle and Bowen \cite{SinaiMarkov},\cite{SRBattractor},\cite{Bowen1974}, while the Bernoulli property is mainly due to Ornstein and  Weiss \cite{GeoBernoulli}. See also \cite{BowenBernoulli}.

In this work we investigate the thermodynamic formalism for a generalization of the above systems, namely when instead of a single hyperbolic map we have an Anosov action. Let us recall that a differentiable action of a Lie Group on a compact Riemannian manifold $\alpha:G\curvearrowright M$ is Anosov if the $G$-orbits of $\alpha$ foliate $M$ and there exists $g\in G$ an Anosov element, i.e.\@ there exist an $\alpha(g)$-invariant decomposition $TM=E^s\oplus E^c\oplus E^u$ and $C>0,0<\lambda<1$ such that $E^c_x=T_x(G\cdot x)$ and with respect to the metric on $M$ it holds
\begin{align*}
&v\in E^s_x\Rightarrow \forall n\geq 0\quad \norm{D_x\alpha(g)^n(v)}\leq C\lambda^n\norm{v}\\
&v\in E^u_x\Rightarrow \forall n\leq 0\quad \norm{D_x\alpha(g)^n(v)}\leq C\lambda^{-n}\norm{v}.\\
\end{align*}
In this case it is possible to modify the metric to guarantee that $\alpha(g)$ is an isometry when restricted to each $G$-orbit. Diffeomorphisms preserving a splitting with the above properties are called \emph{center isometries} and will be our main focus. See the next section for the precise definition, or the classical reference \cite{HPS} for further information about these systems.

Anosov flows (cf.\@ \cite{AnosovThesis}) constitute the most well understood example of Anosov action, where an analogous theorem as the previous one holds \cite{BowenBernoulli}. The proof relies on the existence of Markov partitions \cite{SymbHyp}, which permit to represent these flows as suspension over a symbolic system; since this technology is unavailable for Anosov actions of rank $\geq 2$, we present a geometric construction of equilibrium states for that context. At this stage however we need to impose some condition on the potential (being constant along the orbits, or being the geometrical potential associated to a regular element). Besides the fact that these cases include the most interesting types of potentials (corresponding to the entropy maximizing measure and the SRB), we also believe that the techniques developed are interesting on their own, and may be applicable to more situations.

For a center isometry $f$ the associated bundles $E^s,E^u,E^c, E^{cs}=E^s\oplus E^c, E^{cu}=E^c\oplus E^u$ are integrable to continuous foliations with smooth leaves $\Fs,\Fu,\Fc, \Fcs, \Fcu$\ called the stable, unstable, center, center stable and center unstable foliations, respectively. If $W$ is a leaf of one of these foliations, then the Riemannian metric of $M$ induces a corresponding metric on $M$, and we will consider $W$ equipped with the topology induced by this distance. Our results require the assumption that $\Fs,\Fu$ are minimal foliations, i.e.\@ every one of their leaves is dense (this condition is implicit in the classical rank-one case cited above, due to the mixing hypothesis).

\begin{theoremA}[$c$-constant case]\hypertarget{theoremA}{}
	Let $f:M\rightarrow M$ be a center isometry of class $\mathcal{C}^2$ and $\varphi:M\rightarrow \Real$ a H\"older potential that is constant on center leaves. Assume that the stable and unstable foliations of $f$ are minimal. Then there exist $\muf \in \PTM{f}{M}$ and families of measures $\mu^u=\{\mux\}_{x\in M}, \mu^{cu}=\{\mcux\}_{x\in M}, \mu^s=\{\mux\}_{x\in M}, \mu^{cs}=\{\mcsx\}_{x\in M}$ satisfying the following.
    	\begin{enumerate}
		\item The probability $\muf$ is the unique equilibrium state for the potential $\varphi$.
		\item For every $x\in M$ the measure $\mu^{\ast}_x,\ast\in \{u,s,cu,cs\}$ is a Radon measure on $W^{\ast}(x)$ which is positive on relatively open sets, and $y\in W^{\ast}(x)$ implies $\mu^{\ast}_x=\mu^{\ast}_y$. 
		\item It holds for every $x\in M$
		\begin{enumerate}
			\item $\mu^{\sigma}_{fx}=e^{\Ptop(\varphi)-\varphi}f_{\ast}\mu^{\sigma}_{x}\quad \sigma\in\{u,cu\}$.
			\item $\mu^{\sigma}_{fx}=e^{\varphi-\Ptop(\varphi)}f_{\ast}\mu^{\sigma}_{x}\quad\sigma\in\{s,cs\}$. 
		\end{enumerate}
		\item If $\xi$ is a measurable partition that refines the partition by unstable (stable) leaves then the conditionals $(\muf)^{\xi}_x$ of $\muf$ are equivalent to $\mux$ (resp. $\msx)$ for $\muf\aep(x)$.
		\item For every $\epsilon>0$ sufficiently small, for every $x\in M$ the measure $\muf|D(x,\epsilon)$ has  product structure with respect to the pairs $\mux,\mcsx$ and $\msx,\mcsx$, i.e.\@ it is equivalent to $\mux\times \mcsx$ and to $\msx\times \mcux$ .
		\item For every $\epsilon>0$ there exist some positive constants $a(\epsilon),b(\epsilon)$ satisfying, for every $x\in M,n\geq 0$
		\[
		a(\epsilon)\leq\frac{\muf(D(x,\epsilon,n))}{e^{\SB(x)-n\Ptop(\varphi)}}\leq b(\epsilon).
		\]
	\end{enumerate} 
\end{theoremA}

\begin{theoremAp}[SRB case]\hypertarget{theoremAp}{}
	Let $f:M\rightarrow M$ be a center isometry of class $\mathcal{C}^2$ and $\varphi=-\log |\det Df|E^u|$. Assume that the unstable foliation of $f$ is minimal, and let $\mu_x=\{\mu_x^u\}$ be the family of Lebesgue measures on unstable leaves, that is, $\mu_x^u$ is the volume on $\Wu{x}$ corresponding to the induced metric. Then there exist $\muf \in \PTM{f}{M}$ and a family of measures $\mu^{cs}=\{\mcsx\}_{x\in M}$ satisfying the following.
	\begin{enumerate}
		\item The probability $\muf$ is the unique equilibrium state for the potential $\varphi$.
		\item For every $x\in M$ the measure $\mu^{cs}_x$ is a Radon measure on $W^{cs}(x)$ which is positive on relatively open sets, and $y\in W^{cs}(x)$ implies $\mu^{cs}_x=\mu^{cs}_y$. 
		\item It holds for every $x\in M$,
		\begin{enumerate}
			\item $\mu^{cs}_{fx}=|\det Df|E^u|^{-1}\cdot f_{\ast}\mu^{cs}_{x}$.
			\item $\mu^{u}_{fx}=|\det Df|E^u|\cdot f_{\ast}\mu^{u}_{x}$   (change of variables theorem).
		\end{enumerate}
        \item If $\xi$ is a measurable partition that refines the partition by unstable leaves then the conditionals $(\muf)^{\xi}_x$ of $\muf$ are equivalent to $\mux$
        \item For every $\epsilon>0$ sufficiently small, for every $x\in M$ the measure $\muf|D(x,\epsilon)$ has  product structure with respect to the pair $\mux,\mcsx$. 
		\item For every $\epsilon>0$ there exist some positive constants $a(\epsilon),b(\epsilon)$ satisfying, for every $x\in M,n\geq 0$
		\[
		a(\epsilon)\leq\frac{\muf(D(x,\epsilon,n))}{|\det D_{f^n(x)}f^{-n}|E^u(f^n(x))|}\leq b(\epsilon).
		\]
	\end{enumerate} 
\end{theoremAp}

Both theorems are similar, except that in \hyperlink{theoremAp}{Theorem A'} we cannot guarantee the existence of a  family $\mu^s=\{\msx\}_{x\in M}$ on stable leaves satisfying the quasi-invariance property $\mu^{s}_{fx}=|\det Df|E^u|^{-1}\cdot f_{\ast}\mu^{s}_{x}$. We point out however that if we assume that also $\Fs$ is minimal, then there exists a family $\mu^{cu}=\{\mcux\}_{x\in M}$ satisfying the corresponding quasi-invariance. See below for more discussion on the existence of these families. 

\begin{corollaryA}\hypertarget{corollaryA}{}
	Let $\alpha:G\curvearrowright M$ be an Anosov action of class $\mathcal{C}^2$. Suppose that $g\in G$ is an Anosov element such that its associated stable and unstable foliations are minimal and denote $f=\alpha(g)$. Let $\varphi:M\rightarrow \Real$ be a H\"older potential that is either constant along orbits of $\alpha$, or the geometrical potential associated to $f$. Then there exist a unique equilibrium state $\muf \in\PTM{f}{M}$ for $\varphi$ and families of measures $\mu^u=\{\mux\}_{x\in M},\mu^{cs}=\{\mcsx\}_{x\in M}$ satisfying the $1-6$ of \hyperlink{theoremA}{Theorem A}.
\end{corollaryA}

In the rank-one case the previous Theorem (Corollary) was established first by G. Margulis \cite{TesisMarg} for the entropy maximizing measure ($\varphi\equiv 0$), and later by R. Bowen and D. Ruelle \cite{BowenRuelle} for general potentials, except for $2,3$ which appear to be new in this generality; the local product structure is due to N. Haydn \cite{localproductstructureHaydn} (see also \cite{Leplaideur2000}). The last item above is the so called Gibbs property. See \cite{EquSta} and \cite{Haydn1992}. The uniqueness of the equilibrium state is known from other works, for example \cite{Climenhaga2020} (see also below for the SRB case).

\smallskip 

It is natural to inquiry about the requirements imposed to the potentials. It turns out that in general there may fail to exist families of measures satisfying the quasi-invariance property $3$ of \hyperlink{theoremA}{Theorems A, A'}.

\begin{theoremD}\hypertarget{theoremD}{}
There exists an analytic center isometry $f:\Tor^3\to\Tor^3$ and a $\mathcal{C}^{\oo}$ potential $\varphi:\Tor^3\to\Real$ for which there is no  family of measures $\{\zeta_x^u\}_{x\in M}$ satisfying
\begin{itemize}
	\item $f^{-1}_{\ast}\zeta_{fx}^u=e^{\varphi - \Ptop(\varphi)}\zeta_x^u$;
	\item the measures depend continuously on the point $x$.
\end{itemize}
\end{theoremD}

\smallskip

Let us say a few words about the SRB case. The existence of SRB measures in the context that we are working also follows from the general results in \cite{GibbsSinaiPesin}, while the uniqueness was initially proven by Dolgopyat in \cite{lecturesugibbs}. Nevertheless, the other parts of the theorem are new, and therefore we are able to make a contribution to the study of SRB measures for partially hyperbolic systems in a setting that has not been explored thoroughly yet (i.e.\@ zero Lyapunov exponents in the center direction).

At this point we bring to the attention of the reader the following characterization of such measures: an invariant measure $\mm$ is an SRB (that is, an equilibrium state for $\varphi(x)=-\log|\det Df|E^u_x|$) if and only if for every measurable partition subordinated to the unstable foliation, the conditional measures induced by $\mm$ are absolutely continuous with respect to Lebesgue on the corresponding leaf. Here we are able to give similar characterization for equilibrium states, that is new even in the classical hyperbolic case; on the other hand, for certain non-uniformly interval expanding maps there exist similar characterizations \cite{uniqueeqpiecewise}, \cite{buzzitermo}.  We highlight that the result relies on the existence of a family of measures as given in \hyperlink{theoremA}{Theorem A}, and could be applied under different hypotheses that guarantee such existence. This is the case for time-one maps of (mixing) Anosov flows due to the results of Haydn cited above (see also \Cref{sec:applications_and_concluding_remarks}), or for Anosov diffeomorphisms (where the condition of the potential being constant on center leaves is automatic).
      
\begin{theoremB}\hypertarget{theoremB}{}
Let $f:M\to M$ be a $\mathcal C^2$ center isometry and $\varphi:M\to\Real$ a Hölder potential. Suppose the existence of a family $\mu^u=\{\mux\}_{x\in M}$ satisfying $2,3$ of \hyperlink{theoremA}{Theorem A}. For an $f$-invariant measure $\mm$, the following conditions are equivalent.
\begin{enumerate}
\item $\mm$ is an equilibrium state for $\varphi$.
\item For every $\mm$-measurable partition $\xi$ subordinated to $\Fu$, the induced conditional measures are absolutely continuous with respect to $\{\mux\}_x$.
\end{enumerate}
Furthermore, if there also exists a family $\mu^s=\{\msx\}_{x\in M}$ satisfying $2,3$ of \hyperlink{theoremA}{Theorem A}, then conditions $1$ and $2$ are also equivalent to:

\begin{enumerate}
\item[3.] for every $\mm$-measurable partition $\xi$ subordinated to $\Fs$, the induced conditional measures are absolutely continuous with respect to $\{\msx\}_x$. 
\end{enumerate}	
\end{theoremB}

The last part is somewhat surprising in the SRB case (for systems where the stable measures exists, as mixing hyperbolic ones), because in general $\msx$ will not be absolutely continuous with respect to Lebesgue; assuming existence of these Condition $3$ is just Condition $2$ for $f^{-1}$.

\begin{remark}\label{rem:valeparaLyaczero}
More generally, in the previous theorem instead of $f$ being a center isometry one can assume that for every $f$-invariant measure the center Lyapunov exponents vanish. The proof is the same.
\end{remark}

As a consequence we get this new piece in the ergodic theory of Anosov flows (or diffeomorphisms).

\begin{corollaryB}\hypertarget{corollaryB}{}
Given a transitive Anosov flow $\Phi=\{f^t\}_t$ of class $\mathcal{C}^2$ that is not a suspension, there exists a family $\mu^u\in\mathrm{Meas}(\Fu)$ satisfying the following: for any $\mm\in\PM[M]$, if
\begin{enumerate}
 	\item $\mm$ is $f^t$-invariant for some $t\neq \emptyset$, and
 	\item $\mm$ has conditionals absolutely continuous with respect to $\mu^u$,
 \end{enumerate}
 then $\mm$ is the unique equilibrium state for the system $(\Phi,\varphi)$. 
\end{corollaryB}

\smallskip

Due to the uniqueness of the equilibrium state given in \hyperlink{theoremA}{Theorems A,}\hyperlink{theoremAp}{A'} it follows from convexity of the pressure function that it is ergodic, however this type of general argument does not yield any stronger metric property. Here we will rely on the precise description that we have of the equilibrium state and employ geometrical methods to establish that in fact this measure is a Bernoulli measure. Again, the proof is written in more generality than what was assumed in the aforementioned theorems.

\begin{theoremC}\hypertarget{theoremC}{}
Let $f:M\to M$ be a $\mathcal C^2$ center isometry and $\varphi:M\to \Real$ a potential that is constant on center leaves, or the SRB potential. Suppose that the unstable foliation of $f$ is minimal and let $\muf$ be the (ergodic) equilibrium state for $(f,\varphi)$. Then the system $(f,\muf)$ is metrically isomorphic to a Bernoulli shift.

In particular, if $f:M\to M$ is conservative, then the volume is an Bernoulli measure for $f$. 
\end{theoremC}

\smallskip
 
 The methods developed in this article are of geometric flavor, and it is the hope of the authors that they can shed on light in the still under development theory of thermodynamic formalism for partially hyperbolic systems. Evidence of their efficacy appears in a companion article \cite{ContributionsErgodictheory}, where we establish new results in the ergodic theory of the classical rank-one case.

 \paragraph{Comparison with other works}Other geometrical constructions exist in the literature, particularly the work of Climenhaga, Pesin and Zelerowicz \cite{Climenhaga2020} based on geometric measure theory, and the work of Spatzier and Vischer \cite{Spatzier2016} studying compact extensions of Anosov systems. Our results have some intersection with the first cited work, but our methods are completely different, and allow us to have more control over the equilibrium measures, particularly in terms of their conditional measures along the invariant foliations. Climenhaga et al. obtain the equilibrium state in some more general settings that the ones considered here: they construct families of measures $\mu^u$ as the ones described in \hyperlink{theoremA}{Theorem A} (that is, satisfying $2,3,4$) and use them to construct an unique equilibrium state for the system satisfying $6$. On the other hand, they do not seem to get much control on the transverse measures (the $\mcsx$ in our theorems), and their existence is deduced only locally (disintegrating in small rectangles). Our approach is different: we first construct the measures $\mcsx$ globally, and then use them to define the equilibrium measure in cases where a family $\mu^u$ satisfying $2,3$ of \hyperlink{theoremA}{Theorem A} exists. We remark that for the potentials considered in \cite{Climenhaga2020} we could use the existence of $\mu^u$ given their work to construct $\muf$, and in particular verify the hypotheses of \hyperlink{theoremB}{Theorem B}, and perhaps \hyperlink{theoremC}{Theorem C} with some additional work (for the Bernoulli property some knowledge on the unstable measures is required). The extra control on the transverse measure is used in \cite{ContributionsErgodictheory} to deduce an strengthening of Furstenberg-Marcus theorem on unique ergodicity of the unstable foliation of mixing Anosov flows.

 In contrast, the precise characterization given in \hyperlink{theoremB}{Theorem B} for equilibrium states is (as far as we know) new in the literature. In addition, \hyperlink{theoremC}{Theorem C} establishes very fine statistical properties for the equilibrium state, using geometrical methods (julienne quasi-conformality) that probably could be used in other situations. It may be fruitful, for example, to analyze a possible synergy between these methods and the results of Climenhaga, Pesin and Zelerowicz. 

 We point out that in \cite{LiWuBernoulli} it is also shown that the equilibrium measure built by \cite{Climenhaga2020} is Bernoulli. In this work we present two proofs of the fact the system is Kolmogorov, one essentially the same as \cite{LiWuBernoulli} (which is due to Ledrappier), and the other more geometrical one based in arguments developed by Pugh and Shub to study (stably) ergodicity of conservative partially hyperbolic maps. For the Bernoulliness, we adapt the method of Ornstein and Weiss. This is also the approach followed in Li and Wu's work, however the authors do not give all details, and seem to leave some not obvious parts for the reader to complete. Due to the importance of this theorem we elaborate more on the proof. 

 Additional related work appears in the manuscript \cite{Buzzi2019} by Buzzi, Fisher and Tahzibi, containing some results similar to ours for perturbations of time-one maps of Anosov flows in dimension three in the case of the entropy maximizing measure, and in the recent preprint \cite{Bonthonneau2021} by Bonthonneau, Guillarmou and Welch, studying SRB measures for the action, instead of a regular element as in our case. The techniques in this last cited work are of functional analytic type, and it seems that our approaches may be complementary. This remains as in interesting research line to be investigated. 

\paragraph{Systems verifying the hypotheses of the Theorems and related questions} The conditions imposed on the potential $\varphi$, being constant on center leaves or equal to the geometrical potential, are used to guarantee the existence of the family $\mu^u$. As it can be seen using \hyperlink{theoremD}{Theorem D}, some additional assumptions are necessary, although at this stage it is still not clear what are those.

\begin{question}
 If $f:M\to M$ is a $\mathcal C^2$ center isometry with minimal strong foliations and $\varphi:M\to \Real$ is a Hölder potential, what are the conditions that guarantee the existence of a family $\mu^u$ satisfying $2, 3$ of \hyperlink{theoremA}{Theorem A}?
 \end{question}

In particular:

\begin{question}
Assume the hypotheses of \hyperlink{theoremA}{Theorem A} and furthermore that $f$ is accessible, meaning that any bi-saturated set (that is, a set that is union of both stable and unstable leaves). Does there exist a continuous family $\mu^u\in\mathrm{Meas}(\Fu)$ in this case?
\end{question}

We remark that the example given in the \hyperlink{sec:Appendix}{Appendix} does not satisfy this condition, but only essential accessibility (any bi-saturated set by $\Fs,\Fu$ has Lebesgue measure either zero or one).

\smallskip 
 
Requiring that $\varphi$ is constant on center leaves is usually a strong requirement, because in several situations there exist dense center leaves, which forces $\varphi$ to be constant. Although this may be true, the constant zero potential corresponds to the entropy maximizing measure of the system which is undoubtedly important. On the other hand, there are several interesting cases without dense center leaves (for example, group extensions of hyperbolic systems), where the condition of constancy on center leaves includes an infinite dimensional space of potentials. 

 For the SRB potential the existence of $\mu^u$ is automatic, and the study of the associated equilibrium state is paramount in smooth ergodic theory (including for example the conservative case). We emphasize however that given the existence of $\mu^u$ the results of this article can be applied. There are situations where this is the case, for example:
 \begin{enumerate}
 	\item $f=$ time-one map of a mixing Anosov flow (see \Cref{sub:hyperbolic_flows}).
 	\item The potentials having the $cs$-Bowen property considered by Climenhaga, Pesin and Zelerowicz in \cite{Climenhaga2020} (cf. \Cref{lem:Bowenpropertycs}). This includes in particular the family of potentials $\varphi_t=-t\log|\det Df|E^u|$.  
 \end{enumerate}
This is the reason why we have written our results in more generality, with the hope that it could be used in other situations. Observe for example that in the first case the family $\mu^s$ also exists, and therefore by \hyperlink{theoremB}{Theorem B} we get that the (unique) SRB measure is characterized by a family of measures on the stable (which are not absolutely continuous, unless the system is conservative).   

\smallskip 
 
The proof of \hyperlink{theoremB}{Theorem B} is based on the abstract methods used by Ledrappier and Young for the SRB measure, and in particular apply also for systems whose Lyapunov exponents in the center direction are zero for every invariant measure. One is led to ask whether this can be extended in the non-uniformly hyperbolic setting, although this seems to be out of reach at the moment. Perhaps one could start with the following.

\begin{question}
Suppose that $f:M\to M$ is a $\mathcal C^2$ diffeomorphism, $\varphi:M\to\Real$ is a Hölder potential and assume that the system $(f,\varphi)$ has an non-uniformly hyperbolic equilibrium state $\muf$ (in the sense of Pesin \cite{LyaPesin}). Construct a family of $\{\mu^u_x\}_{x\in M'}$ where $\muf(M')=1$ satisfying the quasi-invariance hypothesis $2$ of \hyperlink{theoremA}{Theorem A}.
\end{question}

 The fact that equilibrium states are determined by some families of measures on a dynamical structure (the invariant foliations) seem to indicate strong restrictions on the geometry of these measures. In \Cref{sec:a_characterization_of_equilibrium_states} we establish the existence of a family of measures $\{\nu^u_x\}_{x\in M}$ on the unstable leaves such that if $\mm$ is an invariant measure satisfying: given $\xi$ a $\mm$-measurable partition subordinated to $\Fu$, the conditional measure of $\mm$ on the atom $\xi(x)$ is equal to $\nux(\ \cdot\ |\xi(x))$, for $\mm-$a.e.$(x)$. Then $\mm$ is an equilibrium state of the system.

\begin{question}
Assume that $f:M\to M$ is a $\mathcal C^2$ center isometry, $\varphi: M\to\Real$ is a Hölder potential, and suppose the existence of a family $\nu^u=\{\nu^u_x\}_{x\in M}$ as explained above. If $\mm$ is a (non necessarily invariant) probability with conditionals given by $\nu^u$, is $\mm$ an equilibrium state?
\end{question}   

Even though this seems unlikely due to the lack of control in the transverse directions, in \cite{ContributionsErgodictheory} we give a positive answer to this problem for codimension-one Anosov flows.

\paragraph{Organization of the rest of the article}The rest of the paper is organized as follows. In \Cref{sec:measures_along_foliations} we develop general arguments to establish the existence of families of measures on foliations with prescribed invariance properties (conformality). These results are then used in \Cref{sec:measures_along_invariant_foliations_of_center_isometries} to construct the equilibrium state for a center isometry as in Theorem A. Then in \Cref{sec:a_characterization_of_equilibrium_states} we tackle the problem of the characterization of equilibrium states in terms of their conditional measures on stable/unstable leaves. \Cref{sec:finer_ergodic_properties} is devoted to the proof the Bernoulli property of the considered system, and in the last Section we include a discussion of the rank-one case (in particular, the proof of \hyperlink{corollaryB}{Corollary B}), together with some questions. The proof of \hyperlink{theoremD}{Theorem D} is given in an \hyperlink{sec:Appendix}{Appendix}, since it can be read almost independently of the rest of the work (although \hyperlink{theoremA}{Theorem A} and \hyperlink{theoremB}{Theorem B} are used). 

\paragraph{Acknowledgements} We are very thankful to the referee for her/his careful reading of the manuscript, giving  several thoughtful suggestions, bringing to our attention some references and pointing out errors in the original version of this work.

The first author would like to thank the hospitality of the Penn State Math department, where this project was initiated.

\section{Measures along foliations} 
\label{sec:measures_along_foliations}

For $r\geq 1$ we denote by $\diffM{M}{r}$ the group of $\mathcal{C}^r$ diffeomorphisms of a manifold $M$. Let $f\in \diffM{M}{r}$ and $\mu$ be an $f$-invariant measure. Recall that (see \cite{Rokhlin}) given a $\mu$-measurable partition $\xi$ of $M$ we can disintegrate $\mu$ in a family of conditional probabilities $\{\mu_x^{\xi}\}_{x\in M'}$ with $\mu(M')=1$ and such that for every $f\in \Lp[1](M,\mu)$ the conditional expectation of $f$ in $\mathcal{B}_{\xi}$, the $\sigma$-algebra generated by $\xi$, is given by 
\[
	\ie{\mu}{f|\mathcal{B}_{\xi}}=\int f d\mu_x^{\xi},\quad x\in M'.
\]
Furthermore, if $\xi$ is countably generated (that is, $\xi=\bigvee_n \xi_n$ with $\xi_n$ finite partition for every $n\in\Nat$) then for $x,y\in M'$, if $y$ belongs to the atom of $\xi$ containing $x$ then $\mu^{\xi}_x=\mu^{\xi}_y$. In this case, denote by $M_{\xi}=M/\xi$ the space of atoms of $\xi$ and consider $\pi_{\xi}: M\rightarrow  M_{\xi}$, $\pi_{\xi}(x)=\xi(x)$, the atom that contains $x$ (this map is well defined $\mu$-a.e.). We equip $M_{\xi}$ with the $\sigma$-algebra induced by $\pi_{\xi}$ and consider the \emph{transverse measure}\footnote{From now on we omit the $\ast$ in push forwards.} $\mu_{\xi}=(\pi_{\xi})\mu$. The considerations above imply in particular that if $\xi$ is countably generated, then $\displaystyle{\eta=\xi(x)\mapsto \mu^{\eta}:=\mu_x^{\xi}}$ defines a measurable function from $M_{\xi}$ to the set of Borel probabilities on $M$, and furthermore     
\[
A\in\BM\Rightarrow \mu(A)= \int_{M_{\xi}}\mu^{\eta}(A)\mu_{\xi}(\der\eta).
\]
This shows that $\mu$ is completely determined provided that we know $\mu_{\xi}$ and $\{\mu^{\eta}\}_{\eta\in M_{\xi}}$. 

Our interest however is constructing a particular $f$-invariant measure $\mu$ (the equilibrium state); for this we will argue backwards and construct first the transverse measure and the conditionals. Of course this presents several technical difficulties, mainly because we do not have the partition $\xi$, nor there is a well defined notion of what it means to be $\mu$-measurable. 

A way to circumvent this problem is using some (continuous) invariant structure of $f$ to define the partition. In the next part we present a general construction that allow us to build families of measures satisfying certain (quasi-)invariant properties.

The core of our construction in this part is inspired by the work of G. Margulis \cite{TesisMarg}, who considered (mixing) Anosov flows and the trivial potential $\varphi\equiv 0$.

\subsection{Families of quasi-invariant measures} 
\label{sub:families_of_quasi_invariant_measures}

Let $X$ be a locally compact metric space and $f:X\to X$ a bi-measurable map. From now on, by a measure on $X$ we mean a Radon measure, and denote the set of such measures by $\RM[X]$. The set of continuous functions with compact support on $X$ is denoted by $\CMc[X]$. The set $\CMc[X]$ is equipped with the uniform topology, while $\RM[X]$ is equipped with the vague one. 

\smallskip 
 
\noindent\textbf{Notation:} If $\phi$ is measurable we write $f\phi:=\phi\circ f$. If $\mu\in\RM[X]$ the measure $f\mu=f_{\ast}\mu$ is the unique measure $\nu$ such that 
\[
	\forall \phi\in\CMc[X],\quad \int f\phi \der\mu =\int \phi \der\nu 
\]

\smallskip 

With the above it is direct to verify that for $\mu,\nu\in\RM[X]$, $\mu=\rho\nu$ it holds $f\mu=f^{-1}\rho f\nu$. Recall that $\mu\in\RM[X]$ is quasi-invariant if $f^{-1}\mu \sim \mu$; in this case we can write $f^{-1}\mu=\mathtt{h}_{\mu} \mu$ where $\mathtt{h}_{\mu}=\frac{d f^{-1}\mu}{d \mu}$. Similarly, $f^{-n}\mu=\mathtt{h}_{\mu}^{(n)} \mu$. 

As it can be deduced directly from the uniqueness of the Radon-Nikodym derivative, the family $\{\mathtt{h}_{\mu}^{(n)}\}_{n\in \Z}$ is a multiplicative cocycle over $f$ with generator $\mathtt{h}_{\mu}$.  

\paragraph*{Multiplicative cocycles} A family $\{h^{(n)}:X\to \Real_{\geq 0}\}_{n\in\Z}$ is a multiplicative cocycle over $f$ if 
\[
	\forall n,m \in \Z, x\in M\Rightarrow h^{(n+m)}(x)=h^{(n)}(f^mx)\cdot h^{(m)}(x).
\]
Given $h:X\to \Real_{>0}$, it generates a multiplicative cocycle over $f$ (and in this case we call $h$ the generator of the cocycle) by defining
\[
h^{(n)}(x)=\begin{dcases}
\prod_{k=0}^{n-1}h(f^kx)& k>0\\
1 & k=0\\
\prod_{k=n}^{-1} h^{-1}(f^kx) &k<0.
\end{dcases}
\]
Observe that if $\{h^{(n)}\}_{n\in\Z}$ is a positive multiplicative cocycle over $f$ (with generator $h$) then $\{a^{(n)}:=\log h^{(n)}\}_{n\in\Z}$ is an additive cocycle (with generator $\log h$):
\[
	\forall n,m \in \Z, x\in M\Rightarrow a^{(n+m)}(x)=a^{(n)}(f^mx)+a^{(m)}(x).	
\] 	

\smallskip 
 
 It will be useful to introduce the following.

 \begin{definition}
 Let $M,\Lambda$ be sets. We say that $M$ is a bundle over $\Lambda$ if can be written as 
 \[
 M=\bigsqcup_{i\in\Lambda} M_i
 \]
 where $M_i$ is a locally compact metric space, for every $i\in \Lambda$. We say that $f:M\to M$ is a bundle isomorphism if there exists a bijection $q:\Lambda\to \Lambda$ such that $f|:M_i\to M_{q(i)}$ is an homeomorphism, for every $i\in \Lambda$.
 \end{definition}

\begin{remark}
In the definition above we are not assuming, a priori, any topology on $M$, and therefore we do not require any global continuity for bundle morphisms. In the cases of interest $M$ will be manifold and each $M_i$ an immersed (but not necessarily embedded) sub-manifold.
\end{remark}

For example, if $M$ is a (Riemannian) manifold and $\F=\{W(x)\}_{x\in M}$ a foliation on $M$ with smooth leaves, then the metric on $M$ induces corresponding metrics, which in turn define a natural topology on each leaf, and we can write $M=\bigsqcup_{L\in \Lambda} L$, where $\Lambda=M/\F$; thus $M$ is a bundle over $\Lambda$. If furthermore $f:M\to M$ is a diffeomorphism and $\F$ is $f$-invariant, then $f$ induces a bijection $q:\Lambda\to \Lambda, q(L)=f(L)$, and thus $f$is a bundle equivalence. 

\smallskip 

Consider the bundles over $\Lambda$ given as
\begin{align*}
&\mathrm{Con}:=\bigsqcup_{i\in\Lambda} \CMc[M_i]\\
&\mathrm{Rad}:=\bigsqcup_{i\in\Lambda} \RM[M_i]
\end{align*}
and consider also the space of sections of $\mathrm{Rad}$, 
\[
\mathrm{Meas}:=\{\nu:\Lambda\rightarrow \mathrm{Rad}:\nu_i:=\nu(i)\in \RM[M_i]\forall\ i\in\Lambda\}.
\]
The natural topology of $\mathrm{Meas}$ is the weak topology determined by $\mathrm{Con}$; it is a locally convex topology. Let $\mathrm{Con}^+:=\{\phi\in\mathrm{Con},\phi\geq 0, \phi\not\equiv 0\}$.

\begin{lemma}\label{lem:Rproperty}
	Suppose that $\mathcal{A}\subset \mathrm{Meas}$ satisfies 
	\begin{enumerate}
		\item for every $\psi\in \mathrm{Con}$ there exists $c(\psi)>0$ with the property that for every $\mu\in\mathcal{A}, \mu(\psi)\leq c(\psi)$.
		\item For every $\psi \in \mathrm{Con}^+$ there exists $d(\psi)>0$ with the property that for every $\mu\in\mathcal{A}, \mu(\psi)\geq d(\psi)$.
	\end{enumerate}
	Then $\clo{\mathcal{A}}\subset \mathrm{Meas}$ is compact, and furthermore it does not contain the zero section.
\end{lemma}

\begin{proof}
	This is a direct consequence of Tychonoff's theorem.
\end{proof}

From now we assume that $f$ is a bundle equivalence.

\begin{definition}
$\nu\in\mathrm{Meas}$ is quasi-invariant if for every $i\in\Lambda$, $f^{-1}\nu_{qi}\sim \nu_{i}$ with positive continuous Radon-Nikodym derivative $\rho_i=\frac{\der f^{-1}\nu_{qi}}{\der\nu_i}$. In this case $\rho:M\to \Real_{>0}$ given by $\rho|M_i=\rho_i$ will be called the Jacobian of $\nu$. 
\end{definition}

One verifies directly that if $\nu$ is quasi-invariant then $\rho$ generates a (multiplicative) cocycle over $q$, and we write $\tilde{\rho}=\log \rho$ for the corresponding additive cocycle. \label{construcciondeseccion}

Now suppose that $\varphi:M\to \Real$ is a function such that $\varphi|M_i$ is continuous for each $i$,
and $\nu\in \mathrm{Meas}$ is quasi-invariant: we are interested in finding $\mu\in\mathrm{Meas}$ quasi-invariant with Jacobian $e^{-\varphi}$. Consider $\nu^n=e^{\SB}\cdot f^{-n}\nu $, i.e.
\begin{equation}\label{eq:nun}
	\nu^n_i=e^{\SB}\cdot f^{-n}\nu_{q^n(i)}=e^{\SB+S_n\tilde{\rho}}\cdot\nu_i,
\end{equation}

and let $\mathcal{C}^{+}$ be the positive cone generated by $\{\nu^n\}_{n\geq 0}$,
\begin{equation}
	\mathcal{C}^{+}=\Set{\sum_{j=1}^ka_i\nu^{n_j}:a_j\geq 0,k\in\mathbb{N}}.
\end{equation}

\begin{theorem}\label{thm:existenciafamiliamedidas}
Suppose that $M=\bigsqcup_{i\in\Lambda} M_i$ is a bundle over $\Lambda$, $f:M\to M$ is a bundle equivalence, $\varphi:M\to \Real$ is a function such that its restriction to each $M_i$ is continuous, and $\nu\in\mathrm{Meas}$ is quasi-invariant. Assume that $\{\nu^n=e^{\SB}\cdot f^{-n}\nu\}_{n\geq 0}$ satisfies:
\begin{itemize}
	\item if $\phi\in\mathrm{Con}^+$, then $\nu^n(\phi)>0$ for all $n$.
	\item Given $\phi, \psi \in \mathrm{Con}$ with $\phi\not\equiv 0$ non-negative, there exists $e(\phi,\psi)>0$ such that for every $n\geq 0$ it holds
	\[
	\frac{\nu^n(\psi)}{\nu^n(\phi)}\leq e(\phi,\psi).
	\]
\end{itemize}
Then there exist $P\in\Real$ and a quasi-invariant $\mu \in\cl{\mathcal{C}^{+}}$ with Jacobian $e^{P-\varphi}$. Moreover, $\mu$ has full support on each $M_i$.
\end{theorem}

\begin{proof}
Define $S:\mathrm{Meas}\rightarrow \mathrm{Meas}$ by 
\[
	S(\mu)_i=e^{\varphi}\cdot f^{-1}\mu_{q(i)}.
\]
Then $S$ is continuous and sends $\mathcal{C}^{+}$ to itself. We compute 
\begin{equation}\label{eq:Snun}
S(\nu^n)_i=e^{\varphi}\cdot f^{-1}\nu_{q(i)}^n= e^{\varphi}\cdot f^{-1}\left(e^{\SB}f^{-n}\nu_{q^{n+1}(i)}\right)=e^{\varphi+\SB\circ f}\cdot f^{-n-1}\nu_{q^{n+1}(i)}=\nu^{n+1}_i. 
\end{equation}
Fix a non-negative, non-identically zero function $\phi_0\in\mathrm{Con}$ and let $\hat{\nu}^n:=\frac{\nu^n}{\nu^n(\phi_0)}$; consider the set 
\begin{align}\label{eq:combinacionesconvexas}
\mathcal{X}&=\clo{\Set{\sum_{i=1}^ka_i\hat{\nu}^{n_i}:a_i\geq 0,k\in\mathbb{N},\sum_{i=1}^ka_i=1}}\\
		   &=\clo{\Set{\nu\in\mathcal{C}^+:\nu(\phi_0)=1}}.
\end{align}
By \Cref{lem:Rproperty} and the assumed hypotheses, $\mathcal{X}$ is a compact convex subset of $\mathrm{Meas}$, and by \Cref{eq:Snun} the normalized operator \(\tilde{S}\nu:=\frac{S\nu}{S\nu(\phi_0)}\) sends $\mathcal{X}$ to itself. Hence, by the Schauder-Tychonoff fix point theorem, there exists $\mu\in\mathcal{X}$ such that $\widetilde{S}(\mu)=\mu$. In other words, $\psi\in \mathrm{Con}$ implies
\begin{equation}\label{eq.margulis}
\mu_{q(i)}(\psi\circ f^{-1}e^{\varphi\circ f^{-1}})=e^P \cdot \mu_{i}(\psi)\quad (\Rightarrow f^{-1}\mu_{q(i)}=e^{P-\varphi}\cdot \mu_i),
\end{equation}
where $e^P=S\mu(\psi_0)>0$.

Observe that if $\psi\in \mathrm{Con}$ is non-negative and non-identically zero, then for some constant $c(\psi)>0$ it holds $\hat{\nu}^n(\psi)>0$ for all $n\geq 0$, and thus the same is valid for $\mu$. Therefore $\mu$ has full support on each $M_i$, and the proof of the Theorem is complete.
\end{proof}


\subsection{Quasi-invariant measures along leaves of foliations: the transverse measure} 
\label{sub:quasi_invariant_measures_along_leaves_of_foliations}

We will now apply the construction of the previous part to the following setting.

\begin{enumerate}
	\item[H-1] $M$ is a compact (closed) smooth Riemannian manifold and $\F=\{F(x)\}_{x\in M}$ is continuous foliations with smooth leaves. We consider the bundle structure of $M=\bigsqcup_{F\in \Lambda} F$ over $\Lambda=M/\F$, and denote $\mathrm{Con}(\F),\mathrm{Meas}(\F)$ the corresponding associated bundles.
	\item[H-2] $f$ is a diffeomorphism of $M$ of differentiability class $\mathcal{C}^2$.
	\item[H-3] $\F$ is $f$-invariant; the induced map by $f$ on $\Lambda$ is denoted by $q$.
	\item[H-4]  There exists a complementary continuous foliation $\mathcal{G}=\{G(x)\}_{x\in M}$ of $\F$ with smooth leaves, and such that
	\begin{enumerate}
		\item $\mathcal{G}$ is minimal: each leaf $G(x)$ is dense.
		\item $\mathcal{G}$ is $f$-invariant.
		\item $\mathcal{G}$ is contracting under $f$: if $d_{\mathcal{G}}$ denotes the intrinsic distance (on the corresponding leaf of $\mathcal{G}$) we have
		\[
		\forall x,y\in M, y\in G(x)\Rightarrow d_{\mathcal{G}}(fx,fy)\leq \lambda d_{\mathcal{G}}(x,y).
		\]
		for some $0<\lam<1$. 
	\end{enumerate}
\end{enumerate}

The idea is that if $B$ is a foliated box corresponding to $\F$, we will use the (measurable) partition $\F|B$ to play the role of the quotient space where we define our transverse measure. The existence of a complementary contracting foliation is used (crucially) to ensure the compactness property necessary for \Cref{lem:Rproperty}.

\smallskip

\noindent\textbf{Notation:} If $A\subset M, \ep>0$ let
\[
	G(A,\ep):=\{y\in M:\exists x\in A/ d_{\mathcal{G}}(x,y)<\ep\},
\]
and if $A=\{x\}$ we write $G(x,\ep)=G(\{x\},\ep)$.

\begin{definition}Let $\delta>0$.
	\begin{enumerate}
		\item Two relatively compact sets $A_1 \subset F_{i_1} ,A_2 \subset F_{i_2}$ are $\delta$-equivalent if there exists an holonomy transport (Poincaré map) $\hol[\Gcal]:F_{i_1}\to F_{i_2}$ induced by $\Gcal$ satisfying:
		\begin{enumerate}
			\item $\hol[\Gcal]|A_1$ is an homeomorphism of $A_1$ onto $A_2$.
			\item For every $x\in A_1,\hol[\Gcal](x)\in G(x,\delta)$.
		\end{enumerate}
		
		\item Two functions $\psi_1,\psi_2\in \mathrm{Con}(\F)$ are $\delta$-equivalent if $\supp(\psi_1)$ is $\delta$-equivalent to $\supp(\psi_2)$ and for every $x\in \supp(\psi_1)$, $\psi_2(\hol[G](x))=\psi_1(x)$. Here $\supp$ denotes the support of the corresponding continuous function, while $\hol[G]$ is the holonomy map connecting $\supp(\psi_1)$ with $\supp(\psi_2)$.
	\end{enumerate}
\end{definition}

Observe that given $\psi \in \mathrm{Con}(\F),\delta>0$ there exists $\gamma=\gamma(\psi,\delta)>0$ such that if 
$x\in G(\supp(\psi),\gamma)$ then 
\begin{itemize}
	\item the transverse foliation defines an homeomorphism onto its image $\hol[\Gcal]|:\supp(\psi)\to F(x)$.
	\item $\psi_{F(x)}:=\psi\circ \hol[\Gcal]$ is continuous, of compact support, and $\delta$-equivalent to $\psi$.
\end{itemize}

We seek to find a quasi-invariant $\nu\in\mathrm{Meas}(\F)$ with regular Radon-Nikodym derivative. A natural class is obtained by fixing a Riemannian metric on $M$ and considering $\Leb\in \mathrm{Meas}(\F)$ given by
\[
	F\in\Lambda\Rightarrow \Leb_F=\text{ induced Lebesgue measure on the leaf }F.
\]
By the basic change of variables theorem, $\Leb$ is quasi-invariant its Jacobian is $\rho_F=|\det Df|F|$. Due to compactness of $M$ we have that
\[
	 \rho\text{ is uniformly bounded from below: }\rho(x)\geq C_{\rho}>0\ \forall x\in M, 
\]
and since $f\in \diffM{M}{1+\alpha}$ we get that $\rho$ (and therefore, $\tilde{\rho}=\log\rho$) is $\alpha$-H\"older. 

\begin{remark}
In the abstract definition of Jacobian given before we did not consider the variation of this function for points in different fibers of $M$. Note however that in the hypotheses of this part, the globally defined map $\rho:M\to\Real_{>0}$ is continuous. If furthermore $T\F$ is Hölder (something common in applications), then $\rho$ is globally Hölder as well. 
\end{remark}

We will use the minimality of $\Gcal$ to compare sets (functions) defined on different leaves of $\F$. It is necessary however to control how families of measures $\nu=\{\nu_F\}_{F\in\F}$ vary with the transverse holonomy given by $\Gcal$. By transversality between $\F$ and $\Gcal$ and compactness of $M$ it holds that given $\delta_0>0$ there exists $\delta_1>0$ so that for every $x,y\in M, y\in G(x,\frac{\delta_0}{2})$,
\[
	F(y,\delta_1)\text{ is }\delta_0\text{-equivalent to some subset of } F(x)
\]

\begin{definition}\label{def:strongabsolutecontinuous}
$\nu\in\mathrm{Meas}(\F)$ is absolutely continuous if for every (locally defined) holonomy transport $\hol[\Gcal]:F_{i_1}\to F_{i_2}$, $\hol[\Gcal]\nu_{F_{i_1}}\sim \nu_{F_{i_2}}$. It is strongly absolutely continuous if there exist $\delta_0>0$ and $J:\{(x,y,z)\in M\times M\times M, y\in G(x,\delta_0), z\in F(y,\delta_1)\}\to \Real$ continuous such that if $B=F(y,\delta_1), A=\hol[\Gcal](B) \subset F(x)$ then
\[
   \hol[\Gcal]\nu_{F(x)}|A=J(x,y,\cdot)\nu_{F(y)}|B	
\] 
\end{definition}

\begin{proposition}\label{pro:lebesgueesabscont}
$\Leb$ is strongly absolutely continuous.
\end{proposition}
\begin{proof}
This result is classical by now, and relies on the contracting behavior of $\Gcal$ together with the fact of $f$ being of class $\mathcal{C}^{1+\alpha}, \alpha>0$. See \cite{ErgAnAc}.
\end{proof}

For strongly absolutely continuous sections we have the following.

\begin{proposition}\label{pro:MargulisAbsCont}
Assume that $\nu$ is strongly absolutely continuous and of full support on every $F\in\F$. Then given $\ep>0$ there exists $\delta>0$ such that for every $\psi_1,\psi_2\in \mathrm{Con}(\F)$ that are $\delta$-equivalent and non-identically zero, it holds
\[
\left|\int \psi_1 \der\nu_{F_{i_1}}-\int \psi_2 \der\nu_{F_{i_2}}\right|<\delta \int \abs{\psi_1}  \der\nu_{F_{i_1}}.	
\]
\end{proposition}

\begin{proof}
We start by noting that since $\nu$ is strongly absolutely continuous,  given $\ep>0$ there exists $\delta$ such that if $A,B$ are $\delta$-equivalent (with non-trivial measure) then
\begin{equation}\label{eq:derivadaRNnu}
\Big|\frac{\nu(A)}{\nu(B)}-1\Big|<\ep.
\end{equation}

The remaining part of the argument is spelled in \cite{TesisMarg}; for convenience of the reader we repeat it here. 

\noindent\textbf{Claim:} Given $\ep>0$ there exists $\delta>0$ such that for every $\psi_1,\psi_2\in \mathrm{Con}(\F)$ that are $\delta$-equivalent, non-identically zero and non-negative, it holds

\[
\left|1-\frac{\int \psi_2 \der\nu_{F_{i_2}}}{{\int \psi_1 \der\nu_{F_{i_1}}}}\right|<\delta.	
\]

\smallskip

\begin{proofw}
Given $\ep>0$ let $\delta>0$ as above, and consider $\psi_1, \psi_2$ under the hypotheses of the claim. Let $A \subset F_{i_1}, B \subset F_{i_2}$  be the supports of $\psi_1, \psi_2$. Then by the Radon-Nikodym theorem we can write
\[
	\int \psi_1 d\nu_{F_{i_1}}=\int \psi_2\cdot j \der\nu_{F_{i_2}},
\]
where, by \eqref{eq:derivadaRNnu}, $|1-j(z)|<\delta$ for every $z\in B$. Using that $\psi_1, \psi_2$ are non-negative, we deduce
\begin{align*}
(1-\delta)\cdot \int \psi_2 \der\nu_{F_{i_2}}\leq \int \psi_1 \der\nu_{F_{i_1}} \leq (1+\delta)\cdot \int \psi_2 \der\nu_{F_{i_2}}
\end{align*}
which implies the claim.
\end{proofw}

\smallskip 

Given $\psi_1,\psi_2\in \mathrm{Con}(\F)$ that are $\delta$-equivalent, non-identically zero, we decompose into their positive and negative parts, $\psi_1=\psi_1^{+}-\psi_1^{-}, \psi_1^{-}, \psi_2=\psi_2^{+}-\psi_2^{-}$ and use the previous claim:

\begin{align*}
\left|\int \psi_1 \der\nu_{F_{i_1}}-\int \psi_2 \der\nu_{F_{i_2}}\right|&\leq \left|\int \psi_1^+ \der\nu_{F_{i_1}}-\int \psi_2^+ \der\nu_{F_{i_2}}\right|+ \left|\int \psi_1^- \der\nu_{F_{i_1}}-\int \psi_2^- \der\nu_{F_{i_2}}\right|\\
&\leq\delta \int \psi_1^+ \der\nu_{F_{i_1}}+\delta \int \psi_1^- \der\nu_{F_{i_1}}=\delta \int |\psi_1| \der\nu_{F_{i_1}}.
\end{align*}
The proof is complete.
\end{proof}

In what follows we will be interested in quasi-invariant sections $\nu\in \mathrm{Meas}(\F)$. We introduce the following definition.

\begin{definition}
We say that $\nu\in \mathrm{Meas}(\F)$ is appropriate for $f$ if satisfies:
\begin{enumerate}
	\item is strongly absolutely continuous,
	\item for every leaf $F\in \F$, $\nu_F$ has full support,
	\item is quasi-invariant with H\"older Jacobian. 
\end{enumerate}
\end{definition}

\begin{remark}\label{rem:lebesgueispp}
It follows by \Cref{pro:lebesgueesabscont} that $\Leb$ is appropriate.
\end{remark}

Fix $\nu\in \mathrm{Meas}(\F)$ appropriate for $f$ and consider $\varphi:M\to\Real$ a H\"older continuous potential. Let $\rho$ be the Jacobian of $\nu$, $\hat{\rho}=\log \rho$ and $\hat{\varphi}=\varphi+\hat{\rho}$; this is a $(C_{\hat{\varphi}},\theta)$-H\"older function. For $n\in \Nat$ take $\nu^{n}=e^{\SB}\cdot f^{-n}\nu=e^{S_n\hat{\varphi}}\nu$ (cf. \Cref{eq:nun}).

\begin{lemma}\label{lem:controlnun}
	There exists a constant $D_1>0$ and $\ell:\Real_{>0}\to\Real_{>0}$ 	with the following property. For every $\delta>0$, if $\psi_1,\psi_2\in \mathrm{Con}^+$ are $\delta$-equivalent and $n\geq0$, then it holds
	\[
	\nu^n(\psi_1)\leq \ell(\delta)\cdot e^{D_1\delta^{\theta}}\nu^n(\psi_2). 
	\]
	Furthermore, $\ell(\delta)\xrightarrow[\delta\to 0]{}1$.
\end{lemma}

\begin{proof}
	Define $D_1:=C_{\hat{\varphi}}\sum_{i=0}^{\infty} \lambda^{i\theta}$; since $\lambda<1$, $D<+\oo$. Assume that $\psi_1,\psi_2$ are $\delta$-equivalent, $\psi_2\circ \hol[\Gcal] =\psi_1$. Observe that for every $x\in \supp(\psi_1)$, for every $n\geq 0$, it holds
\[
\abs{\SB[\hat{\varphi}](\hol[\Gcal] x)-\SB[\hat{\varphi}](x)}\leq \sum_{i=0}^{\oo} C_{\hat{\varphi}}\cdot d_{\Gcal}(f^i\hol[\Gcal] x,f^ix)^{\theta}\leq D_1\delta^{\theta}.
\]
It then follows that
\begin{align*}
\nu^n(\psi_1)&=\int e^{S_n\hat{\varphi}} \psi_2\circ \hol[\Gcal] \der\nu=\int e^{S_n\hat{\varphi}-S_n\hat{\varphi}\circ \hol[\Gcal]} \psi_2\circ \hol[\Gcal] e^{S_n\hat{\varphi}\circ \hol[\Gcal]} \der\nu\\
&\leq e^{D_1\delta^{\theta}} \int \psi_2 e^{S_n\hat{\varphi}} \der\hol[\Gcal]\nu\leq \ell(\delta)e^{D_1\delta^{\theta}} \nu^n(\psi_2)
\end{align*}
where $\ell$ is an upper bound of $\frac{d \hol[\Gcal]\nu}{d\nu}$, and therefore converges to $1$ as $\delta\to 0$.
\end{proof}

\begin{lemma}
	Let $A_1\subset F_{i_1},A_2\subset F_{i_2}$ be relatively open and pre-compact. Then there exist $\hat{e}(A_1,A_2)>0$ such that for every $n\geq 0$ it holds
	\[
	\frac{1}{\hat{e}(A_1,A_2)}\leq \frac{\nu^n(A_1)}{\nu^n(A_2)}\leq \hat{e}(A_1,A_2).
	\]
\end{lemma}

\begin{proof}
Let us start noting that since $\Gcal$ is minimal (on $M$ compact) we have the following property: for any $A\subset F, F\in \F$ there exist $\delta(A),r(A)>0$ such that for every $x\in M$ one can find $B_x\subset A$ that is $\delta(A)$-equivalent to $F(x,r(A))$.

By relative compactness of $\clo{A_1}$, the closure of $A_1$ inside $F_{i_1}$, it follows that $\clo{A_1}\subset \cup_1^m F(x_j,r(A_2))$, where each plaque $F(x_j,r(A_2))$ is $\delta(A_2)$-equivalent to some $B_j\subset A_2$.

Fix two of these sets $E_1=F(x_j,r(A_2)), E_2=B_j$ and denote by $\hol[\Gcal]$ the holonomy transport such that $\one_{E_1}=\one_{E_2}\circ \hol[\Gcal]$. By regularity of $\nu_{F_{i_2}}$ we can find a non-decreasing sequence $(\phi_k)_k\subset \CMc[F_{i_2}]$ satisfying
\begin{enumerate}
	\item $k\in\mathbb{N}\Rightarrow 0\leq \phi_k\leq 1, \supp(\phi_k)\subset \clo{A_2}$.
	\item $\lim_{k\mapsto\oo}\phi_k(x)=\one_{A_2}(x)$. 
	\item $\nu^n(A)=\lim_k\uparrow\nu^n(\phi_k)$. 
\end{enumerate}
Define $\psi_k=\phi_k\circ \hol[\Gcal]$ and apply the previous lemma to deduce that for every $n$,
\begin{align*}
\frac{\nu^n(E_1)}{\nu^n(E_2)}=\frac{\lim_k \nu^n(\psi_k)}{\sup_k \nu^n(\phi_k)}\leq \sup_k \frac{\psi_k}{\phi_k}\leq \ell(\delta(A_2))e^{D_1\delta(A_2)^{\theta}}:=e_0. 
\end{align*}
Therefore,
\begin{align*}
\frac{\nu^n(A_1)}{\nu^n(A_2)}\leq \frac{\nu^n(F(x_j,r(A_2))}{\nu^n(A_2)}
\leq m\cdot\max_{j} \frac{\nu^n(F(x_j,r(A_2)))}{\nu^n(B_j)}\leq m\cdot e_0.
\end{align*}
From here the conclusion follows.
\end{proof}

\begin{corollary}\label{cor:nucompacidad}
Fix $\psi\in \mathrm{Con}^+(\F)$. Then for every $\phi\in \mathrm{Con}(\F)$ there exists $\hat{e}(\phi,\psi)>0$ such that for every $n\geq 0$ it holds
\[
\frac{\nu^n(\phi)}{\nu^n(\psi)}\leq \hat{e}(\phi,\varphi).
\]
\end{corollary}

\begin{proof}
This is essentially Lemma $2.4$ of \cite{TesisMarg}. For $r>0$ let $A_r:=\psi^{-1}(r,\oo)$ and note that since $\psi\in \mathrm{Con}$ is non-negative and non-identically zero, there exists $r>0$ such that $A_r$ is relatively open and pre-compact inside the leaf of $\F$ that contains $\supp(\psi)$. Take $A$ open, relatively compact containing $\supp(\phi)$ and use the previous Lemma to deduce
\[
\frac{\nu^n(\phi)}{\nu^n(\psi)}<\frac{\norm{\phi}_{\oo}\cdot\nu^n(A)}{r\cdot \nu^n(A_r)}<\frac{\norm{\phi}_{\oo}}{r}\ \hat{e}(A,A_r). 
\]
\end{proof}

We have thus shown that the in the assumed conditions $H-1$ to $H-4$, there exist an appropriate an $\nu$ (\Cref{rem:lebesgueispp}), and therefore by \Cref{cor:nucompacidad} we get that the hypotheses of \Cref{thm:existenciafamiliamedidas} are satisfied. Thus, we conclude the existence of $P\in\Real$ and $\mu\in \mathrm{Meas}(\F)$ such that

\begin{enumerate}
	\item $\mu$ is in the closure of the positive cone generated by $\{\nu^n\}_{n\geq 0}$ (and thus has full support on each leaf of $\F$),
	\item $\mu$ is quasi-invariant, with Jacobian $e^{P-\varphi}$.
\end{enumerate}


\paragraph{Strong absolute continuity of the quasi-invariant section} 
\label{par:strong_absolute_continuity_of_the_quasi_invariant_section} It turns out that the quasi-invariance of $\mu$ implies that is strongly absolutely continuous. We keep working with $\nu, \mu$ as in previous part, that is
\begin{itemize}
 	\item $\nu$ is appropriate for $f$,
 	\item $\mu$ is quasi-invariant with Jacobian $e^{P-\varphi}$, and in the closure of the positive cone generated by $\{\nu^n\}_{n\geq 0}$.
 \end{itemize} 

Denote for $x\in M$, $\mu_x=\mu_{F(x)}$. By quasi-invariance we get
\begin{equation}
\forall n\in\Nat,\quad f^{-n}\mu_{f^nx}=e^{nP-\SB}\mu_x.
\end{equation}

\begin{lemma}
Given $\ep>0$ there exists $\delta>0$ such that for every $\psi_1,\psi_2\in\mathrm{Con}^+$ that are $\delta$-equivalent, then
\[
	\left|\frac{\mu(\psi_1)}{\mu(\psi_2)}-1\right|<\ep.
\]
\end{lemma}

\begin{proof}
By \Cref{lem:controlnun}, for every $n\in\Nat$ we have $\hat{\nu}^n(\psi_1)\leq \ell(\delta)\cdot e^{D_1\delta^{\theta}}\hat{\nu}^n(\psi_2)$; the inequality extends to convex combinations of $\{\hat{\nu}^n\}_n$ and therefore, to elements of $\mathcal{X}$ (cf.\@ \Cref{eq:combinacionesconvexas}). In particular $\mu$ satisfies the claim.

\end{proof}

We now fix $x_0, y_0$  in the same leaf of $\Gcal$ and consider $\hol[\Gcal]:A(x_0)\subset F(x_0)\to B(y_0)\subset F(y_0)$ the corresponding Poincaré map. For $x\in A(x_0)$ define
\begin{equation}\label{eq:Jacobianohol}
\Jac(x)=\prod_{j=0}^{\oo}\frac{e^{\varphi\circ f^j(\hol[\Gcal]x)}}{e^{\varphi\circ f^j(x)}}=\lim_n e^{\SB(\hol[\Gcal]x)-\SB(x)}.
\end{equation}
Since $\varphi$ is H\"older and $\Gcal$ is contracting, the previous formula defines a continuous function $\Jac:A(x_0)\rightarrow \Real$.

\begin{proposition}\label{pro:Jacobinanohol}
It holds $(\hol[\Gcal])^{-1}\mu_{y_0}=\Jac \cdot \mu_{x_0}$.
\end{proposition}

\begin{proof}
Denote $h=\hol[\Gcal]$, and for $n\geq 0$ let $h_n$ be the Poincaré map defined on a neighborhood of $f^nx$ inside $f^n(A(x_0))$. As $f$ preserves $\F,\Gcal$, $h=f^{-n}\circ h_n\circ f^{n}$ and therefore

\begin{align*}
\MoveEqLeft h^{-1}\mu_{y_0}=(f^{-n}\circ h_n^{-1})e^{\SB\circ f^{-n}-nP}\mu_{f^ny_0}=f^{-n}\left(e^{\SB\circ h\circ f^{-n}-nP}h_n^{-1}\mu_{f^ny_0}\right)\\
&=f^{-n}\left(e^{\SB\circ h\circ f^{-n}-nP}\left(\mu_{f^nx_0}+\left(h_n^{-1}\mu_{f^ny_0}-\mu_{f^nx_0}\right)\right)\right)\\
&=e^{\SB\circ h-\SB}\mu_{x_0}+f^{-n}e^{\SB\circ h\circ f^{-n}-nP}\left(h_n^{-1}\mu_{f^ny_0}-\mu_{f^nx_0}\right).
\end{align*}
By the lemma above, given $\ep>0$ there exists $n_0$ such that for $n\geq n_0$ it holds: for any $\psi\in \mathrm{Con}^+$ supported in $f^n(A(x_0))$,
\[
	\abs{h_n^{-1}\mu_{f^ny_0}(\psi)-\mu_{f^nx_0}(\psi)}<\ep h_n^{-1}\mu_{f^ny_0}(\psi)
\]
which in turn implies
\begin{align*}
\abs{f^{-n}e^{\SB\circ h\circ f^{-n}-nP}\left(h_n^{-1}\mu_{f^ny_0}-\mu_{f^nx_0}\right)(\psi\circ f^n)}&<\ep \mu_{f^ny_0}(\psi\circ h_n^{-1}e^{\SB\circ f^{-n}-nP})\\
    &=\ep h^{-1}\mu_{y_0}(\psi\circ f^n).
\end{align*}
Thus we can write $h^{-1}\mu_{y_0}=J_n\cdot \mu_{x_0}+\upsilon_n$, where $\upsilon_n$ converges to zero in the vague topology, and $\{J_n\}_n$ is a sequence of continuous functions that converges uniformly to $\Jac$. This concludes the proof.
\end{proof}

\begin{remark}\label{rem:jacobianoconvergeauno}
Using that $\Gcal$ is contracting under $f$ and $\varphi$ is Hölder, one verifies directly that 
$\normC{\Jac-1}{0}\xrightarrow[y_0\mapsto x_0]{}0$.
\end{remark}


\section{Invariant measures for center isometries} 
\label{sec:measures_along_invariant_foliations_of_center_isometries}

We will use the previous construction in the case when $f$ is a center isometry. For convenience of the reader we recall the definition and some basic facts below.

\begin{definition}\label{def:centerisometry}
	Let  $M$\ be a closed manifold. A  diffeomorphism  $f\in \diffM{M}{r}, r\geq 1$ is a center isometry if there exist a continuous splitting of the tangent bundle of the form 
	\[TM=E^u\oplus E^c\oplus E^s
	\]
	where both bundles $E^s,E^u$\ are non-trivial, a (continuous) Riemannian metric $\norm{\cdot}$ on $M$ and 
	a constant $0<\lambda<1$ such that
	\begin{enumerate}
		\item All bundles $E^u,E^s,E^c$\ are $Df$-invariant.
		\item For every $x\in M$, for every unit vector $v^{\ast}\in E^{\ast}_x, \ast=s,c,u$, 
		\begin{gather*}
		\norm{D_xf(v^{c})}=1\\
		\norm{D_xf^n(v^{s})},\norm{D_xf^{-1}(v^{u})}<\lambda.
		\end{gather*}
	\end{enumerate}
\end{definition}

Center isometries are a subclass of the so called \emph{Partially Hyperbolic Systems}. From their theory we need the following.

\begin{theorem}\label{thm:centerisometry}
	If $f\in \diffM{M}{r}$\ is a center isometry then the bundles $E^{s}, E^u, E^c$, $E^{cs}=E^c\oplus E^s,E^{cu}=E^c\oplus E^u$ are (uniquely) integrable to continuous foliations $\Fs=\{W^s(m)\}_{x\in M}$, $\Fu=\{W^u(m)\}_{x\in M},\Fc=\{W^c(m)\}_{x\in M},\Fcs=\{W^{cs}(m)\}_{x\in M},\Fcu=\{W^{cu}(m)\}_{x\in M}$\ respectively called the \emph{stable},\emph{unstable},\emph{center}, \emph{center stable} and \emph{center unstable} foliations, whose leaves are $\mathcal{C}^r$\ immersed submanifolds. Moreover, leaves of $\Fcs$ are saturated by leaves of $\Fs,\Fc$, and leaves of $\Fcu$ are saturated by leaves of $\Fu,\Fc$.
\end{theorem}
\begin{proof}
	Integrability of the bundles $E^s,E^u$ is consequence of the classical Stable Manifold theorem (see for example Theorem 4.1 in \cite{PesinLect}), while the rest of the assertions can be deduced from theorem 7.5 in \cite{PartSurv}. 
\end{proof}

\noindent\textbf{Standing hypotheses for the rest of the section:} $f:M\rightarrow M$ is center isometry of class $\mathcal{C}^2$ and $\varphi:M\to\Real$ is a H\"older potential. The metric in $M$ is assumed to make $E^s,E^c,E^u$ mutually perpendicular.\footnote{This can always be achieved by passing to an equivalent metric.} Both foliations $\Fs,\Fu$ are minimal.

\smallskip

We will apply the results of \Cref{sec:measures_along_foliations}: take $\F=\Fcu, \Gcal=\Fs$, $\nu=\Leb$, the Lebesgue measure on leaves of $\Fcu$. By the $\mathcal C^2$ hypothesis on $f$ and \Cref{pro:lebesgueesabscont}, $\nu$ is appropriate for $f$, and since $\Fs$ is minimal we are in the setting discussed in \Cref{sub:quasi_invariant_measures_along_leaves_of_foliations}, therefore the conclusion of \Cref{cor:nucompacidad} is valid. Then \Cref{thm:existenciafamiliamedidas} allows us to deduce the existence of an appropriate $\mu^{cu}\in \mathrm{Meas}(\Fcu)$ for $f$, with Jacobian $e^{P-\varphi}$ for some $P\in\Real$. Finally by \Cref{pro:Jacobinanohol}, $\mu$ is strongly absolutely continuous (therefore appropriate for $f$). We summarize this in the Proposition below.

\begin{proposition}\label{pro:medidacu}
	There exists a section $\mu^{cu}=\mu^{cu}_{\varphi}\in \mathrm{Meas}(\Fcu)$ and $P\in \Real$ such that:
	\begin{enumerate}
		\item For every $\psi\in \mathrm{Con}^+(\Fcu,)$ it holds  $\mu^{cu}(\psi)>0$.
		\item For every $x\in M$, $f^{-1}\mcux[fx]=e^{P-\varphi}\mcux[x]$.
		\item If $\hs=\hs_{x_0,y_0}:A(x_0)\subset \Wcu{x_0}\rightarrow B(y_0)\subset\Wcu{y_0}$ is the Poincaré map determined by the stable holonomy that sends $x_0$ to $y_0$, then
		\[
		(\hs)^{-1}\mu^{cu}_{y_0}=\Jacs\cdot \mu^{cu}_{x_0}
		\]
		where
		\begin{equation}
		\Jacs(x)=\prod_{j=0}^{\oo}\frac{e^{\varphi\circ f^j(\hs x)}}{e^{\varphi\circ f^j(x)}}.
		\end{equation} 
    \end{enumerate} 
\end{proposition}

Note that $f^{-1}$ is also a center isometry whose stable foliation coincides with the unstable of $f$, and likewise its unstable foliation coincides with the stable of $f$. By applying the previous arguments to $f^{-1}$ and the potential $\varphi\circ f^{-1}$ we conclude the following.

\begin{proposition}\label{pro:medidacs}
There exist $\mu^{cs}=\mu^{cs}_{\varphi} \in \mathrm{Meas}(\Fcu)$ and $P'\in \Real$ such that:
\begin{enumerate}
		\item For every $\psi\in \mathrm{Con}^+(\Fcs)$ it holds  $\mu^{cs}(\psi)>0$.
		\item For every $x\in M$, $f^{-1}\mcsx[fx]=e^{\varphi-P'}\mcsx[x]$.
		\item If $\hu=\hu_{x_0,y_0}:A(x_0)\subset \Wcs{x_0}\rightarrow B(y_0)\subset\Wcs{y_0}$ is the Poincaré map determined by the unstable holonomy that sends $x_0$ to $y_0$, then
		\[
		(\hu)^{-1}\mu^{cs}_{y_0}=\Jacu\cdot \mu^{cs}_{x_0}
		\]
		where
		\begin{equation}
		\Jacu(x)=\prod_{j=1}^{\oo}\frac{e^{\varphi\circ f^{-j}(\hu x)}}{e^{\varphi\circ f^{-j}(x)}}.
		\end{equation} 
    \end{enumerate} 
\end{proposition}

\subsection{The conditional measures: potentials constant along the center and SRBs} 
\label{sub:the_conditional_measures_potentials_constant_along_the_center_and_srbs}

We are interested in constructing $\mu^u\in\mathrm{Meas}(\Fu)$ satisfying analogous properties as $\mu^{cu}$. Our previous construction however does not apply in this case since the transverse foliation to $\Fu$ is not contracting. To bypass this problem we assume some condition on $\varphi$, namely that it is either
\begin{itemize}
\item constant on center leaves ($c$-constant case), or
\item $\varphi=-\log\det Df|E^u$ (SRB case). 
\end{itemize}

In the SRB case we define 
\[
\mux=\text{Lebesgue measure on }\Wu{x};
\]  
by the change of variables theorem $f^{-1}\mux[fx]=\det Df|E^u\cdot \mux=e^{-\varphi}\mux[x]$. For the $c$-constant case, one deduces by compactness of $M$ the existence of some $\ccen>0$ such that for every $x\in M$ there is a well defined projection $\pi_x^c:\Wc{\Wu{x},\ccen}\to\Wu{x}$ by sliding along local center plaques. Given $x$ we define a measure on $\Wu{x}$ by setting
\begin{align}\label{eq.defmuxcconstant}
\mux=\pi^c_x\cdot\mcux
\end{align}
Note that $\mu^u\in \mathrm{Meas}(\Fu)$ is quasi-invariant with Jacobian $e^{P-\varphi}$; indeed by using that 
$\varphi$ is constant on center leaves and that $f$ is an isometry on these, we deduce
\begin{align}\label{eq:invarianciainest}
f^{-1}\mux[fx]=f^{-1}\pi^c_{fx}\mcux[fx]=\pi^c_{x}f^{-1}\mcux[fx]=\pi^c_{x}e^{P-\varphi}\mcux[x]=e^{P-\varphi}\pi^c_x\mcux[x]=e^{P-\varphi}\mux[x].
\end{align}
Arguing with $\mu^{cs}$ instead of $\mu^{cu}$ we can construct analogously a family of measures $\mu^s=\{\mu^{cs}_x\}_{x\in M}$ on stable leaves. An important case of $c$-constant potential is when $\varphi\equiv 0$; in this case we denote $\mathscr{m}^{\ast}\in \mathrm{Meas}(\F^{\ast}),\ \ast\in\{s,u,cs,cu\}$ the corresponding sections.

\begin{definition}
$\mathscr{m}^{\ast}$ are the Margulis measures on $\F^{\ast}$.
\end{definition}

It follows that
\begin{align}\label{eq:margulisinvariance}
&f^{-1}\mathscr{m}^{cu}_{fx}=e^{h}\mathscr{m}^{cu}_x\\
&f^{-1}\mathscr{m}^{u}_{fx}=e^{h}\mathscr{m}^{u}_x
\end{align} 
for some $h\in\Real$.

\begin{proposition}\label{pro:Margulisinvariantehol}\hfill
\begin{enumerate}
	\item $\mathscr{m}^{cu}$ ($\mathscr{m}^{cs}$) is invariant under $\hs$ holonomy (respectively, $\hu$ holonomy).
	\item The constant $h$ is positive.
\end{enumerate}
\end{proposition}
\begin{proof}
By part $3$ of \Cref{pro:medidacu} it holds $\Jacs\equiv1$, thus implying the first part. 

Recall the construction of $\mathscr{m}^{cu}$ on page \pageref{construcciondeseccion}, with $\nu=\Leb$ and $\nu^n=f^{-n}\nu$ and a fixed normalized $\psi_0\in\mathrm{Con}^+(\Fcu)$. Thus we have for every $n\geq 0$, 
	\begin{align*}
    f^{-1}\nu^n(\psi_0)\geq \min \{\det Df|E^{cu}\}\nu^{n}(\psi_0)\Rightarrow f^{-1}\hat{\nu}^n(\psi_0)\geq \min \{\det Df|E^{cu}\}. 
	\end{align*}
   	This implies the same inequality for every member of $\mathcal{X}$, and in particular 
	\[
	e^h=\mathscr{m}^{cu}(\psi_0\circ f^{-1})\geq \min \{\det Df|E^{cu}\}>1.
	\]
\end{proof}
\begin{remark}
The reader can perceive that in principle $h$ (and $P,P'$) depend on $\psi_0$. We will show in the next section that $h=h_{\mathrm{top}}(f)$ and $P=P'=P_{\mathrm{top}}(\varphi)$.
\end{remark}

We return to the general case; we will use the family $\{\mux\}_x$ to define some new measures $\{\nu^u_x\}_x$ that will serve as conditional measures on local unstable manifolds for the equilibrium state.

For each $x\in M, y\in\Wu{x}$ let 
\begin{equation}\label{eq:JacobianoDelta}
\Delta_x^u(y):=\prod_{k=1}^{\oo}\frac{e^{\varphi\circ f^{-k}(y)}}{e^{\varphi\circ f^{-k}(x)}},
\end{equation}
and note that $\Delta_x^u:\Wu{x}\to\Real$ is continuous. Define 
\begin{equation}\label{eq.nux}
\nu^u_x=\Delta_x^u\ \mux.
\end{equation}

One verifies easily that $\nux$ is a Radon measure on $\Wu{x}$ which is positive on open sets, and furthermore 
\begin{align}\label{eq:invariancianu}
f^{-1}\nux[fx]=\Delta_{fx}^u\circ f\cdot f^{-1}\mux[fx]=\Delta_{fx}^u\circ f \cdot e^{P-\varphi}\mux =e^{P-\varphi(x)}\nux.
\end{align}

For $y\in \Wu{x}$ we have $\nu^u_y=c(y,x)\cdot \nux$ with $c(y,x)>0$, hence $\{\nu^u_y\}_{y\in\Wu{x}}$ defines a projective class of measures $[\nu^u_x]$. In particular if $\xi$ is a $\mcux$-measurable partition subordinated to $\Fu$ we have that the probability measures

\begin{equation}\label{eq:condicionalxi}
\nu^{\xi}_x:=\nux(\ \cdot\ |\xi(x))
\end{equation}
satisfy $\nu^{\xi}_x=\nu^{\xi}_y$ for $y\in\xi(x)$.

Before moving on we establish the following lemma that will allow us to compare different unstable measures. Let $\ep>0$ be so that any set of diameter less than $2\ep$ is contained in a foliation chart of $\Fu$, and consider the compact space
\[
	N=\{(x_0,y_0,z)\in M^3: y_0\in\clo{\Wu{x_0,\ep/2}}, z\in \clo{\Wc{x_0,\ep/2}}\}
\]
and for $(x_0,y_0,z)\in N$ define
\begin{equation}\label{eq.Jacobianouconholonomia}
T_{x_0,y_0}(z)= \prod_{k=1}^{+\oo}\frac{e^{\varphi\circ f^{-k}(\hu_{x_0,y_0}(z))}}{e^{\varphi\circ f^{-k}(z)}}.
\end{equation}

\begin{lemma}\label{lem:JacobianoT}
It holds that:
\begin{enumerate}
\item the constant $c(y,x)$ converges uniformly to $1$ as $y\to x$, and
\item $T:N\to\Real$ is (uniformly) continuous.
\end{enumerate}
\end{lemma}

\begin{proof}
The first follows directly since $c(y,x)=\Delta_y^u(x)$. For the second, for ever $k\geq 0$ consider the continuous function $T_k:N\to\Real$ defined as
\[
T_k(x_0,y_0,z)= \sum_{i=1}^{k}\varphi\circ f^{-i}(\hu_{x_0,y_0}(z))-\varphi\circ f^{-i}(z)
\]
and note that by the Weierstrass' test, $T_k$ converges uniformly to 
\[
\sum_{k=1}^{+\oo}\varphi\circ f^{-k}(\hu_{x_0,y_0} z)-\varphi\circ f^{-k}(z)
\]
which implies that $\log T$ is continuous, and thus $T$ is continuous.
\end{proof}


\subsection{The equilibrium state}\label{sub:theequilibrium}

It is seldom the case that for a given invariant measure $\mu$ for $f$, the partition by unstable manifolds is $\mu$- measurable, therefore we will have to work in sets $B$ where $\Fu|B$ defines a sufficiently nice partition. Since we are trying to use the family $\{\mcsx\}_{x\in M}$ as a substitute for the transverse measure of the equilibrium state, we will have to deal with the technical difficulty of having to show coherence between different choices $B,B'$; this will be bypassed using our control on the behavior of $\{\mcsx\}_{x\in M}$ under unstable (local) holonomies. We will now make some geometrical preparations to deal with this problem.

\paragraph{Dynamical boxes}To take advantage of the local product structure between $\Fcs$ and $\Fu$, we introduce some new notation. Due to the fact that the invariant bundles are mutually perpendicular there exists $\clps>0$ such that for any $0<\epsilon\leq\clps$ it holds: for every $x,y\in M, d(x,y)<\ep$ implies
\begin{align}\label{eq:productstructure}
&\#\Ws{x,2\ep}\cap\Wcu{y,2\ep}=1\\
&\#\Wu{x,2\ep}\cap\Wcs{y,2\ep}=1
\end{align}

\begin{definition}\label{def:dynamicalboxes}
For $0<\epsilon\leq \clps$ and $x\in M$ we consider the sets
\begin{align}
\label{eq:Pcs}\Pcs{x,\epsilon}&:=\Wc{\Ws{x,\epsilon},\epsilon}\\
\label{eq:Pcu}\Pcu{x,\epsilon}&:=\Wc{\Wu{x,\epsilon},\epsilon}\\
\label{eq:dynbox}B(x,\epsilon)&:=\bigcup_{\crampedclap{y\in \Wu{x,\epsilon}}}\ \hu_{x,y}(\Pcs{x,\epsilon})
\end{align}
where $\hu_{x,y}:\Wcs{x}\rightarrow \Wcs{y}$ is the locally defined Poincaré map sending $x$ to $y$.

Sets $B(x,\epsilon)$ constructed in this fashion will be referred as \emph{dynamical boxes}, while $\hu_{x,y}(\Pcs{x,\epsilon})$ will be referred as the $cs$-plaques of $B(x,\ep)$.
\end{definition}

By shrinking $\clps$ if necessary can guarantee also that for $0<\ep\leq \clps$, for every $x,x'\in M$,   
\begin{itemize}[leftmargin=*]
\item $\Pcu{x,\epsilon}\cap \Pcu{x',\epsilon}\neq\emptyset\Rightarrow \Pcu{x,\epsilon}\cap \Pcu{x',\epsilon}\subset \Pcu{x'',3\epsilon}$ for some $x''\in M$.
\item $B(x,\ep)\cap B(x',\epsilon)\neq\emptyset\Rightarrow B(x,\epsilon)\cap B(x',\epsilon)\subset B(x'',5\epsilon)$ for some $x''\in M$.	
\end{itemize}

\smallskip 
 
To understand better the structure of dynamical boxes we note the following.

\begin{lemma}\label{lem:holisometry}
Fix $x\in M$ and for $y\in\Ws{y}$ consider the locally defined holonomy $\hs_{x,y}:\Wcu{x}\to\Wcu{y}$. Then $\hs$ sends center leaves into center leaves and $\hs|\Wc{x}$ is an isometry. Similarly for unstable holonomies.
\end{lemma}	

\begin{proof}
Let $\delta>0$ sufficiently small so that $\hs_{x,y}:A=\Wcu{x,\delta}\to B\subset\Wcu{y}$ is a well defined homeomorphism. Since $f$ preserves $\Fcu,\Fc,\Fs$ it follows \[\hs_{x,y}=f^{-n}\circ \hs_{f^nx,f^ny}\circ f^n\quad \forall n\geq 0\]
where $\hs_{f^nx,f^ny}:f^n(A)\to f^n(B)\subset \Wc{f^ny}$. Using that $f|\Fc$ is an isometry and $f|\Fu$ expands distances, it follows that $\hs$ sends $\Wc{x,\delta}$ to $\Wc{y}$ and that the differentiable map $\hs_{f^nx,f^ny}|\Wc{f^nx}$  approaches uniformly the identity for large $n$, hence the claim.
\end{proof}

We remark that in general $\hs$ does not send $\Fu$ to itself. See figure \ref{fig:pcu}.

\begin{figure}[h]
	\centering
	\includegraphics[width=0.7\linewidth]{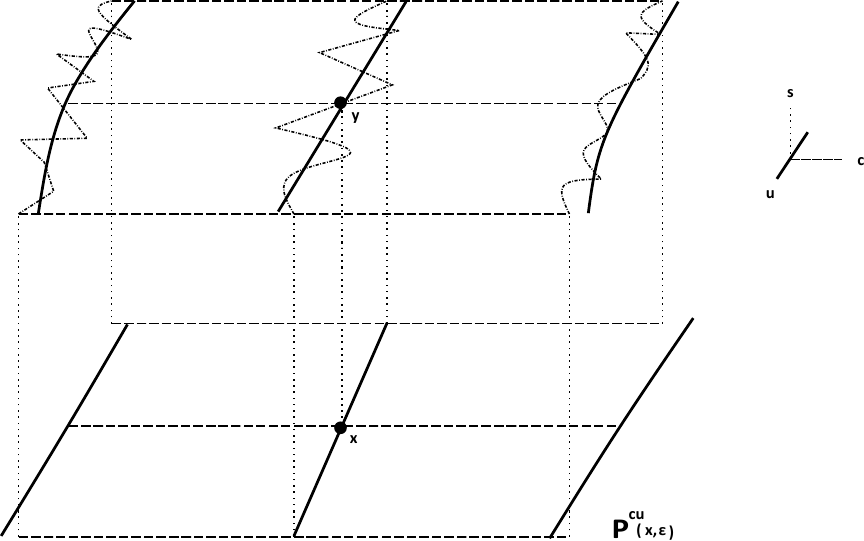}
	\caption{Comparison between $\Pcu{x,\epsilon},\Pcu{y,\epsilon}$ for $y\in\Ws{x}$. The stable holonomy sends each center plaque $\Wc{x',\ep}\in \Pcu{x,\epsilon}$ to a center plaque $\Wc{y',\ep}\in \Pcu{y,\epsilon}$, but the image of unstable plaques in $\Pcu{x,\epsilon}$  are only H\"older submanifolds in $\Pcu{y,\epsilon}$. Nevertheless, $\hs_{x,y}(\Wu{x,\epsilon})\to\Wu{x,\epsilon}$ in the $\mathcal{C}^0$ topology as $y\to x$}
	\label{fig:pcu}
\end{figure}

\begin{corollary}\label{cor:dynamicalboxsymmetric}
For every $0<\ep\leq \clps$ and $x\in M$ it holds
\[
B(x,\ep)=\bigcup_{\crampedclap{y\in \Ws{x,\epsilon}}}\ \ \hs_{y,x}(\Pcu{x,\epsilon})
\]
where $\hs_{x,y}:\Wcu{x}\rightarrow \Wcu{y}$ is the locally defined Poincaré map sending $x$ to $y$.
\end{corollary}

\begin{proof}
This is direct consequence of Lemma \ref{lem:holisometry} above plus the fact that $\Fc$ sub-foliates $\Fcu,\Fcs$.
\end{proof}

In view of this Corollary, sets $\hs_{x,y}(\Pcu{x,\epsilon})$ will be referred as the $cu$-plaques of $B(x,\ep)$. Another consequence is the following.

\begin{corollary}
For every $0<\ep\leq \clps$ and $x\in M$, $B(x,\ep) \subset M$ is open.
\end{corollary}

\begin{proof}
Fix $x\in M, \ep>0$ and define $b:\Wc{x,\ep}\times \Ws{x,\ep}\to \Wcs{x}$ by
\[
	b(y,z)=\Ws{y,2\ep}\cap \Wc{z,2\ep}.
\]
This map is continuous, injective and by \Cref{lem:holisometry} $\Im(b)=\Pcs{x,\epsilon}$. The invariance of domain theorem now implies that $\Pcs{x,\epsilon}$ is an open subset of $\Wcs{x}$.

Analogously, define $\tilde{b}: \Pcs{x,\ep}\times\Wu{x,\ep}\to M$,
\[
	\tilde b(y,z)=\hu_{x,z}(y);
\]
$\tilde b$ is continuous, injective and by definition $B(x,\ep)=\Im(\tilde b)$. Another application of the invariance of domain theorem implies that $B(x,\ep)$ is open.
\end{proof}

\smallskip 
 
\paragraph{The equilibrium state.}We now construct the equilibrium state. Fix $B=B(x_0,\epsilon)$ a dynamical box with $0<\ep<\clps/2$. By the choice of $\epsilon$ the set $\Wu{x_0,B}=\Wu{x_0,2\epsilon}\cap B$ consists of a unique unstable plaque. Similar considerations hold for the other foliations.

If $W$ is a $cs$-plaque of $B$ and $x\in B$ denote $x_w=\Wu{x,B}\cap W$. For $U\subset B$  open sub-box, define the function $\alpha_{W,B}^U:W\rightarrow \Real$ by
\begin{equation}\label{eq.alphaWB}
\alpha_{W,B}^U(w)=\nu^u_{w}(U\cap \Wu{w,B}).
\end{equation}

\begin{lemma}\label{lem:semicont}
$\alpha_{W,B}^U$ is upper semi-continuous, hence measurable.
\end{lemma}

\begin{proof}
The SRB case is direct, so we will consider only when $\varphi$ is constant along center leaves. 

We use the following notation: for $\gamma>0$ and $C$ contained in a $cu$-plaque $Z\subset B$, we write 
\[
\mathrm{Col}_{\gamma}(C):=\{y\in Z:\exists y'\in C\text{ s.t. }d_{\Fcu}(y,y')<\gamma\}.
\]
Consider a sequence $(y_m)_m\subset W$, $y_m\xrightarrow[m\mapsto\oo]{}y$, let
\begin{align*}
A=\Wu{y,B}\cap U,\\
A_m=\Wu{y_m,B}\cap U.
\end{align*}
 By definition of $\mu^u_y$ \eqref{eq.defmuxcconstant}, $\mu_y^u(A)=\mcux[y](\Wc{A,\ccen})$, and since $\mcux[y]$ is a Radon measure, 
\[
 	\mu_y^u(A)=\inf_{\gamma>0} \mcux[y]\left(\mathrm{Col}_{\gamma}(\Wc{A,\ccen})\right).
\]
Similar considerations apply to $\mu_{y_m}^u(A_m)$, and to the measures $\nu^u_y, \nu^u_{y_m}$; in particular, if 
\[
A_{\gamma}=\mathrm{Col}_{\gamma}(\Wc{A,\ccen}\cap \Wu{y}),
\]
then $A_{\gamma}$ is an open neighborhood of $A$ inside $\Wu{y}$, and
\[
	\nu^u_y(A)=\inf_{\gamma>0} \int_{\Wc{A_{\gamma},\ccen}} \Delta_y^u\circ \pi^c_y \der \mcux[y].
\]
Fix $0<\gamma<\clps-\ccen$. By \Cref{lem:JacobianoT} the function $y'\mapsto \Delta_{y_w'}(y')$ is positive and uniformly continuous on $B$, hence given $\tau>0$ there exists $0<\rho<\tau$ such that $y'\in B$, \(d(y,y')<\rho\Rightarrow \left|\frac{\Delta_{y_w}(y)}{\Delta_{y_w'}(y')}-1\right|<\tau\). Due to continuity of the unstable and center unstable foliations, there exists $0<\rho'<\rho$ so that if $y'\in B, d(y,y')<\rho'$, then 
\begin{itemize}
	\item the local stable holonomy between $\Wcu{y',\ccen+\gamma}$ and $\Wcu{y}$ is well defined, with 
	 $\Jacs[y,y'](z)<1+\tau$, for every $z\in \hs_{y,y'}(\Wcu{y,\ccen+\gamma})$ (cf.\@ \Cref{rem:jacobianoconvergeauno}).
	 \item $\Wc{\Wu{y',B}\cap U,\ccen}$ is $\rho$-equivalent to some subset $E_{y'}\subset\mathrm{Col}_{\gamma}(\Wc{A,\ccen})$.
\end{itemize}
Using continuity of the product of functions we can further assume 
\begin{itemize}
 	\item $|\Delta^u_{y'}\circ \pi^c_{y'}\circ \hs_{y',y}(z)-\Delta^u_y\circ \pi^c_y(z)|<\tau$ for every $z\in E_{y'}$.
 \end{itemize}

Take $m_0$ so that $m\geq m_0$ implies $d(y_m,y)<\rho'$ and write $E_m=E_{y_m}$, 
\[
g_m:=\hs_{y_m,y}| : E_m\rightarrow \Wc{A_m,\ccen}.
\]

Then
\begin{align*}
\MoveEqLeft \nu^{u}_{y_m}(A_m)=\int_{\Wc{A_m,\ccen}}\Delta^{u}_{y_m}\circ \pi^c_{y_m}\der\mcux[y_m]=\int_{g_m(E_m)}\Delta^{u}_{y_m}\circ \pi^c_{y_m}\der\mcux[y_m]\\
&=\int_{E_m}\Delta^{u}_{y_m}\circ \pi^c_{y_m}\circ g_m\cdot \Jacs[y,y_m]\der\mcux[y]\leq (1+\tau)\int_{E_m}\Delta^{u}_{y_m}\circ \pi^c_{y_m}\circ g_m \der\mcux[y]\\
&<(1+\tau)\left(\int_{E_m}\Delta^{u}_{y}\circ \pi^c_{y}\der\mcux[y]+\tau\cdot\mcux[y](E_m)\right) \\
&\leq(1+\tau)\left(\int_{\Wc{A_{\gamma},\ccen}}\Delta^{u}_{y}\circ \pi^c_{y}\der\mcux[y]+\tau\cdot\mcux[y](\mathrm{Col}_{\gamma}(\Wc{A,\ccen})\right)\\
\end{align*}
hence
\begin{align*}
&\limsup_m \alpha_{W,B}^U(y_m)=\limsup_m \nu^{u}_{y_m}(A_m)\leq (1+\tau)\left(\int_{\Wc{A_{\gamma},\ccen}}\Delta^{u}_{y}\circ \pi^c_{y}\der\mcux[y]+\tau\cdot\mcux[y](\mathrm{Col}_{\gamma}(\Wc{A,\ccen})\right).
\end{align*}
Taking limit as $\tau\mapsto 0$ we obtain
\begin{align*}
&\limsup_m \alpha_{W,B}^U(y_m)\leq \int_{\Wc{A_{\gamma},\ccen}}\Delta^{u}_{y}\circ \pi^c_{y}\der\mcux[y] \shortintertext{which in turn implies}\\
& \limsup_m \alpha_{W,B}^U(y_m)\leq \inf_{\gamma>0}\int_{\Wc{A_{\gamma},\ccen}}\Delta^{u}_{y}\circ \pi^c_{y}\der\mcux[y]=\alpha_{W,B}^U(y),
\end{align*}
as we wanted to show.
\end{proof}

Standard arguments of measure theory together with the previous Lemma permit us to define a Borel measure $\mathrm{m}^{W,B}$ on $B$ by the condition
\begin{equation*}
U\subset B \text{ open sub-box }\Rightarrow \mathrm{m}^{W,B}(U):=\int \alpha_{W,B}^U(w)\mcsx[w_0](\der w)
\end{equation*}
where $w_0\in W$. 

Let $W'$ be another $cs$-plaque of $\Fcs|B$ and consider $\hu_{w_0,w_0'}:W\rightarrow W'$ the corresponding Poincaré map. Fix $U\subset B$ open, and to simplify the notation write $U_w^u:= U\cap \Wu{w,B}$. We compute
\begin{align*}
\mathrm{m}^{W',B}(U)&=\int_{W'} \nu^u_{w'}(U^u_{w'})\mcsx[w_0'](\der w')=
\int_{W} \nu^u_{\hu_{w_0,w_0'}(w)}(U^u_{\hu_{w_0,w_0'}(w)})
\Jacu[w_0,w_0'](w)\mcsx[w_0](\der w)\\
&=\int_{W} \left(\nu^u_{\hu_{w_0,w_0'}(w)}(U^u_{w})\prod_{k=1}^{+\oo}\frac{e^{\varphi\circ f^{-k}\hu_{w_0,w_0'}(w)}}{e^{\varphi\circ f^{-k}(w)}}\right)\mcsx[w_0](\der w)\\
&=\int_W \nu^u_w(U_w)\mcsx[w_0](\der w)=\mathrm{m}^{W,B}(U),
\end{align*}
hence $\mathrm{m}^{W',B}=\mathrm{m}^{W,B}$: we write $\mathrm{m}^B=\mathrm{m}^{W,B}$ for any plaque $W\in\Fcs|B$. The following is now clear.

\begin{lemma}\label{lem:independencebox}
	Consider $0<\epsilon<\epsilon'<\frac{\clps}{2}$ and $x,x'\in M$ such that $B=B(x,\epsilon)\subset B(x',\epsilon')$. Then $\mathrm{m}^B=\mathrm{m}^{B'}|B$.
\end{lemma}

Take a finite covering $\mathscr{B}=\{B_1=B(x_1,\ccov),\ldots, B_R=B(x_R,\ccov)\}$ of $M$ with $0<\ccov<\frac{\clps}{5}$ and define the measure $\mathrm{m}$ by the condition: $A\subset B_i$ is Borel $\Rightarrow$ $\mathrm{m}(A):=\mathrm{m}^{B_i}(A)$. If $A\subset B_i\cap B_j$ then by our choice of $\ccov$ there exists $x\in M$ with $A\subset B(x,\epsilon_0)$, and thus by the previous lemma $\mathrm{m}^{B_i}(A)=\mathrm{m}^{B_j}(A)$, hence $\mathrm{m}$ is a well defined Borel measure on $M$. With no loss of generality we assume further that for every $i$ there exists $x_i'$ such that $f(B_i)\subset B(x_i',\clps)$. 

\begin{proposition}\label{pro:mesinvariante}
	The measure $\mathrm{m}$ is $f$-invariant.
\end{proposition}
\begin{proof}
	Fix $\mathrm{m}^{B_i}=\mm^{W_i,B_i}$, let $B'=B(x_i',\clps)$ so that $f(B_i)\subset B(x_i',\clps)$ and denote $W'$ the $cs$-plaque of $B'$ containing $f(W_i)$. Consider also $U\subset f(B_i)$ open sub-box of $B'$. Then $f^{-1}U \subset B_i$ is an open sub-box, and for $w\in B_i$, $(f^{-1}U)^u_w=f^{-1}(U^u_{fw})$, therefore
    \[
    f\mm^{W_i,B_i}(U)=\mm^{W_i,B_i}(f^{-1}U)=\int_{W_i} \nu^u_w(f^{-1}U_{fw}^u)\mcsx[w_0](\der w).
    \]
    The quasi-invariance of $\nu^u_w(f^{-1}U_{fw}^u)$, formula \eqref{eq:invariancianu} permits us to write
    \[
    \nu^u_w\left(f^{-1}(U_{fw}^u)\right)=e^{\varphi(w)-P}\nu^u_{fw}(U_{fw}^u)
    \]
    thus
    \begin{align*}
    f\mm^{W_i,B_i}(U)=\int_{W_i}  e^{\varphi(w)-P}\nu^u_{fw}(U_{fw}^u)\mcsx[w_0](\der w)=e^{-P}\int_{W_i}  e^{\alpha^{W',U}(fw)\varphi(w)}\mcsx[w_0](\der w) \shortintertext{which by 2. of \Cref{pro:medidacs} can be written as}\\
    e^{P'-P}\int_{fW_i} \alpha^{W',U}(z)\mcsx[fw_0](\der z)=e^{P'-P}\int_{W'} \alpha^{W',U}(z)\mcsx[fw_0](\der z)=e^{P'-P}\mm^{W',B'}(U),
    \end{align*}
    where in the last equality we have used that $U \subset f(B_i)$. Since $U$ is an arbitrary open sub-box, we have proven that $f\mm^{W_i,B_i}=e^{P'-P}\mm^{W',B'}$, and this implies $f\mm=e^{P'-P}\mm$. The measure $\mm$ is finite on compact sets and $fM=M$, thus $P=P'$ and $\mm$ is $f$-invariant.
\end{proof}

\begin{corollary}\label{cor:PigualP}
	It holds $P=P'$.
\end{corollary}

\begin{definition}\label{def.gibbsmeasure}
	We denote $\muf$ the probability measure $\frac{\mathrm{m}}{\mathrm{m}(M)}$.
\end{definition}

\begin{remark}\label{rem:otroeq}
Assuming that existence of $\{\mu^s_x\}_{x\in M}$ satisfying $f^{-1}\msx[fx]=e^{\varphi-P'}\msx[x]$ (for example, in the $c-$constant case), we can argue similarly and construct another invariant measure $\tmuf$ which is given locally as
\[
\tmuf(U)=\int_V \nu_v^s(U_v^s)\mcux[z_0](\der v).
\]
where $V=$ is a $cu$-plaque of some dynamical box containing $U$, $z_0\in V$.
\end{remark}

\smallskip

\noindent\textbf{Convention:} The family $\mathscr{B}=\{B_1,\ldots B_R\}$ together with corresponding center stable plaques $W_1,\ldots W_R$ is considered to be fixed.

\smallskip 
 
Directly from the construction we get:

\begin{corollary}\label{cor:condicionalesdem}
If $\xi$ is a $\muf$ measurable partition subordinated to $\Fu$ then for $\muf$-almost every $x$ it holds
\[
	(\muf)_x^{\xi}=\frac{\nu^u_x(\cdot|\xi(x))}{\nu^u_x(\xi(x))}.
\]
Assuming the existence of $\tmuf$, similarly we get that if $\xi$ is a $\tmuf$-measurable partition subordinated to $\Fs$, then for $\tmuf$-almost every $x$ it holds
\[
	(\tmuf)_x^{\xi}=\frac{\nu^s_x(\cdot|\xi(x))}{\nu^s_x(\xi(x))}.
\]
\end{corollary}

\begin{remark}
We will prove later that $\muf=\tmuf$, provided that these exist.
\end{remark}

\subsection{Product structure and the Gibb's property} 
\label{sub:product_structure_and_the_gibb_s_property}

We start establishing the following.

\begin{proposition}[Product structure of $\muf$]\label{pro:estructuraproducto}
There exists $0<\cpsm\leq\ccov$ such that for every $0<\ep\leq\cpsm$ there exists $c(\ep)>0$ satisfying for every $x\in M$,
\[
\frac{1}{c}\leq \frac{\muf(B(x,\ep))}{\mux(\Wu{x,B(x,\ep)})\cdot\mcsx{(\Pcs{x,\epsilon}})}\leq c.
\]
\end{proposition}

The proof follows directly by the definition of $\muf$ with the following lemma.

\begin{lemma}\label{lem:comparacionentretamanhos}
There exists $0<\cpsm\leq\ccov$ such that for $0<\ep\leq \cpsm$ we can find $c(\ep)>0$ such that for all $x,x'\in M$ the following holds.
\begin{align*}
&(\ast)\quad c(\ep)^{-1}\leq \frac{\mcux(\Wcu{x,\ep})}{\mcux[x'](\Wcu{x',\ep})}, \frac{\mcsx(\Wcs{x,\ep})}{\mcsx[x'](\Wcs{x',\ep})}\leq c(\ep)\\
&(\ast\ast)\quad c(\ep)^{-1}\leq \frac{\mux(\Wu{x,\ep})}{\mux[x'](\Wu{x',\ep})} \leq c(\ep)\\
&(\ast\ast\ast)\quad c(\ep)^{-1}\leq \frac{\nux(\Wu{x,\ep})}{\nux[x'](\Wu{x',\ep})} \leq c(\ep)
\end{align*}
\end{lemma}

\begin{proof}
Let $\cpsm$ be the Lebesgue number associated to the covering $\mathscr{B}$. We start with $(\ast)$; it suffices to consider the first set of inequalities. Since $\mu^{cu}_x$ is a Radon measure of full support, we have the following: for $\ep>0$ fixed, there exists $\ga>0$ and $c_0(x,\ep)>0$ such that
\[
	x'\in\Wcu{x,\ga}\Rightarrow c_0(\ep,x)^{-1}\leq \frac{\mcux(\Wcu{x,\ep})}{\mcux[x'](\Wcu{x',\ep})}\leq c_0(\ep,x).
\]
By \Cref{lem:JacobianoT} it is no loss of generality to assume that the same holds for $x'\in D(x,\ga)$, and hence due to the compactness of $M$ we get a uniform $c(\ep)$ independent of $x$.

For $(\ast\ast)$, the claim is immediate for the SRB case, so we just need to consider when $\varphi$ is constant along centrals. But this follows from as in the previous part since $\mux(A)=\mcux[x](\Wc{A,\ccen})$.

Finally, $(\ast\ast\ast)$ is consequence of $(\ast\ast)$ together with \Cref{lem:JacobianoT}.
\end{proof}

\begin{remark}\label{rem:estructuraproducto}
Arguing similarly and re-defining $c,\cpsm$ if necessary we can guarantee that for every $x\in M$,
\[
\frac{1}{c(\ep)}\leq \frac{\tmuf(B(x,\ep))}{\msx(\Ws{x,B(x,\ep)})\cdot\mcux{(\Pcu{x,\epsilon}})}\leq c(\ep).
\]
\end{remark}

Next we establish a Gibbs's type property for $\muf$ (resp. $\tmuf$) that will allow us to show that it is an equilibrium state for the potential $\varphi$. By reducing $\cpsm$ if necessary, we can assume that for every $0<\ep<\cpsm$, for every $\epsilon$-ball $D\subset M$ the distance between $x,y\in D$ is given by
\begin{align*}
d(x,y)=\inf\{&\mathrm{length}(\alpha):\alpha:[0,1]\rightarrow D \text{ piecewise }\mathcal{C}^1,\alpha(0)=x,\alpha(1)=y, \alpha'\in E^{\ast}\ \ast=s,c,u\}.
\end{align*}
With this metric and using that $f$ is a center isometry we have 
\begin{align}\label{eq:discoBowen}
&\Benu=\Ben|\Wu{x}=f^{-n}\Wu{f^nx,\epsilon}\\
\nonumber &\Ben=\bigcup_{\mathclap{y\in \Wcs{x,\epsilon}}}\ \Benu[y] 
\end{align}

We note the following classical remark. 

\begin{lemma}\label{lem:Bowenproperty}
Given $\epsilon>0$ there exists $K(\ep)>0$ with the property that for every $x,y\in M, n\geq0$ it holds
\begin{enumerate}
\item \(y\in \Benu\Rightarrow |\SB(x)-\SB(y)|<K\).
\item \(y\in \Ben\Rightarrow |\SB(x)-\SB(y)|<nK\). If $\varphi$ is constant on central leaves, then \(y\in \Ben\Rightarrow |\SB(x)-\SB(y)|<K\).
\end{enumerate}	
Moreover, $K(\epsilon)\mapsto 0$ as $\epsilon\mapsto0$. 
\end{lemma}

\begin{proof}
	We will prove only the second part, as the arguments for the first are included in the proof of this case. Let $y_s:=\Ws{y,\epsilon}\cap f^{-n}\Wcu{f^nx,\epsilon}$. Points in the same stable leaf are exponentially contracted and $\varphi$ is H\"older, thus
	\[
	|\SB(x)-\SB(y)|\leq |\SB(x)-\SB(y_s)|+C_1(\epsilon)
	\]
for some constant $C_1(\epsilon)$ that converges to zero as $\epsilon$ tends to zero. Since $f^ny_s\in \Wcu{f^nx,\epsilon}$, the point $z_u:=\Wu{f^ny,\epsilon}\cap \Wc{f^nx,\epsilon}$ is well defined. Observe then that the distance $d(f^{-j}z_u,f^{n-j}y_s)$ is exponentially contracted under $f^{-1}$, and thus for some other similar constant $C_2(\epsilon)$ we have
	\[
	|\SB(x)-\SB(y)\leq |\SB(x)-\SB(f^{-n}z_u)|+C_2(\epsilon).
	\]
Since $f$ is a center isometry and $\varphi$ is H\"older it follows that
\[
|\SB(x)-\SB(f^{-n}z_u)|\begin{dcases}
=0 & c\text{-constant case}\\
\leq C_{\varphi}n\ep^{\theta}:=C_{3}(\ep) & \text{ in general}. 
\end{dcases}
\]
Defining $K(\ep):=C_{1}(\ep)+C_{2}(\ep)+C_{3}(\ep)$ we finish the proof.
\end{proof}

Therefore, in general, the Birkhoff's sums for points in the same $(\ep,n)$-Bowen ball differ by a factor that grows linearly with $n$. This is not that convenient; to circumvent this the authors in \cite{Climenhaga2020} work with a type of potential that satisfy for sufficiently small $\ep$,
\[
	\exists\ Q_{\varphi}>0: \forall x,y\in M, y\in \Wcs{x,\ep}, \forall n\in\Nat \Rightarrow |\SB(x)-\SB(y)|<Q_{\varphi}. 
\]
They call this the \emph{$cs$-Bowen property for $\varphi$}, and they prove that the geometric potential satisfies this condition. In fact, inspection of their arguments shows that the following holds. 

\begin{lemma}\label{lem:Bowenpropertycs}
If $\mu^u$ is strongly absolutely continuous\footnote{\Cref{def:strongabsolutecontinuous}}, then $\varphi$ satisfies the $cs$-Bowen property.

In particular, the geometrical potential satisfy the $cs$-Bowen property.
\end{lemma}

\begin{proof}
The proof is similar as the proof of Theorem $5.2$ of the above cited article. For convenience of the reader we re-phrase these arguments in our setting.

Since $\mu^u$ is strongly absolutely continuous, given $\tau>0$ there exists $\rho>0$ so that for every $x,y\in M$, it holds that if $A\subset \Wu{x}, B \subset \Wu{y}$ are $\rho$-equivalent then $\hol[cs]_{x,y}\mu^u_y|B=J(x,y,\cdot)\mu^u_{x}|A$, and $|J(x,y,z)-1|<\tau$ for every $z\in A$. Take $\ep>0$ sufficiently small so that $y\in \Wcs{x,\rho/2}$ implies that $\Wu{y,\ep}$ is $\rho$-equivalent to some subset of $\Wu{x}$. Recalling that $D^u(y,\ep,n)=f^{-n}\Wu{f^ny,\ep}$ and since $f$ is a center isometry, it follows that if $y\in \Wcs{x,\rho/2}$ then for every $n\geq 0$, $D^u(y,\ep,n)$ is $\rho$-equivalent to some subset of $J(x,\ep,n) \subset\Wu{x}$, and $f^nJ(x,\ep,n)$ $\rho$-equivalent to $\Wu{f^ny,\ep}$. By reducing $\rho$ if necessary we can guarantee that $f^n(J(x,\ep,n)) \subset \Wu{f^nx,2\ep}$.

Quasi-invariance of $\mu^u$ and \Cref{lem:Bowenproperty} now implies
\begin{align*}
\frac{\mux[f^ny](\Wu{f^ny,\ep})}{\mux{f^n(J(x,\ep,n))}}&=\frac{\int_{D^u(y,\ep,n)} e^{-\SB(z)}\der\mux[y](z)}{\int_{J(x,\ep,n)} e^{-\SB(w)}\der\mux(w)}=e^{\SB(x)-\SB(y)}\frac{\int_{D^u(y,\ep,n)} e^{\SB(y)-\SB(z)}\der\mux[y](z)}{\int_{J(x,\ep,n)} e^{\SB(x)-\SB(w)}\der\mux(w)}\\
&\begin{dcases}
\leq  e^{\SB(x)-\SB(y)+2K(\ep)}\frac{\mux[y](D^u(y,\ep,n))}{\mux[x](J(x,\ep,n))}\\
\geq  e^{\SB(x)-\SB(y)-2K(\ep)}\frac{\mux[y](D^u(y,\ep,n))}{\mux[x](J(x,\ep,n))}
\end{dcases}
\end{align*}
and therefore
\[
	\frac{1-\tau}{1+\tau}e^{-2K(\ep)}\leq e^{\SB(x)-\SB(y)}\leq \frac{1+\tau}{1-\tau}e^{2K(\ep)},
\]
for every $n\geq 0$.
\end{proof}

\begin{remark}
If $\varphi$ satisfies the $cs$-Bowen property then it is easy to see that one can change $nK(\ep)$ in \Cref{lem:Bowenproperty} by $K(\ep)$.
\end{remark}

\smallskip 
 
Now we can deduce: 

\begin{corollary}[Gibbs-like property of $\muf$]\label{cor:gibbslineal}
	Given $\epsilon>0$ there exists $d(\epsilon)>0$ such that for every $n\in\mathbb{N},x\in M$ it holds
	\begin{align*}
	d^{-1}(\ep)\leq &\frac{\nux(\Benu)}{e^{\SB(x)-nP}} \leq d(\ep)\\
	d^{-1}(\ep)e^{-K(\ep)n}\leq &\frac{\muf(\Ben)}{e^{\SB(x)-nP}} \leq d(\ep)e^{K(\ep)n}
	\end{align*}
	where $K(\ep)$ is the constant given in \Cref{lem:Bowenproperty}. 

	If $\varphi$ does not depend on centrals, or if $\varphi$ is the geometric potential, then in the second line of inequalities we can substitute $nK(\ep)$ by $K(\ep)$.
\end{corollary}

\begin{proof}
	To simplify the notation we will write $a\simeq b$ if the quotient $\frac{a}{b}$ is bounded above and below with constants only depending on $\epsilon$. The first part is direct consequence of \Cref{lem:comparacionentretamanhos} and \eqref{eq:invariancianu}. For the second part we consider $B_i$ so that 
    $D(x,\ep,n)\subset B_i$: by the considerations made above \Cref{lem:independencebox} it is no loss of generality to assume that $W_i=\Wcs{x,B}$. Thus we have $\mathrm{m}^{W_i,B_i}(\Ben)=\int_{\Wcs{x,\epsilon}} \nux[y_w](\Benu[y])\mcsx(\der y)$, and therefore
	\begin{align*}
	\mathrm{m}^{W_i,B_i}(\Ben)&\simeq\int_{\Wcs{x,\epsilon}} e^{\SB(y)-nP}\mcsx(\der y)\quad &\text{(by the first part)}\\
	&\simeq e^{\SB(x)-nP\pm nK}\mcsx(\Wcs{x,\epsilon})\quad&\text{(by \Cref{lem:Bowenproperty})}\\
	&\simeq e^{\SB(x)-nP\pm nK} \quad &\text{(by \Cref{lem:comparacionentretamanhos})}.
	\end{align*}
	The last claim follows from \Cref{lem:Bowenproperty,lem:Bowenpropertycs}. 
\end{proof}

We are ready to show that $\muf$ is an equilibrium state.

\begin{proposition}\label{pro:PigualPresion}
	It holds $P=\Ptop(\varphi)=P_{\mu_{\varphi}}(\varphi).$
\end{proposition}

\begin{proof}
	Fix $0<\epsilon<\cpsm$ small. By the previous Corollary, and since $\muf$ is invariant, one has
    \[
    \frac{\log d(\ep)}{n}-K(\ep) \leq P-(\int \log \muf(\Ben) \der\muf +\int \varphi \der\muf) \leq \frac{\log d(\ep)}{n}+K(\ep)
    \]
    which by the Brin-Katok formula \cite{brinkatok} gives
    \[
    -K(\ep)\leq P-(h_{\muf}(f;A)+\int \varphi \der\muf)\leq K(\ep)\Rightarrow P_{\muf}(\varphi)=P.
    \]
    Next we consider a maximally $(n,2\ep)$-separated set $E_n$: then for $\{\Ben\}_{x\in E_n}$ are pairwise disjoint and $M=\bigcup_{x\in E_n} D(x,2\ep,n)$. Take $A^n=\{A^n_x:x\in E_n\}$ partition by Borel sets so that for every $x\in E_n$, $D(x,\ep,n)\subset A_x^n\subset D(x,2\epsilon,n)$, and compute\footnote{Recall the definition of $S(\ep,2)$ on page $2$.}
    \begin{align*}
	S(2\epsilon,n)\leq \sum_{x\in E_n} e^{\SB(x)}\leq \sum_{x\in E_n}\muf(D(x,\ep,n))  e^{nP+nK(\ep)}\leq e^{nP+nK(\ep)}\sum_{x\in E_n}\muf(A_x^n)= e^{nP+nK(\ep)}.
	\end{align*}
	It follows that $\Ptop(\varphi)\leq P$, and therefore by the variational principle $P=\Ptop(\varphi)= P_{\muf}(f)$.
\end{proof}

\begin{remark}
The Gibbs property given in \Cref{cor:gibbslineal} is presented in more generality than needed. Still, since the theory is still under development, there is perhaps some value in presenting the property in this form. In particular, it shows that the presence of the linear term $nK(\ep)$ instead of $K(\ep)$ is enough to establish \Cref{pro:PigualPresion}.
\end{remark}

Since $P=0$ by construction in the SRB case, we get:

\begin{corollary}\label{cor.obvioSRB}
If $\mm_{\scriptscriptstyle SRB}$ is the measure associated to $\varphi=-\log\det Df|E^u$, then $P_{\mm_{\scriptscriptstyle SRB}}(f)=0$.
\end{corollary}

\section{A characterization of equilibrium states} 
\label{sec:a_characterization_of_equilibrium_states}

The aim of this section is to characterize equilibrium states in terms of their conditional measures along strong unstable leaves, very much as the case of SRB measures, where they are characterized by the property of their conditionals being absolutely continuous with respect to Lebesgue measures along unstables. We will rely on the seminal work of Ledrappier-Young \cite{LedYoungI}.

\smallskip

\noindent\textbf{Hypotheses of this section:} $f:M\rightarrow M$ is a $\mathcal{C}^2$ center isometry, $\varphi:M\to\Real$ is a Hölder potential, and there is a family of measures $\mu^{u}$ satisfying $2, 3$ of \hyperlink{theoremA}{Theorem A} for $\varphi$.

\smallskip 

Of course, the existence of such families is guaranteed only on some cases (for example, when $\varphi$ is either constant along centrals or $\varphi=-\log |\det Df|E^u|$), but since the arguments that we will present only depend on this existence, there is value in writing this part abstractly.

Let $\mm\in \PTM{f}{M}$ and $\nu^{\ast}\in \mathrm{Meas}(\F^{\ast}), \ast=s,u$. We say that $\mm$ has conditionals 
equivalent to $\nu^{\ast}$ along $\F^{\ast}$ if for every $\mm$-measurable partition $\xi$ subordinated to the (possibly non-measurable partition) $\F^{\ast}$ it holds $\mm^{\xi}_x\sim \nu_x^{\ast}$ for $\maex[x]$. See \cite{Rokhlin} for a discussion of measurable partitions and conditional measures.

Define,
\begin{align}
\label{eq:tienecondicionalmu} \PM[M, \mu^u]&=\{\mm\in \PM[M]:\mm\text{ has conditionals equivalent to } \mu^u\}\\
\label{eq:tienecondicionalmuinv} \PTM{f}{M, \mu^u}&=\PM[M, \mu^u]\cap \PTM{f}{M}.
\end{align}

We will prove the following.

\begin{theorem}\label{thm:condicionalesdeterminanee}
The set $\PTM{f}{M\, \mu^u}$ coincides with the set of equilibrium states for $\varphi$.
\end{theorem}

\begin{remark}
Implicit in the statement of the above theorem is the existence of at least one equilibrium state for $(f,\varphi)$. For center isometries this is automatic: $f$ is what is called $h$-expansive and thus by \cite{topologicalconditional} the function $\mm\in \PTM{f}{M}\to h_{\mm}(f)$ is upper semi-continuous, and therefore has a maximum. Note however that this abstract method does not give any information about the equilibrium state besides its existence.
\end{remark}
\paragraph{Increasing partitions} A $\mm$-measurable partition $\xi$ is called increasing if it refines its image ($f\xi<\xi$). If $\xi$ is a partition recall that we denote by $\mathcal{B}_{\xi}(M)$ the sub-$\sigma$ algebra of $\BM[M]$ that consists of $\xi$-saturated sets. The mesh of $\xi$ is the supremum of $\{\diam(A):A\in \xi\}$.

\begin{definition}\label{def:SLY}
We call a measurable partition $\xi$ a SLY partition if it is increasing, subordinated to $\Fu$ and furthermore for $\mm\aep(x)\in M$ the atom $\xi(x)$ contains a neighborhood of $x$ inside $\Wu{x}$.
\end{definition}

In our restricted setting, a SLY partition automatically satisfies the following additional properties:
\begin{enumerate}
\item $\bigvee_{n\geq 0} f^{-n}\xi=\BM$.
\item $\bigwedge_{n\geq 0} f^n\xi=\mathcal{B}^u$
\end{enumerate}
where $\mathcal{B}^u$ is the $\sigma$-algebra of \Fu\ saturated sets. A key result is the following.

\begin{theorem}\label{thm:SLY}
There exist SLY partitions of arbitrarily small mesh. If $\xi$ is a SLY partition then
\[
h_{\mm}(f)=h_{\mm}(f,\xi)=H_{\mm}(f^{-1}\xi|\xi).
\]
\end{theorem}

The existence part is essentially due to Sinai \cite{SinaiMarkov} (compare Section 3 in \cite{LedStrelcyn} and lemma $2.4.2$ in \cite{LedYoungI}, where the construction is done in more generality). The second part is a deep result of Ledrappier and Young - Corollary 5.3 in \cite{LedYoungI}. We remark that although in this last article the result is presented for ergodic measures, it extends easily to non - ergodic ones by using the ergodic decomposition. We have tacitly used that the strong unstable manifolds coincide with the Pesin's unstable manifolds, as consequence of the fact that $f$ does not have any positive Lyapunov exponents along center directions, independently of the measure.

\subsection{Equilibrium states have conditionals equivalent to \texorpdfstring{$\mu^u$}{mu}}

Now we fix  $\xi$ a SLY partition for the measure $\mm$, and denote by $\mm_x=\mm^{\xi}_x$ the conditional measure of $\mm$ on $\xi(x)$. The main result of \cite{LedYoungI} admits the following generalization in our setting.

\begin{theorem}\label{thm:marginalesunicas}
If $\mm$ is any equilibrium state corresponding to $(f,\varphi)$ then
\[
\mm_x\sim\mux \text{ on }\xi(x)\quad\text{ for } \mm\aep(x).
\]
\end{theorem}

To give a self contained presentation we spell the complete proof instead of just citing the relevant parts of Ledrappier-Young's work. Assume for the moment that the conclusion of the theorem was true. Then for $\mm\aep(x)$ we have
\[
\der \mm_x=\rho_x \der\mux,
\]
where $\rho_x$ is a non-negative function such that $\int_{\xi(x)} \rho_x \der\mux=1$. Using that the partition $\xi$ is measurable we can assemble all $\{\rho_x\}$ into one measurable function $\rho:M\rightarrow \Real$. Since $\mm\in \PTM{f}{M}$ we have
\[
f\mm_{x}^{f^{-1}\xi}=\mm_{fx} (=\mm_{fx}^{\xi})\quad \mm\aep(x).
\]
On the other hand $f^{-1}\xi>\xi$, thus for any $A\in \BM[M]$ we can write
\begin{align*}
\frac{1}{\mm_x(f^{-1}\xi(x))}\int_{A\cap f^{-1}\xi(x)}\rho \der\mux &=\frac{1}{\mm_x(f^{-1}\xi(x))}\mm_x(A\cap f^{-1}\xi(x))=\mm_{x}^{f^{-1}\xi}(A)=\mm_{f_x}(fA\cap \xi(fx))\\
&=\int_{f(A\cap f^{-1}\xi(x))}\rho \der\mufx=\int_{A\cap f^{-1}\xi(x)}\rho\circ f e^{P-\varphi}\der\mux,
\end{align*}  
hence by uniqueness of the conditional measures $\mm\aep(x)$ it holds
\[
\frac{\rho(y)}{\mm_x(f^{-1}\xi(x))}=\rho(fy)\cdot e^{P-\varphi(y)}\quad \text{ for }\mu_x\aep(y)\in f^{-1}\xi(x).
\]
As the denominator appearing in the previous equation is $f^{-1}\xi$- measurable, we finally conclude that
the function
\[
x\mapsto \frac{\rho(x)}{\rho(f^{-1}x)}e^{P-\varphi(f^{-1}x)}
\]
is $\xi$ - measurable (i.e.\@ constant on the atoms of $\xi$). By a simple inductive argument one is led to conclude that $\rho$ satisfies
\begin{equation}\label{eq:jacobianoloc}
y\in \xi(x)\Rightarrow \frac{\rho(y)}{\rho(x)}=\prod_{k=1}^{\oo}\frac{e^{\varphi\circ f^{-k}(y)}}{e^{\varphi\circ f^{-k}(x)}}=\Delta_x^u(y).
\end{equation} 

\begin{lemma}
For every $x\in M$ the function $\Wu{x}\ni y\mapsto \log \Delta_x^u(y)$ is locally Lipschitz and hence uniformly bounded away from $0$ and $\oo$ on the atom $\xi(x)$.
\end{lemma}

The proof is straightforward using that for every $x\in M$ the set $\Wu{x}$ is an immersed submanifold and the fact that $\varphi$ is globally Hölder.

From \eqref{eq:jacobianoloc} we deduce that $\rho$ should have the form
\begin{equation}
\rho(y)=\frac{\Delta_x^u(y)}{L(x)},\quad y\in \xi(x)
\end{equation}
with $L(x)=\int_{\xi(x)} \Delta_x^u(y)d\mux(y)$.

We now define a measure $\tilde{\mm}$ by requiring $\tilde{\mm}=\mm$ on $\mathcal{B}_{\xi}$ and such its conditionals on $\xi$ are given by \(\der\tilde{\mm}_{x}=\frac{\Delta_x^u}{L(x)}\der\mu_x=\rho \der\mu_x\) (cf.\@ \eqref{eq:condicionalxi}). The proof of \Cref{thm:marginalesunicas} will be achieved by establishing that $\mm=\tilde{\mm}$, and for this it suffices to show that $\mm=\tilde{\mm}$ on every $\mathcal{B}_{f^{-n}\xi},n\geq 0$, since the previous $\sigma$ - algebras generate $\BM$.

\smallskip 

We start by noticing the following: for $q(x)=\tilde{\mm}_{x}(f^{-1}\xi(x))$ we have
\begin{align*}
q(x)&=\frac{1}{L(x)}\int_{f^{-1}(\xi(fx))} \Delta_x(f^{-1}fy) \frac{e^{P-\varphi(y)}}{e^{P-\varphi(f^{-1}fy)}}\der\mux\\
&=\frac{1}{L(x)}\int_{\xi(fx)} J_x^u(f^{-1}z) e^{\varphi(f^{-1}z)-P}\der\mufx(z)=\frac{L(fx)}{L(x)}e^{\varphi(x)-P}\leq 1.
\end{align*}
Thus $\frac{L(fx)}{L(x)}\leq e^{P-\varphi(x)}\Rightarrow \log \frac{L\circ f(x)}{L(x)}\leq P-\varphi(x)$, and therefore $\max\{\log \frac{L\circ f}{L},0\}$ is in $\Lp[1](\mm)$. A classic lemma of measure theory (cf.\@ proposition $2.2$ in \cite{LedStrelcyn}) then implies 
\[
\int \log \frac{L\circ f}{L} \der\mm=0\Rightarrow \int -\log q(x)\der\mm(x)=P-\int\varphi \der\mm. 
\]
We have shown the following.

\begin{lemma}\label{lem:lognu}
It holds,
\[
P- \int\varphi \der\mm =\int -\log\tilde{\mm}_{x}(f^{-1}\xi(x))\der\mm(x)=\int -\log\left(\int_{f^{-1}\xi(x)} \rho(y)\der\mux(y)\right) \der\mm(x).
\]
\end{lemma} 

By definition, the measures $\mm$ and $\tilde{\mm}$ coincide on $\mathcal{B}_{\eta}$; to prove that these two measures coincide it suffices to establish that $\mm$ and $\tilde{\mm}$ coincide on $\mathcal{B}_{f^{-1}\xi}$ provided that $\mm$ is an equilibrium state, and argue by induction.

Denote by $g_{f^{-1}\xi}=\frac{d\tilde{\mm}}{d\mm}\big|_{f^{-1}\xi}$, and observe that since the partition $f^{-1}\xi$ restricted to each atom $\xi(x)$ is countable, we have for $\mm\aep(x)$
\[
\dermm=\frac{\tilde{\mm}_{x}(f^{-1}\xi(x))}{m_x(f^{-1}\xi(x))}.
\] 
To prove that $\mm=\tilde{\mm}$ on $\mathcal{B}_{f^{-1}\xi}$ we will show that $\dermm\equiv 1$ for $\mm\aep(x)$. Observe that for $\mm\aep(x)$ one has $\mm_x(f^{-1}\xi(x))>0$, hence the function $\dermm$ is well defined for $\mm\aep(x)$, but not necessarily for $\tilde{\mm}\aep(x)$. This point is not emphasized in \cite{LedYoungI}, and some minor extra considerations are required to complete the argument\footnote{This was remarked some time ago by the second author in some private notes.}. For the sake of exactness we will be precise in what follows.

For an atom $\xi(x)$ we write $f^{-1}\xi|\xi(x)=\{A^x_j\}_{j=1}^{\oo}$. We compute
\begin{align}
\nonumber&\int_{\xi(x)} \dermm[y]\der\mm_x(y)=\int_{\xi(x)} \frac{\tilde{\mm}_{y}(f^{-1}\xi(y))}{\mm_y(f^{-1}\xi(y))}\der\mm_x(y)
=\int_{\xi(x)}\frac{\tilde{\mm}_{x}(f^{-1}\xi(y))}{\mm_x(f^{-1}\xi(y))}\der\mm_x(y)\\
&=\sum_{\substack{j=1\\  \mm_x(A^x_j)>0}}^{\oo}\int_{A^x_j}\frac{\tilde{\mm}_{x}(A^x_j)}{\mm_x(A^x_j)}\der\mm_x(y)
=\sum_{\mathclap{\substack{j=1\\  \mm_x(A^x_j)>0}}}^{\oo}\tilde{\mm}_{x}(A^x_j)\leq \sum_{j=1}^{\oo}\tilde{\mm}_{x}(A^x_j)=\tilde{\mm}_x(\xi(x))=1 \label{eq:ineqjensen}\\
&\Rightarrow \int  \dermm \der\mm(x)\leq 1 \label{eq:ineqjensendos}
\end{align} 

Now we use our hypothesis that $\mm$ is an equilibrium state: by \Cref{thm:SLY} and \Cref{lem:lognu},
\[
\int-\log \mm_x(f^{-1}\xi(x))\der\mm(x)=H_\mm(f^{-1}\xi|\xi)=P-\int \varphi \der\mm=\int -\log \tilde{\mm}_x(f^{-1}\xi(x))\der\mm(x)
\]
and thus \(\int \dermm \der\mm(x)=0\). Using Jensen's inequality and \eqref{eq:ineqjensendos} we conclude
\begin{align*}
0=\int \log \dermm \der\mm(x)\leq \log\int \dermm \der\mm(x)
\leq \log 1=0.
\end{align*}
Strict convexity of the logarithm function implies then 
\[
\dermm =1\quad  \mm\aep
\] 
We conclude that the chain of inequalities appearing are equalities, in particular for equation \eqref{eq:ineqjensen}. As a result we have $\mu_x(A_j^x)>0\Rightarrow \mm_x(A^x_j)>0$, and this in turn imply that $\mm\aep(x)$ it holds $\mm_x=\tilde{\mm}_x$ on $f^{-1}\xi|\xi(x)$, since these measures coincide on all the atoms and not only on the ones of $\mm_x$ positive measure. Consequently, $\mm=\tilde{\mm}$ on $\mathcal{B}_{f^{-1}\xi}$, and this finishes the general inductive step necessary to prove \Cref{thm:marginalesunicas}.

\begin{corollary}
Suppose that $\mm\in \PTM{f}{M}$ is an equilibrium state for the potential $\varphi$ and $\xi$ is a SLY partition as above. Then for $\mm\aep(x)$ the conditional measures of $\mm$ with respect to $\xi$ are given by 
\[
\mm^{\xi}_x=\tilde{\mm}^{\xi}_x=\frac{\Delta_x^u}{L(x)}d\mux|\xi(x)\quad L(x)=\int_{\xi(x)}\Delta_x^u\der\mux.
\]
In particular the disintegration can be extended to every leaf of $\Fu$ and $\mm\in \PTM{f}{M,\mu^u}$.
\end{corollary}

We remark that above we only used the existence of an equilibrium state with conditional satisfying the quasi-invariance condition $3$ of \hyperlink{theoremA}{Theorem A}.

\subsection{Conditional measures determine equilibrium states} 
\label{sub:conditional_mesures_determine_equilibrium_states}


We are interested now in the following converse of \Cref{thm:marginalesunicas}.

\begin{theorem}\label{thm:conversemarginal}
Let $\mm\in \PTM{f}{M}$ and assume that with respect to some SLY partition $\xi$ we have that 
$\mm_x<<\mu_x^u$ for $\mm\aep(x)$. Then $\mm$ is an equilibrium state.
\end{theorem}

\begin{proof}
The proof uses similar techniques to the ones used above, and is based in \cite{LedStrelcyn}. Consider $\xi$ a SLY partition for $\mm$ and define the measure $\tilde{\mm}$ on $M$ by requiring that for $A\subset M$ Borel, 
\begin{equation}
\tilde{\mm}(A)=\int \mux(A\cap \xi(x)) \der m(x).
\end{equation}
Observe that $\tilde{\mm}$ is not necessarily finite, but only $\sigma$-finite. By hypothesis $\der\mm_x=\rho \der\mu_x$, thus $\mm<<\tilde{\mm}$ and moreover $\mm\aep(x)$ it holds
\begin{equation}
\frac{\der\mm}{\der\tilde{\mm}}=\rho \quad \tilde{\mm}\aep(x)
\end{equation}
(see proposition $4.1$ in \cite{LedStrelcyn}). Note that $\mm^{f^{-1}\xi}_x=\frac{\mm_x(\cdot |f^{-1}\xi)}{\mm_x(f^{-1}\xi(x))}$; by an analogous computation as the one carried in the proof of \Cref{lem:lognu} we get
\[
\mm_x(f^{-1}\xi(x))=\frac{\rho\circ f(x)}{\rho(x)}e^{\varphi(x)-P},
\]
Finally, we compute
\begin{align*}
0\leq h_{\mm}(f)&=H(f^{-1}\xi|\xi)=\int-\log \mm_x(f^{-1}\xi(x))\der\mm(x)\\
&=P-\int \varphi d\mm +\int -\log\frac{\rho\circ f(x)}{\rho(x)} \der\mm(x)
\end{align*}
and since $\varphi$ is integrable (bounded), the function $-\log^{-} \frac{\rho\circ f}{\rho}$ is integrable as well. Using again Proposition $2.2$ of \cite{LedStrelcyn} we conclude that $\log \frac{\rho\circ f}{\rho}$  has zero integral, hence
\[
h_{\mm}(f)=P-\int \varphi \der\mm 
\]
and $\mm$ is an equilibrium state. The proof is complete.
\end{proof}

Joining \Cref{thm:marginalesunicas,thm:conversemarginal} we get \Cref{thm:condicionalesdeterminanee}. Below we list some of its consequences.

\begin{corollary}\label{cor:condicionales}
Assume also the existence of a family $\mu^s$ satisfying $2, 3$ of \hyperlink{theoremA}{Theorem A} for $\varphi$. If $\mm$ is an equilibrium state then $\mm$ has conditionals equivalent to $\mu^s,\mu^u$.  
\end{corollary}

\begin{proof}
The set of equilibrium states for $f$ coincides with the one for $f^{-1}$.
\end{proof}

\begin{remark}
The corollary above does not require uniqueness of the equilibrium state.
\end{remark}

Since the measures $\mux$ are positive on relative open sets, proposition $4.1$ in \cite{Katok1996} implies:

\begin{corollary}\label{cor:entropos}
For every equilibrium state $\mm$ for $\varphi$ it holds $h_{\mm}(f)>0$
\end{corollary}

Finally, we note that the previous theorem gives another proof of the fact that $\muf$ is an equilibrium state (and $P$ is the topological pressure), as the conditionals of $\muf$ are absolutely continuous with respect to $\mu^u$; see \Cref{cor:condicionalesdem}.

\section{Kolmogorov and Bernoulli properties} 
\label{sec:finer_ergodic_properties}

The aim of this part is to establish the Bernoulli property for equilibrium states. Again, our arguments are written in some more generality than the context that we have been working, and in particular most of them do not require uniqueness of the equilibrium state (at least under very mild additional hypotheses), using that $\muf=\tmuf$ instead.

For the remaining part of this section $f:M\to M$ is a center isometry of class $\mathcal C^2$, $\varphi$ is a Hölder potential, and $\muf=\tmuf$ are the equilibrium states of the system $(f,\varphi)$ constructed in \Cref{sec:measures_along_invariant_foliations_of_center_isometries}.

\subsection{Kolmogorov processes} 
\label{sub:kolmogorov_processes}


We will prove first that $(f,\muf)$ has the Kolmogorov property, and in particular it is mixing of all orders. For this we give two proofs. The first, which works without any extra assumptions, consists in verifying that our system satisfies the hypotheses of an abstract theorem due to Ledrappier \cite{MeasuresK}. This argument does require uniqueness of the equilibrium state for the considered potentials, which is known from \cite{Climenhaga2020}.

The second proof is more geometrical, and is based on the methods originally developed by Grayson, Pugh and Shub for establishing the stable ergodicity of the time-one map of the geodesic flow on manifolds of constant sectional curvature \cite{StErgGeo}. In this reasoning we use $\muf=\tmuf$ and ergodicity instead of uniqueness of the equilibrium state, but we require an extra technical assumption, namely the so called \emph{accessibility property}, which is a common hypothesis for these type of arguments. Nonetheless, we think that there is value in giving this more geometrical proof, with the hope that it could be applied to equilibrium states of other systems, particularly because the first method seems ad-hoc to our setting.

\paragraph{Completely positive entropy} Let $T:M\to M$ be a homeomorphism of the compact metric space $M$, fix $\mm\in \PTM{T}{M}$ and denote $\mathrm{Pin}(T,\mm)$ the Pinsker $\sigma$-algebra of the system, i.e.\@ the $\sigma$-algebra generated by
\[\{\xi \text{ countable measurable partition of }M: H_{\mm}(\xi)<\oo, h_{\mm}(T;\xi)=0 \}.
\]
We use the following characterization of Kolmogorov systems, although the reader may want to take it as the definition of the Kolmogorov property. 

\begin{theorem}[Rohlin-Sinai]\label{thm:RohlinSinai}
	The system $(T,\mm)$ has the Kolmogorov property if and only if $\mathrm{Pin}(T,\mm)=\{\emptyset,M\}\modo$.
\end{theorem}

To conclude that $(f,\muf)$ is a Kolmogorov system we use the results of \cite{MeasuresK}. We give some definitions in order of making our presentation self contained.

\begin{definition}\label{def:vaguelyexp}
We say that a map $T:M\to M$ of a compact metric space is weakly expansive\footnote{The terminology ``weakly expansive'' has several meanings in the literature. In this work we are only concerned with the one given in this page.} if for every compact metric space $N$ and continuous map $S:N\to N$, the map
\[
	\PTM{T\times S}{M\times N}\ni\mm \to h_{\mm}(T\times S)-h_{\pi_N\mm}(S)\in\Real 
\]
is upper semi-continuous.  
\end{definition}

If $T:M\to M$ is $h$-expansive then it is weakly expansive; see page $266$ in \cite{MeasuresK} and \cite{topologicalconditional}. It follows that if $f$ is a center isometry, then it is weakly expansive.

\begin{theorem}[Ledrappier]\label{thm:LedraK}
Suppose that $T:M\to M$ is weakly expansive and $\varphi:M\to \Real$ continuous. Let $R:M\times M\to M\times M$ be the product map of $T$ with itself, and denote $\Phi:M\times M\to\Real$ the function $\Phi(x,y)=\varphi(x)+\varphi(y)$. If $(R,\Phi)$ has a unique equilibrium state then $(T,\varphi)$ has an unique equilibrium state $\mm$ and $(T,\mm)$ is Kolmogorov.
\end{theorem}

\smallskip

We apply the above theorem to $T=f$; note that $R=(f,f):M\times M\to M\times M$ is a center isometry of class $\mathcal C^2$ with invariant decomposition 
\[
	E^{\sigma}_R\approx E^{\sigma}_f\oplus E^{\sigma}_f \subset TM\oplus TM\approx T(M\times M),\quad \sigma\in \{s,c,u,cs,cu\}.
\]
Moreover, by unique integrability of the stable, center and unstable foliations in center isometries (cf. the references given in the proof of \Cref{thm:centerisometry}), it follows that for every $(x,y)\in M\times M$
\[
	W^{\sigma}_R(x,y)=W^{\sigma}_f(x)\times W^{\sigma}_f(y).
\]
We deduce that $\F^s_R, \F^u_R$ are minimal. Next we consider $\varphi:M\to \Real$ a Hölder potential and $\Phi=(\varphi, \varphi)$. Then $\Phi$ is clearly Hölder with respect to the distance $d((x,y),(x',y'))=\max\{d(x,y),d(x',y')\}$ on $M\times M$. Observe:
\begin{enumerate}
 	\item if $\varphi$ is constant along center leaves of $f$, then $\Phi$ is constant along center leaves of $R$.
 	\item If $\varphi=-\log |\det Df|E^u|$, then
 	\[
 	\Phi=-\log |\det Df|E^u|-\log |\det Df|E^u|=-\log |DR|E^u_R|.
 	\]
 \end{enumerate} 

In both cases the system $(R,\Phi)$ satisfies the hypotheses of \hyperlink{theoremA}{Theorem A}, and therefore has a unique equilibrium state, hence by \Cref{thm:LedraK}  $(f,\muf)$ is Kolmogorov.

\begin{remark}
In the recent paper \cite{Call2022} the authors use a similar approach to show the $K$-property of equilibrium states associated to Hölder potentials, where the underlying dynamics is given by the geodesic flow of a (closed) rank-one manifold. 
\end{remark}

\paragraph{Juliennes and Equilibrium States.}Now we give a second, more geometrical, proof of the fact that $(f,\muf)$ is Kolmogorov. We will rely heavily on the technology developed in \cite{StableJulienne} and its subsequent refinements. As a guide, we use the general lines of Section $4$ in \cite{AccStaErg}. In this part we work in the conditions spelled at the beginning of the section, but we do not require uniqueness of the equilibrium state, nor its ergodicity, and only that $\muf=\tmuf$ (or even more generally, it is enough to assume that these measures are equivalent to each other, with uniformly bounded Radon-Nikodym derivatives). 

\smallskip

We say that a subset $X\subset M$ is $s$-saturated ($u$-saturated) if it consists of whole stable (respectively, unstable) leaves. If it is both $s$- and $u$-saturated we say that it is bi-saturated. Given a point $x\in M$ its accessibility class $\mathrm{ACC}(x)$ is the smallest bi-saturated set containing $x$. Analogously, we say that $X\subset M$ is $\mm$-essentially $s$-saturated ($u$-saturated) if there exists a $s$-saturated ($u$-saturated) Borel set $X_0$ so that $\mm(X\Delta X_0)=0$. By \cite{LedYoungI} we have:

\begin{theorem}[Ledrappier-Young]
	Let $f:M\to M$ be a $\mathcal C^2$ partially hyperbolic diffeomorphism preserving a Borel measure $\mm$. Assume that $f$ is a center isometry (or more generally, that all Lyapunov exponents corresponding to vectors in the center direction are zero). If $X\in \mathrm{Pin}(f,\mm)$ then $X$ is both $\mm$-essentially $s$-saturated and $\mm$-essentially $u$-saturated.
\end{theorem} 

What we will show is that if $X\in\BM[M]$ is $\muf$-essentially $s$-saturated, then it coincides almost everywhere with a $s$-saturated set $\mathscr{D}(X;\mathscr{J}^{ucs})$ which is relevant from the point of view of the measure (it consists of some sort of density points, see below), and analogously, if $X\in\BM[M]$ is $\muf$-essentially $u$-saturated then it coincides almost everywhere with a $u$-saturated set $\mathscr{D}(X;\mathscr{J}^{scu})$. After that we will establish that if $X$ is both $\muf$-essentially $s$-saturated and $\muf$-essentially $u$-saturated, there is equality $\mathscr{D}(X;\mathscr{J}^{scu})=\mathscr{D}(X;\mathscr{J}^{ucs})$, hence under these hypotheses $X$ coincides with a bi-saturated set. If we knew that any bi-saturated is either empty or the whole manifold, then we would be able to deduce, by the previous theorem, that $\mathrm{Pin}(f,\muf)$ is trivial and therefore the system $(f,\muf)$ is Kolmogorov. Lack of non-trivial bi-saturated sets is a common hypothesis in the theory.

\begin{definition}
The map $f$ is accessible (or has the accessibility property) if there exists $x$ so that $\mathrm{ACC}(x)=M$. 
\end{definition}  

The script of this proof was developed by Pugh and Shub as part of their program to establish stable ergodicity of (conservative) partially hyperbolic systems. In our case however, we do not have as much control for the equilibrium measure as one has for volume, therefore additional work is required. As a remark for the reader versed in Pugh-Shub's program we point out that $\mm$-essential accessibility\footnote{The map $f$ is $\mm$-essentially accessible if any bi-saturated set has measure zero or one.} seems inadequate for our purposes, since in principle we do not know the equilibrium measure. We do not know whether in the context of \hyperlink{theoremA}{Theorem A} the system $(f,\muf)$ is always $\muf$-essentially accessible or not.

For now we do not assume accessibility, but we do assume that $\mm$ is positive on open sets.

\begin{definition}
	We say that the family $\mathscr{E}=\{\{E_n(x)\}_{n\in \mathbb{N}}\}_{x\in M}$ is a regular Vitali\footnote{More general definitions of Vitali bases (not regular) are available in the literature (cf. \cite{StableJulienne}). The definition presented here is enough for our purposes.} basis if for every $x\in M$ the family $\{E_n(x)\}_{n\in \mathbb{N}}$ is a local basis of neighborhoods at $x$ such that
	\begin{enumerate}
		\item each $E_n(x)$ is a Borel set.
		\item $\diam(E_n(x))\xrightarrow[n\mapsto\oo]{} 0$.
	\end{enumerate}
	The regular basis is said to be closed if each $E_n(x)$ is closed set, and it is decreasing if $E_n(x)\supset E_{n+1}(x)$, for all $x\in M, n\in \Nat$.
\end{definition}

\begin{definition}
	Let $\mathscr{E}$ be a regular Vitali basis and $X\in\BM[M]$. We say that $x\in M$ is a $\mathscr{E}$-density point of $X$ if 
	\begin{align*}
	\lim_{n\rightarrow\oo}\frac{m(X\cap E_n(x))}{m(E_n(x))}=1.
	\end{align*}
	We denote by $\mathscr{D}(X;\mathscr{E})$ the set of $\mathscr{E}$-density points of $X$.
\end{definition}

In particular, density points for the basis $\{D(x,\frac{1}{2^n})\}_n$ will be called $B\gui L$ density points\footnote{$B\gui L$ stands for Besicovitch-Lebesgue}. It is standard that for any Borel set $X$, $\mm$-almost every point of $X$ is a $B\gui L$ density point of $X$. In the case when $\mm$ is a volume, the concept of $\mm$-essentially accessibility is a natural assumption because there is a well defined reference class of measures, and by the discussion above one can prove the Kolmogorov property establishing that for any $X\subset M$ which is both $\mm$-essentially $s$ and $u$-saturated, the set of its $B\gui L$ density points is bi-saturated. This is the approach followed in \cite{AccStaErg}.

However, for $\varphi\gui$equilibrium measures we do not have as much control on the properties as one has for volumes, and in particular we do not know how to control $B\gui L$ density points. Instead, we will work with density points corresponding to some dynamical defined Vitali bases called \emph{juliennes}, and show first that under our hypotheses 
\begin{enumerate}
	\item any $X\in\BM[M]$ coincides $\mm$-almost everywhere with the set of its julienne density points. 
	\item If $X$ is $\mm$-essentially $s$-saturated then the set of its julienne density points is $s$-saturated.
\end{enumerate}

\smallskip

\noindent\textbf{Convention:} From now on we specialize in the case $\mm=\tmuf$.

\smallskip


For some simplifications in the notation we choose to work with $\tmuf$, that is, with the construction of a equilibrium measure using $cu$-transversals explained in \Cref{rem:otroeq}. As we previously noted the properties of $\tmuf$ are completely analogous to the ones of $\muf$, and we will refer directly to them without further clarification. Until \Cref{pro.JcsigualJcu} we will not use the equality $\muf=\tmuf$. 

\smallskip

Some finer control in the local holonomies is now required, specifically that they are H\"older continuous. See Theorem A of \cite{HolFol}.

\begin{theorem}[Pugh-Shub-Wilkinson]\label{thm:HolFol}
	There exists $C_H>0,\vartheta\in (0,1)$ such that for every $x,y,z\in M$ with $y\in \Ws{x, \frac{\clps}{2}},z\in \Wu{x,\frac{\clps}{2}}$ the locally defined holonomy maps $\hs_{y,x}:\Wcu{x}\rightarrow\Wcu{y},\hu_{z,x}:\Wcs{x}\rightarrow \Wcs{y}$ are $(C_H,\vartheta)$-H\"older.
\end{theorem}

Choose numbers $0<\varepsilon<\frac{\clps}{2},0<\sigma<1$ such that $d(x,y)<\varepsilon\Rightarrow B(y,\varepsilon)\subset B(x,\frac{\clps}{2})$, $\lambda^{\vartheta}<\sigma$ and define the $u,s,c,cu,cs,ucs,scu$-juliennes as  
\begin{align*}
J^u_n(x)&:=f^{-n}(\Wu{f^nx,\varepsilon})\\
J^s_n(x)&:=f^n(\Ws{f^{-n}x,\varepsilon})\\
B^c_n(x)&:=\Wc{x;\sigma^n\varepsilon}\\
J^{cu}_n(x)&:=\bigcup_{\mathclap{y\in B^c_n(x)}}\ J^u_n(y)\\
J^{cs}_n(x)&:=\bigcup_{\mathclap{y\in B^c_n(x)}}\ J^s_n(y)\\
J^{ucs}_n(x)&:=\bigcup_{\mathclap{y\in J_{n}^{cs}(x)}}\ J^u_n(y)\\
J^{scu}_n(x)&:=\bigcup_{\mathclap{y\in J_{n}^{cu}(x)}}\ J^s_n(y).
\end{align*}

\begin{remark}
	The definition of juliennes for general partially hyperbolic maps is more involved, since no a priori knowledge of the action of $f$ on the center direction is assumed. In our case, the fact that $f$ is a center isometry simplifies all the definitions. 
\end{remark}

By \Cref{cor:gibbslineal} we obtain.

\begin{lemma}\label{lem:controljuliana}
	There exist constants $c_1(\varepsilon),c_2(\varepsilon)>0$ such that for every $x\in M$ it holds
	\[
	c_1\leq \frac{\nsx(J_n^s(x))}{e^{\SB\circ f^{-n}(x)-nP}}\leq c_2.
	\]
\end{lemma}

From now on we omit the reference of the constants respect to $\varepsilon$, which is considered to be fixed and small with respect to $\frac{\clps}{2}$. We then have the following.

\begin{corollary}\label{cor:controljuliana}
	There exist $L_1, L_2>0$ such that for every $x\in M$, for every $n\geq0$ it holds that $y,z\in J^{cu}_n(x)$ implies
	\[
	L_1\leq \frac{\nsx(J_n^s(y))}{\nsx(J_n^s(z))}\leq L_2.
	\] 
\end{corollary}

\begin{proof}
	Arguing as in the proof of \Cref{lem:Bowenproperty} we deduce the existence of some $K'>0$ such that for every $x\in M$ and $n\geq0$,
	\[
	y,z\in J^{cu}_n(x)\Rightarrow |\SB\circ f^{-n}y-\SB\circ f^{-n}z|<K';
	\]
	this relies on the fact the central size of $J^{cu}$ is exponentially small. The claimed property follows now from the previous lemma.
\end{proof}

Using the above and the product structure of $\mm$ (\Cref{pro:estructuraproducto}) we deduce the following.

\begin{corollary}\label{cor:productstructurejulienne}
	There exist $L_3(\ep),L_4(\ep)>0$ such that for every $x\in M$, for every $n\geq0$ it holds
	\[
	L_3\leq\frac{m(J^{scu}_n(x))}{\msx(J^{s}_n(x))\cdot \mcux(J^{cu}_n(x))}\leq L_4. 
	\]
\end{corollary}

\smallskip

The key property of juliennes is that they behave in a controlled matter under holonomies.

\begin{proposition}\label{pro:holjulienne}
	There exist $k\geq 1, n_0\in \Nat$ satisfying: for every $x,x'\in M,x'\in \Ws{x,\varepsilon}, n\geq n_0$ the corresponding stable holonomy $\hs:\Wcu{x,\frac{\clps}{2}}\rightarrow\Wcu{x',\clps}$ satisfies 
	\[
	J_{n+k}^{cu}(x')\subset \hs(J_{n}^{cu}(x))\subset J_{n-k}^{cu}(x')
	\]
\end{proposition}

This is completely analogous to proposition $B.8$ of \cite{AccStaErg}. As our juliennes differ from the ones used in that work, we present the proof.

\begin{proof}
	Take $z\in B^c_n(x), y\in J^u_n(z)$: then $d(z,x)\leq \sigma^n\varepsilon$ and $d(f^ny,f^nz)<\varepsilon$, therefore $d(y,z)<\lambda^n\varepsilon$ (recall that $\max_{x\in M}\{\norm{D_xf|E^s}, \norm{D_xf^{-1}|E^u}\}<\lambda<1$). If $z'=\hs(z), y'=\hs(y)$, then for $n$ sufficiently large $d(x,z),d(x,y)<\varepsilon$, hence
    \begin{align*}
    &a)\quad  d(z',z)<2\varepsilon, d(y',y)<2\varepsilon,\text{ because }d(x,x')<\varepsilon,\\
    &b)\quad d(f^ny',f^nz')\leq d(f^ny',f^ny)+d(f^ny,f^nz)+d(f^nz,f^nz')
	\leq \lambda^n\left(d(y,y')+d(z,z')\right)+d(f^ny,f^nz)\\
	&\hspace*{2.8cm}\leq \lambda^n(2\varepsilon)+\varepsilon<2\varepsilon.
    \end{align*}

  	 Define $w'$ so that $f^nw'=\Wu{f^ny',4\varepsilon}\cap \Wc{f^nz',4\varepsilon}$: it follows that 
   	 $w'=\Wu{y',\clps}\cap \Wc{z',\clps}=\Wu{y',\clps}\cap \Wc{x',\clps}$.

   	The above implies at once that there exists $k\geq 1$ independent of $x$ such that if $n-k\geq n_0$,  $y'\in J^u_{n-k}(w')$. On the other hand, since $d(y,z)<\lambda^n\varepsilon$ by \Cref{thm:HolFol} we get
	\[
	d(z',y')\leq C_H\varepsilon^{\theta}\lambda^{n\vartheta}\Rightarrow d(z',w')\leq 2C_H\varepsilon^{\theta}\lambda^{n\vartheta}
	\]
	Using \Cref{lem:holisometry} we get
	\[
	d(x',w')\leq d(x,z)+d(z',w')<\sigma^n\varepsilon+2C_H\varepsilon^{\theta}\lambda^{n\vartheta}=\sigma^{n-k}(\sigma^k(1+2C_H\varepsilon^{\theta}))\varepsilon
	\]
	and thus if $k$ is sufficiently large, $d(x',w')<\sigma^{n-k}\varepsilon$, and $w'\in B_{n-k}(x')$. This shows that $y'=\hs(y)\in J^{cu}_{n-k}(x')$ and finishes the proof of the second inclusion. The first one is obtained similarly by using $(\hs)^{-1}$. 
\end{proof}

Now we can establish a central proposition.

\begin{proposition}\label{pro:Xsusat}
	If $X$ is $\mm$-essentially $s$-saturated , then the set of its $\mathscr{J}^{scu}$-density points is $s$-saturated. 
\end{proposition}

For the proof, we will employ the following Lemma.

\begin{lemma}\label{lem:Xsusat}
	Let $X_s$ be an $s$-saturated set and $x\in M$. Then
	\[
	\lim_{n\mapsto \oo}\frac{\mm(X_s\cap J_n^{scu}(x))}{\mm(J_n^{scu}(x))}=1\Leftrightarrow \lim_{n\mapsto \oo}\frac{\mcux(X_s\cap J_n^{cu}(x))}{\mcux(J_n^{cu}(x))}=1.
	\]
\end{lemma}

The proof is not too hard, and is direct consequence of \Cref{cor:productstructurejulienne}. We omit the proof, and refer the reader to proposition $2.7$ of \cite{ErgPH}, where a completely analogous statement is proven.

\begin{proof}[Proof of \Cref{pro:Xsusat}]
	Consider $X_s$ an $s$-saturated set such that $\mm(X\Delta X_s)=0$. Let $x$ be a $\mathscr{J}^{scu}\gui$density point of $X$ and let $\hs$ be a (local) stable holonomy sending $x$ to $y$. Using the previous proposition we deduce 
	\[
	\mcux[y](\hs(J_{n+k}^{cu}(x)\cap X_s))\leq \mcux[y](J_n^{scu}(y)\cap X_s)\leq  \mcux[y](\hs(J_{n-k}^{cu}(x)\cap X_s)),
	\]
	and since at this scale $(\hs)^{-1}\mcux[y]$ is uniformly comparable to $\mcux$, we have
	\begin{align*}
	1&=\lim_{n\mapsto \oo}\frac{\mm(X_s\cap J_n^{scu}(x))}{\mm(J_n^{scu}(x))}=\lim_{n\mapsto \oo}\frac{\mcux(X_s\cap J_n^{cu}(x))}{\mcux(J_n^{cu}(x))}=\lim_{n\mapsto \oo}\frac{\mcux[y](\hs(X_s\cap J_n^{cu}(x))}{\mcux[y](\hs(J_n^{cu}(x)))}\\
	&=\lim_{n\mapsto \oo}\frac{\mcux[y](X_s\cap J_n^{cu}(y))}{\mcux[y](J_n^{cu}(y))}=\lim_{n\mapsto \oo}\frac{\mm(X_s\cap J_n^{scu}(y))}{\mm(J_n^{scu}(y))},
	\end{align*}
	i.e.\@ $y$ is also a point of $\mathscr{J}^{scu}\gui$density of $X_s$. 
\end{proof}

We will show now that any measurable $X$ coincides $\mm$-almost everywhere with $\mathscr{D}(X;\mathscr{J}^{scu})$.

\begin{definition}
	Let $\mm\in \PM[M]$ be positive on open sets and non-atomic. A regular Vitali basis $\mathscr{E}$ is $\mm$-engulfing if there exists $L>0$ with the following property: for every $x\in M,n\in\Nat$ there exists a closed set $\widehat{E}_n(x)$ such that
	\begin{enumerate}
		\item $E_{n}(x)\subset \widehat{E}_n(x)$ and $\mm(\widehat{E}_n(x))\leq L\cdot \mm(E_n(x))$.
		\item If $n'\geq n,y\in M$ and $E_{n'}(y)\cap E_n(x)\neq\emptyset$ then $E_{n'}(y)\subset \widehat{E}_n(x)$. 
		\item Given $N_0$ there exists $\delta>0$ such that $\mm(\clo{E_n(x)})<\delta$ implies $n\geq N_0$. 
	\end{enumerate} 
\end{definition}

\begin{theorem}\label{thm:baseengulfing}
	If $\mathscr{E}$ is a closed $\mm$-engulfing basis, then for every $X\in \BM[M]$ the set of its $\mathscr{E}$-density points coincide almost everywhere with $X$. 
\end{theorem}

\begin{proof}
This is proven in \cite{StableJulienne}, Theorem $3.1$. See also the discussion after Corollary $3.2$ in that article.
\end{proof}

Note the following.

\begin{lemma}\label{lem:doublingjscu}
There exist $n_l, l\in \Nat$ so that for every $x\in M, n-l\geq n_l$ implies
\[
	J^{scu}_{n}(y)\cap J^{scu}_{n}(x)\neq\emptyset\Rightarrow \clo{J^{scu}_{n}(y)}\subset J^{scu}_{n-l}(x)
\]
\end{lemma}

\begin{proof}
Observe that $J^{scu}_{n}(y)\cap J^{scu}_{n}(x)\neq\emptyset$ implies $J^{scu}_{n}(y)\subset D(x,2\varepsilon(2\lambda^n+\sigma^n))$; we consider $n\geq n_0$ so that $2\varepsilon(2\lambda^n+\sigma^n)<\frac{\clps}{4}$. Using 
\Cref{pro:holjulienne} one readily deduces the existence of $n_1\geq n_0$ and $k$ satisfying, for $n-k\geq n_1$
\[
	\hs(J^{cu}_n(y)) \subset J^{cu}_{n-k}(x);
\]
observe that the distance $d(\hs(w),w)$ for $w\in J^{cu}_n(y)$ is smaller than $4\varepsilon(2\lambda^n+\sigma^n)$. From here follows the claim.
\end{proof}

The lemma above implies in particular that the density points of the bases $\mathscr{J}^{scu}$ and $\clo{\mathscr{J}^{scu}}=\{\{\clo{J^{scu}_{n}(x)}\}_{n\in\Nat}:x\in M\}$ coincide: since the later is a closed basis we will be able to apply \Cref{thm:baseengulfing} and deduce that for any measurable set $X$, $\mathscr{D}(X;\mathscr{J}^{scu})$ coincides 
$\mm$-almost everywhere with $X$. Of course, we need to check that $\clo{\mathscr{J}^{scu}}$, or equivalently, $\mathscr{J}^{scu}$ is $\mm$-engulfing.

\begin{remark}
The measure $\tmuf$ is non-atomic. Under the hypothesis of ergodicity this is direct from \Cref{cor:entropos}. Without this hypothesis, it can be deduced from \hyperlink{theoremC}{Theorem C} and Proposition $4.1$ in \cite{Katok1996}, because $\msx$ has full support in $\Ws{x}$, for every $x\in M$.
\end{remark}

\begin{proposition}\label{pro:jucsisengulfing}
 $\mathscr{J}^{scu}$ is $\mm$-engulfing.
\end{proposition}

\begin{proof}
Consider $l, n_l$ as in the previous lemma. Arguing as in \Cref{cor:controljuliana} we deduce the existence of $D=D(l)>0$ such that for every $n-l\geq n_l, x\in M$,
\[
	\frac{\mcux[x](J^{cu}_{n-l}(x))}{\mcux[x](J^{cu}_{n}(x))}\leq D.
\] 
Define then $\widehat{J}^{scu}_n(x):=\clo{J^{scu}_{n-l}(x)}$  and observe that due to \Cref{cor:productstructurejulienne} and the inequality above one can deduce the existence of $\hat{D}>0$ such that
\[
\mm(\widehat{J}^{ucs}_n(x))\leq \hat D\cdot\mm(J^{ucs}_{n}(x)). 
\]
This together with \Cref{lem:doublingjscu} implies the $\mm$-engulfing property for $\mathscr{J}^{scu}$.
\end{proof}

The following is now clear.

\begin{corollary}\label{cor:XcoincideconsuspdLebesgue}
Every measurable set $X$ coincides $\mm$-almost everywhere with $\mathscr{D}(X;\mathscr{J}^{scu})$. 
\end{corollary}

We have thus established that if $X$ is $\mm$-essentially $s$-saturated then it coincides almost everywhere with the
$s$-saturated set $\mathscr{D}(X;\mathscr{J}^{scu})$. Interchanging $f$ by $f^{-1}$ and $\tmuf$ by $\muf$ we deduce that if $X$ is $\mm$-essentially $u$-saturated then it coincides almost everywhere with the $u$-saturated set $\mathscr{D}(X;\mathscr{J}^{ucs})$. To continue our reasoning, we now establish:

\begin{proposition}\label{pro.JcsigualJcu}
 Assume that $\muf=\tmuf$. Then for any measurable set $X$, $\mathscr{D}(X;J^{scu})=\mathscr{D}(X;J^{ucs})$.
\end{proposition}

The proof uses the following elementary lemma (cf.\@ Proposition $B.9$ in \cite{AccStaErg}). 

\begin{lemma}
Let $\mm$ be a Borel probability measure on $M$ and suppose that $\mathscr{E}^1=\{\{E_n^1(x)\}_{n\in \mathbb{N}}\}_{x\in M},$ $\mathscr{E}^2=\{\{E_n^2(x)\}_{n\in \mathbb{N}}\}_{x\in M}$ are decreasing regular Vitali bases. Suppose that there exist $l\in \Nat, D=D(l)>0$ such that for every $x\in M$,
\[
E_{n+l}^1(x)\subset E_{n}^2(x)\subset E_{n-l}^1(x)\quad \frac{\mm(E_{n+l}^1(x))}{\mm(E_{n}^1(x))}\geq D. 
\]
Then for every measurable $X \subset M$ it holds $\mathscr{D}(X;\mathscr{E}^1)=\mathscr{D}(X;\mathscr{E}^2)$.
\end{lemma}

\begin{proof}[Proof of \Cref{pro.JcsigualJcu}]
The core of the proof is to establish the existence of $l, n_l\in \Nat$ such that for every $x\in M,n-l\geq n_l$ it holds
\[
J^{scu}_{n+l}(x)\subset J^{ucs}_{n}(x)\subset J^{scu}_{n-l}(x).
\]

Take $y\in  J^{ucs}_{n}(x)$: consider $z,w$ such that  $y\in J^{u}_n(z), z\in J^s_n(w),w\in B^c_{\sigma^n}(x)$ and let $y'=\hs(y)\in \Wcu{x,\clps/2}$. Arguing as in the proof of \Cref{pro:holjulienne} we obtain
\[
d(z,w)\leq \varepsilon\lambda^n\Rightarrow d(y,y')\leq C_H\varepsilon^{\vartheta}\lambda^{n\vartheta}
\]
hence
\[
d(f^ny',f^nw)\leq \varepsilon(1+\lambda^n)+ C_H(\varepsilon)^{\vartheta}\lambda^{n\vartheta}\leq 2\varepsilon 	
\]
if $n$ is sufficiently large. We conclude that if $y''=\Wc{x,2\varepsilon}\cap \Wu{y',2\varepsilon}$, then
for sufficiently large $k$ and $n\geq k$
\[
d(y',y''), d(y'',w)\leq \sigma^{n-k}\Rightarrow y'\in J_{n-k}^u(y''),y''\in B_{\sigma^{n-k}}(w'),
\] 
and thus for some $l>k,n\geq l$ it holds $y'\in J^{cu}_{n-l}(x)$ and finally $y\in J^{scu}_{n-l}(x)$. The other inclusion is similar.

To finish we note that for a given $l$ and since $\mm=\tmuf=\muf$,
\[
	\frac{\mm(J^{ucs}_{n+l}(x))}{\mm(J^{ucs}_{n-l}(x))}
\]
is uniformly bounded from below in $x, n$, because $\mathscr{J}^{ucs}$ is $\mm$-engulfing (cf. \Cref{pro:jucsisengulfing}).
\end{proof}

\begin{corollary}
Assume that $\muf=\tmuf$. If $X$ is $\mm$-essentially $s$- and $u$- saturated, then $X$ coincides $\mm$-almost everywhere with the set bi-saturated set $\mathscr{D}(X)=\mathscr{D}(X;\mathscr{J}^{scu})=\mathscr{D}(X;\mathscr{J}^{ucs})$.
\end{corollary}

This allow us to conclude the Kolmogorov property.

\begin{theorem}\label{thm:propiedadK}
Assume that $\muf=\tmuf$ and either
\begin{enumerate}
	\item $f$ is accessible, or
	\item for every $x$, the transverse measure of $\mcux|\Wu{\Wc{x,\clps},\clps}$ on $\Wc{x,\clps}$ is Lebesgue (induced by the Riemannian metric).
\end{enumerate}
Then the system $(f,\tmuf)$ is Kolmogorov.	
\end{theorem}	

\begin{proof}
Assuming accessibility, the proof follows from our previous discussion: if $X\in\mathrm{Pin}(f,\mm)$ has positive measure then $X$ coincides $\mm$-almost everywhere with a bi-saturated set $\mathscr{D}(X)$, which has to be equal to $M$, therefore $\mm(X)=1$ and $\mathrm{Pin}(f,\mm)$ is trivial.

Alternatively, assume the second condition. Let $X\in\mathrm{Pin}(f,\mm)$ be of positive measure as before, and consider $\mathscr{D}(X)$; by changing $X$ by $\mathscr{D}(X)$ it is no loss of generality to assume that $X$ is bi-saturated. Take $x\in \mathscr{D}(X)$, and let $Y=M\setminus X$: by \Cref{lem:Xsusat} it follows that there exists $n_0$ so that for $n\geq n_0$,
\[
	\mcux(Y\cap J_n^{cu}(x))<\frac{1}{4}\mcux(J_n^{cu}(x)).
\]
If $x'\in \Ws{x}$ we get, since $Y$ is bi-saturated that and the fact that $\hs_{x',x}$ maps isometrically $B^c_{\sigma^n}(x)$ to $B^c_{\sigma^n}(x')$ (\Cref{lem:holisometry}),
\[
	J^{cu}_n(x')\cap Y=\bigcup_{y\in Y\cap B^c_{\sigma^n}(x')} J^s_n(u)
\]
Using that the corresponding transverse measures of $\mcux, \mcux[x']$ on $B^c_{\sigma^n}(x), B^c_{\sigma^n}(x')$ are Lebesgue we thus get 
\[
	\frac{\mcux[x'](Y\cap J_n^{cu}(x'))}{\mcux[x'](J_n^{cu}(x'))}=\frac{\mcux(Y\cap J_n^{cu}(x))}{\mcux[x](J_n^{cu}(x))}<\frac{1}{4}
\]
for every $n\geq n_0$. We remark that we have not compared the sizes of $J^u_n(y)$ and $J^u_n(\hs_{x',x}(y))$ for points $y\in B^c_{\sigma^n}(x)$, since the stable distance between $x,x'$ could be large. 

Now by minimality of $\Ws{x}$, again using that that the transverse measure of $\mcux[y]|\Wcu{y,\varepsilon}$ is Lebesgue and that $Y$ is bi-saturated, we deduce that for every $y\in M$,
\[
	\frac{\mcux[y](Y\cap J_n^{cu}(y))}{\mcux[y](J_n^{cu}(y))}\leq \frac{1}{4}.
\]
This implies that $\mathscr{D}(Y)=\emptyset$, and therefore $\mm(X)=1$. As in first part, we deduce that $(f,\muf)$ is Kolmogorov.
\end{proof}
 
 \begin{remark}
 See \Cref{subs:epilogue} for a discussion on the second possible hypothesis in the previous theorem.
 \end{remark}

\subsection{Bernoulli processes}
Here we establish that $(f,\tmuf)$ is isomorphic to a Bernoulli shift when $\varphi$ is either constant on center leaves (which includes the case when $\tmuf$ is the entropy maximizing measure) of the SRB potential. The hypotheses of \hyperlink{theoremA}{Theorem A} are assumed, in particular $\muf=\tmuf$, therefore $(f,\tmuf)$ is a Kolmogorov system. We rely heavily on the machinery developed by D. Ornstein and B. Weiss to establish the similar statement for the case where $f$ is the time-one map of an Anosov flow and $\muf$ is the Lebesgue measure \cite{GeoBernoulli}  (see also \cite{RatGibbs}). For the sake of completeness we repeat here the main concepts and arguments from that work that we will use. Nonetheless, some familiarity with the Ornstein and Weiss article is desirable to follow this section.

Throughout this part all probability spaces are standard, and all partitions of them are considered to be finite and ordered. If $\parP$ is a partition its $i$-th atom will be denoted $P_i$. For partitions $\parP,\parP[P']$  of the probability space $(X,\mathcal{B}_X,\mm)$ having the same number $k$ of atoms, its partition distance is defined as
\begin{equation*}
|\parP-\parP[P']|:=\sum_{i=1}^k \mm(P_i\vartriangle P_i').
\end{equation*}

\begin{definition} Let $\{\parP[P_j]\}_1^n,\{\parP[Q_j]\}_1^n$ be sequences of partitions of the probability spaces $(X,\mathcal{B}_X,\mm_X)$, $(Y,\mathcal{B}_{Y},\mm_Y)$ having the same number of atoms. Then its $\overline{d}$-distance is
	\begin{equation*}
	\overline{d}(\{\parP[P_j]\}_1^n,\{\parP[Q_j]\}_1^n):=\inf_{g_1, g_2}\frac{1}{n}\sum_1^n|g_1\parP[P_j]-g_2\parP[Q_j]|  
	\end{equation*}
	where $g_1, g_2$ are isomorphisms between $X, Y$ and some fixed Lebesgue space $(Z,\mathcal{B}_Z,\mm_Z)$ (say, the interval).
\end{definition}

To clarify the the meaning of this definition we observe the following. Given a partition $\parP[P]$ of $X$ and $x\in X$ the $\parP$-name of $x$ is simply the atom where is contained. Similarly, for a sequence of partitions $\{\parP[P_j]\}_1^n$ having the same number of elements the $\{\parP[P_j]\}_1^n$-name of $x\in X$ is the $n$-tuple $(P_{j_1(x)},\ldots,P_{j_n(x)})\in \parP[P_{j_1}]\times\cdots \parP[P_{j_n}]$ such that $x\in \cap_{k=1}^n P_{j_k(x)}$. It is straightforward to verify that for partitions $\parP[P], \parP[P']$ of $X$ the partition distance $|\parP-\parP[P']|$ is simply (twice) the expected value of the function $g_{\parP[P],\parP[P']}:X\rightarrow \{0,1\}$ such that 
\[
g_{\parP[P],\parP[P']}(x)=1\Leftrightarrow \text{the }\parP[P], \parP[P']\text{ names of }x\text{ are different}. 	
\]
Consequently, we have that $\overline{d}(\{\parP[P_j]\}_1^n,\{\parP[Q_j]\}_1^n)\leq r$ if and only if there is a way to re-code simultaneously $\{\parP[P_j]\}_1^n,\{\parP[Q_j]\}_1^n$ such that the average of the expected value of points having different $n$-name with respect to the encoding is less than equal to $r$. The following definition and lemma give an important tool to estimate the size of the $\overline{d}$ distance between two sequence of partitions.

\begin{definition}
A measurable map $g:(X,\mm_X)\rightarrow (Y,\mm_Y)$ is said to be $\epsilon$-measure preserving if there exists $\mathscr{Ex}(g)\subset X$ such that $\mu_X(\mathscr{Ex}(g))<\epsilon$ and with the property that $A\subset X\setminus \mathscr{Ex}(g)$ implies 
\[
	\left|\frac{\mm_Y(g(A))}{\mm_X(A)}-1\right|< \epsilon.
\]
In this case $\mathscr{Ex}(g)$ is the exceptional set of $g$. 
\end{definition}

\begin{lemma}\label{lem:lemanombres}
	Let $\{\parP[P_j]\}_1^n,\{\parP[Q_j]\}_1^n$  be partitions of the probability spaces $(X,\mathcal{B}_X,\mm_X)$, $(Y,\mathcal{B}_{Y},\mm_Y)$ having the same number of atoms, and assume that there exists $g:(X,\mm_X)\rightarrow (Y,\mm_Y)$\ $\epsilon$-measure preserving with exceptional set $\mathscr{Ex}(g)$ such that 
	\[
	x\not\in \mathscr{Ex}(g)\Rightarrow \frac{1}{n}\#\{1\leq j\leq n: \parP[P_j]-\text{name of }x\neq  \parP[Q_j]-\text{name of }gx\}\leq \epsilon.
	\]
	Then $\overline{d}(\{\parP[P_j]\}_1^n,\{\parP[Q_j]\}_1^n)\leq 16\epsilon$.
\end{lemma}

See lemma $1.3$ in \cite{GeoBernoulli} for the proof. We will usually apply the previous lemma with $Y\subset X$ a subspace, where $\mm_Y:=\mm_X(\cdot|Y)$ and $\parP[Q_j]=\parP[P_j]|Y$.

\smallskip

The following definition and theorem are central for what follows.

\begin{definition}
	Let $f:(X,\mathcal{B}_X,\mm_X)\rightarrow (X,\mathcal{B}_X,\mm_X)$ be an automorphism and $\parP$ a finite partition of $X$. We say that $\parP$ is very weak Bernoullian (VWB)\footnote{In \cite{GeoBernoulli} the VWB condition is defined with the roles of past and future interchanged. This however is a minor difference, since if a system is Bernoulli then its inverse is Bernoulli as well.} if given $\epsilon>0$ there exists $N_0$ such that for all $N'\geq N\geq N_0$ there exist a family of atoms $\mathscr{G}\subset \bigvee_{N'}^{N} f^{-i} \parP$ satisfying
	\begin{itemize}
		\item $\mm_X(\mathscr{G})>1-\ep$;
		\item for $n\geq 0$ it holds \(A\in \mathscr{G}\Rightarrow \overline{d}(\{f^{j}\parP\}_1^n,\{f^{j}\parP|A\}_1^n)\leq \epsilon\).
	\end{itemize}
	\end{definition}

\begin{theorem}[Ornstein]\label{thm:Ornstein} 
	Let $f:(X,\mathcal{B}_X,\mm_X)\rightarrow (X,\mathcal{B}_X,\mm_X)$ be an automorphism. 
	\begin{enumerate}
		\item If $\parP$ is VWB then $(X,\bigvee_{-\oo}^{\oo}f^n\parP,\mm_X)$ is isomorphic to a Bernoulli shift.
		\item Suppose that there exists an increasing family of $f$-invariant $\sigma$-algebras $\{\mathcal{A}_n\}_{n=1}^{\oo}$ satisfying:
		\begin{enumerate}
			\item $\mathcal{A}_n\nearrow \mathcal{B}_X$.
			\item For every $n$, $(X,\mathcal{A}_n,\mm_X)$ is isomorphic to a Bernoulli shift.
		\end{enumerate}
		Then $(X,\mathcal{B}_X,\mm_X)$ is isomorphic to a Bernoulli shift.
		\end{enumerate}
\end{theorem}

From now on we will specialize in our case of interest $(X,\mathcal{B}_X,\mm_X)=(M,\mathcal{B}_M,\mm=\tmuf)$. The idea is to show that there exist VWB partitions of arbitrarily small diameter and use the previous theorem to conclude that the process $(f,\tmuf)$ is Bernoulli. We fix a family of dynamical boxes $\mathscr{V}=\{V_i\}_{i=1}^R$ together with their corresponding $cu$-transversals $\{W_i\}_{i=1}^R$ (see \Cref{rem:otroeq}), and define $\tmuf$ using these families.

\begin{remark}\label{rem:SRBfmenos}
 In the SRB we will work with $f^{-1}$ instead of $f$, that is, we will prove that the system $(f^{-1},\tmuf)$ is Bernoulli, which in turn implies that $(f,\tmuf)$ is Bernoulli. This is made to guarantee that we control on the stable measures. Of course, one can argue directly with $(f,\muf)$, but the notation becomes heavier.
\end{remark}


Consider a partition $\parP$ such that each atom $P\in \parP$ is a proper set with diameter smaller than the Lebesgue number of $\{V_i\}_{i=1}^R$, and furthermore its boundary $\partial P_i$ is piecewise differentiable. Clearly there exists such partitions of arbitrarily small diameter and thus it is enough to prove that $\parP$ as above is VWB.

\begin{definition}
	Let $A,B\subset M$ with $B$ a dynamical box. We say that $x\in A\cap B$ is a point of ($s-$)tubular intersection if the stable plaque of $B$ containing $x$ is completely contained in $A\cap B$. If every point of $A$ is a point of tubular intersection with $B$ we say that $A$ is a $s$-band.
\end{definition}

\begin{lemma}\label{lem:sbandas}
	Let $C$ be a dynamical box and $\delta>0$. Then there exists $N_1$ such that for all $N'>N\geq N_1$ one can find $\mathscr{H}\subset \bigvee_{N}^{N'}f^{-n}\parP$ with $\mm(\bigcup_{H\in\mathscr{H}} H)\geq 1-\delta$ satisfying that if $A\in \mathscr{H}$ there exists $E\subset A$ with the properties
	\begin{enumerate}
		\item $\mm(E|A)\geq 1-\delta$.
		\item $E\cap C$ is a $s$-band.
	\end{enumerate}
\end{lemma}

Compare lemma $2.1$ in \cite{GeoBernoulli}. Note that in general the disintegration of $\mm$ into local stable manifolds is much worse behaved that the case considered by Ornstein and Weiss (where $\mm$ is Lebesgue).

\begin{proof}
Define the set 
\[
\mathcal{H}:=\{x:\lim_{n\mapsto\oo}\frac{\SB(x)}{n}=\int \varphi \der\mm\}.
\]
The system $(f,\mm)$ is ergodic and $\varphi$ is continuous, thus $\mathcal{H}$ is a full measure $s$-saturated set in $M$. By Egoroff's theorem we can find $\mathcal{H}_0\subset \mathcal{H}$ with $\mm(\mathcal{H}_0)\geq 1-\frac{\delta^2}{4}$ where the above convergence is uniform, hence there exists $N_0$ such that for all $N\geq N_0$, for all $x\in \mathcal{H}_0$ it holds 
\begin{equation*}
\SB(x)-nP=\SB(x)-n(\int \varphi \der\mm + h_{\mm}(f))<-\frac{nh_{\mm}(f)}{2}
\end{equation*}
(recall that $h_{\mm}(f)>0$ by \Cref{cor:entropos}). Let $\zeta:=\max\{\diam \Ws{x,C}:x\in C\}$, and  for $n\geq N_1$ define
\begin{align*}
B_n&=\cup_i\{x\in C\cap f^{-n}P_i: x\text{ is not of tubular intersection}\}\\
&=\cup_i\{x\in f^{n}C\cap P_i: x\text{ is not of tubular intersection}\}.   
\end{align*}
As $f^nC$ is a dynamical box, $B_n\subset \bigcup_{i}\bigcup_{x\in \partial P_i} f^n \Ws{f^{-n}x,\zeta}$. Now by formula \eqref{eq:invariancianu}, \Cref{lem:comparacionentretamanhos} and the fact that $\mathcal{H}_0$ is $s$-saturated we deduce the existence of a constant $d=d(\zeta)>0$ such that 
\[
\forall x\in \partial P_i\cap \mathcal{H}_0,\quad \nsx[x_w](f^n \Ws{f^{-n}x,\zeta})\leq e^{\SB(x_w)-nP}\leq de^{-n\frac{h_{\mm}(f)}{2}}
\]
where $P_i\subset V_i=V(P_i)\in\mathscr{V}$ and $x_w=\Ws{x,V_i}\cap W_i$, $W_i=\Wcu{z_i,V_i}$ being the chosen $cu$-transversal of $V_i$.

It follows that
\[
	\mm(B_n\cap \mathcal{H}_0)\leq \#\parP\cdot \sup_i \mcux[z_i](W_i)\cdot de^{-n\frac{h_{\mm}(f)}{2}}\leq d'e^{-n\frac{h_{m}(f)}{2}}.
\]

Define $\mathcal{B}:=\cup_{N_1}^{\oo}B_n$ where $N_1>N_0$ is such that $\mm(\mathcal{B}\cap \mathcal{H}_0)<\frac{\delta^2}{4}$.

\smallskip

\noindent\textbf{Claim 1:} For all $N'>N\geq N_1$ there exists $\mathscr{H}_1\subset  \bigvee_{N}^{N'}f^{-n}\parP$ such that $\mm(\cup_{A\not\in \mathscr{H}_1}A)\leq \frac{\delta}{2}$ and $A\in \mathscr{H}_1 \Rightarrow \mm(M\setminus\mathcal{H}_0|A)< \frac{\delta}{2}$. 
		
Otherwise $\mm(M\setminus\mathcal{H}_0)\geq \frac{\delta^2}{4}$, contradicting the choice of $\mathcal{H}_0$.

\smallskip 
 
\noindent\textbf{Claim 2:}  For all $N'>N\geq N_1$ there exists $\mathscr{H}_2\subset  \bigvee_{N}^{N'}f^{-n}\parP$ such that $\mm(\cup_{A\not\in \mathscr{H}_2}A)\leq \frac{\delta}{2}$ and $A\in \mathscr{H}_2 \Rightarrow \mm(\mathcal{B}\cap \mathcal{H}_0|A)< \frac{\delta}{2}$. 
		
Otherwise $\mm(\mathcal{B}\cap \mathcal{H}_0)\geq \frac{\delta^2}{4}$, contradicting the definition of $\mathcal{B}$.

\smallskip 
	 	
Let $\mathscr{H}=\mathscr{H}_1\cap\mathscr{H}_2$. Then $\mm(\cup_{A\in\mathscr{H}}A)\geq 1-\delta$ and for $A\in \mathscr{H}$ the existence of the $s$-band $E\subset A$ satisfying the conclusion of the lemma is easily deduced from the above two claims.
\end{proof}

Since $(f,\mm)$ has the Kolmogorov property we have the following.

\begin{lemma}\label{lem:Kconsecuencia}
	Let $C$ be a dynamical box and $\delta>0$. Then there exists $N_2\geq N_1$ such that for any $N'\geq N\geq N_2$ there exists $\mathscr{G}\subset \bigvee_{N}^{N'}f^{-n}\parP$ satisfying
	\begin{itemize}
	 	\item $\mm(\cup_{A\in \mathscr{G}})\geq 1-\delta$;
	 	\item $A\in \mathscr{G}\Rightarrow |\mm(C|A)-\mm(C)|<\delta$.
	 \end{itemize} 
\end{lemma}

Until now we did not make explicit assumptions for the potentials. The conclusion of the lemma below, a refinement of \Cref{pro:medidacs} for the conditional measures on $\Fs$, seems to require to some extra control on the potential, and in particular on the stable measures.

\begin{lemma}\label{lem:jacobianomedidas}
If $\varphi \equiv 0$ or $\varphi=-\log |\det Df|E^u|$ then given $\ep>0$ there exists $\delta>0$ so that for every $x, y\in$, if $A \subset \Ws{x, \frac{\clps}{4}}, B\subset \Ws{y, \frac{\clps}{4}}$ are $\delta-$equivalent relatively open and pre-compact sets, then
\[
	\left|\frac{\nsx(A)}{\nsx[y](B)}-1\right|<\ep.
\]
\end{lemma}

\begin{proof}
In the SRB case we interchange $f$ and $f^{-1}$ as explained in \Cref{rem:SRBfmenos}, and then the $\nsx$ are smooth volumes on the leaves, hence the conclusion of the Lemma follows. Therefore, we only consider the case when $\varphi$ is constant on center leaves. Due to \Cref{lem:JacobianoT}, it is enough the prove the same property for the measures $\msx$ instead of the $\nsx$.   

Recall that  
\[
	\mu^s_x=\pi^c_x\mcsx
\]
where $\pi_x^c:\Wc{\Ws{x},\ccen}\to\Ws{x}$ is the projection collapsing center plaques, $0<\ccen<\frac{\clps}{4}$. By the local product structure of the foliations, if $0<\delta<\frac{\clps}{8}$ then it follows that for every $x,y \in M$, if $A \subset \Ws{x,\frac{\clps}{4}}, B \subset \Ws{y,\frac{\clps}{4}}$ are $\delta$-equivalent then 
\[
	\hu(\Wc{B,\ccen-\delta}) \subset \Wc{A,\ccen} \subset \hu(\Wc{B,\ccen+\delta}), 
\]
where $\hu:\Wcs{y,\clps}\to \Wcs{x}$ is the local holonomy by unstable leaves, and $d(\hu(z),z)\leq 2\delta$ for every $z\in \Wc{B,\ccen+\delta}$. Now applying part $3$ of \Cref{pro:medidacu} we get the result.
\end{proof}

We are ready to prove the analogue of the Main Lemma of \cite{GeoBernoulli} (compare Lemma $8.5$ in \cite{LyaPesin}). 

\begin{proposition}\label{pro:mainOW}
Given $\delta>0$ there exists $\rho>0$ such that if $C$ is a dynamical box with diameter less than $\rho$ and $E\subset C$ is a $s$-band then there exists a bijection $\theta:E\rightarrow C$ such that 
\begin{enumerate}
	\item $\theta\mm(\cdot|E)$ is equivalent to $\mm(\cdot|C)$, with Jacobian $J(x)$ satisfying $|J(x)-1|<\delta$, for every $x\in C$.
	\item For all $n\geq 0,x\in E$ it holds $d(f^{-n}x,f^{-n}\theta(x))\leq \delta$.
	\end{enumerate}
\end{proposition}

\begin{proof}
Consider $\rho>0$ smaller than the Lebesgue number of the covering $\{V_i\}_{i=1}^R$. If $C$ is a dynamical box with $\diam C\leq \rho$ then it is contained in some $V=V_i$, and hence using \Cref{lem:independencebox} it is no loss of generality to assume $C=V$. We thus can write 
\[
C=\bigcup_{y\in \Ws{x_0,\epsilon}}\ \ \hs_{y,x_0}(W), W=\Pcu{x_0,\epsilon}	
\]
for some $x_0\in M, \ep>0$ (cf. \Cref{cor:dynamicalboxsymmetric}). Observe that if $E \subset C$ is an $s$-band, then
\[
	\mm(A|E)=\frac{1}{\mm(E)}\int_W \frac{\nu^s_w(\Ws{w}\cap A)}{\nu^s_w(\Ws{w,C})}\der \mcux{w}.
\]

Fix $E \subset C$ an $s-$band and let $\theta_{x_0}:W\cap E\to W$ be a bijective map that sends $\mcux[x_0]|W\cap E$ to $\mcux[x_0]|W$. For $x\in E$ define $\theta(x)=\hs_{x,x_0}\circ \theta_{x_0}\circ \hs_{x_0,x}(x)$: $\theta$ preserves $cu$-plaques of $C$, therefore $d(f^{-n}x,f^{-n}(\theta(x)))\leq d(x,\theta(x))\leq \rho$ for every $n\geq 0$. Since $E$ is an $s$-band, $\theta$ is bijective (and clearly measurable).

\smallskip 
 
For every $w\in W$, $\theta \nux[w](\cdot|\Ws{w,C})$ is sent to a measure $\tilde{\nu}^s_{\theta_{x_0}(w)}$ which, due to the previous Lemma, is uniformly comparable with $\nux[\theta_{x_0}(w)](\cdot|\Ws{\theta_{x_0}(w),C})$. From this the statement follows easily.
\end{proof}

Now we can apply Lemma $2.3$ of \cite{GeoBernoulli} to obtain the following.

\begin{lemma}\label{lem:mapadebandas}
	Let $\varepsilon,\delta>0$ be given. Then there exists $N$ with the property that for all $N'\geq N$ one can find $\mathscr{G}\subset \bigvee_{N}^{N'}f^{-n}\parP, E \subset M$ and $\theta:E\to M$ satisfying
	\begin{enumerate}
		\item $\mm(\cup_{A\in \mathscr{G}}A)\geq 1-\varepsilon$;
		\item $\theta$ is a bijective onto $\varepsilon$-measure preserving map from $(E,\mm(\cdot|E))$ to $(M,\mm)$; 
		\item for $A\in \mathscr{G}$ then $\theta(E \cap A)=A$ and 
	\begin{enumerate}
		\item $\mm(E\cap A|A)\geq 1-\varepsilon$;
		\item for all $n\geq 0,x\in E\cap A$ it holds $d(f^{-n}x,f^{-n}\theta(x))\leq \delta$.
	\end{enumerate}
	\end{enumerate}
\end{lemma}

For convenience of the reader we spell the proof.

\begin{proof}
Consider $0<\rho$ corresponding to $\delta$ in the previous proposition. Using Besicovitch covering lemma we construct a partition $\{U_0,U_1,\ldots,U_l\}$ such that 
\begin{enumerate}
	\item for $1\leq j\leq l$ the set $U_j$ is a dynamical box of diameter less than $\rho$.
	\item $\mm(U_0)\leq \frac{\varepsilon}{10}$. 
\end{enumerate}
Then using \Cref{lem:sbandas,lem:Kconsecuencia} we deduce that there exists $N$ such that for all $N'\geq N$ we can find $\mathscr{G}\subset \bigvee_{N}^{N'}f^{-n}\parP$ with $\mm(\cup_{A\in\mathscr{G}}A)\geq 1-\varepsilon$ and such that for $A\in \mathscr{G}, 1 \leq j \leq l$ there exists $E_A\subset A$ of relative measure bigger than $1-\varepsilon$ with the properties: 
	\begin{enumerate}
		\item $\sum_1^l|\mm(E_A|U_j)-\mm(B_j)|<\varepsilon$.
		\item $E_A\cap U_j$ is a $s$-band for $1\leq j\leq l$, $E_A\cap U_0=\emptyset$.
	\end{enumerate}
	We let $E:=\bigcup_{A\in \mathscr{G}} E_A\cup U_0$. Finally we use the previous proposition in each $E_A\cap U_j$, and by defining $\theta|U_0$ as the identity we get $\theta:E\rightarrow M$ satisfying the conclusion of the lemma.
\end{proof}

Now we finish the proof.

\begin{theorem}\label{thm:VWB}
	The partition $\parP$ is VWB.
\end{theorem}
\begin{proof}
Let $\epsilon>0$ be given. If $E \subset M$ and $j\in \Nat$, then for every $x\in M$
\[
	x\in\text{atom }A\text{ of } f^j\parP|E\Leftrightarrow f^{-j}x\in\text{atom }f^{-j}A\text{ of } \parP,
\] 
and similarly for the partition $f^j\parP$. It follows that for every $n\in\Nat, x\in E$ the $\{f^j\parP\}_0^n$- and $\{f^j\parP|E\}_0^n$- names are determined by the atoms of $\parP$ that the finite sequence $\{f^{-j}x\}_{j=0}^n$ visits. 

Take $\varepsilon=\frac{\ep}{64}$ and apply the previous Lemma with $\varepsilon, \delta$ where $\delta$ is chosen as follows. Since all atoms $A\in \parP$ have (piecewise) smooth boundaries by disregarding some additional atoms in $\mathscr{G}\subset \bigvee_{N}^{N'}f^{-n}\parP$ we can modify the $\varepsilon$-measure preserving map $\theta:(E,\mm(\cdot|E))\to (M,\mm)$ that we constructed to some other $2\varepsilon$-measure preserving map $\theta':(E',\mm(\cdot|E'))\to (M,\mm)$ such that
\[
x\not\in \mathscr{Ex}(\theta')\Rightarrow \forall n, \frac{1}{n}\#\{1\leq j\leq n: \parP-\text{name of }x\neq  \parP-\text{name of }\theta(x)\}\leq 2\varepsilon.
\]	
To see this note that points $x,y$ such that $d(x,y)<\delta$ and $x,y$ are in different atoms of $\parP$, necessarily are in $N_{\delta}=\cup_{i} N_{\delta}(\partial P_i)$, where $N_{\delta}(\partial P_i)$ denotes the $\delta$-tubular neighborhood of $\partial P_i$; since $\partial P_i$ is piecewise-smooth, this $\delta$-neighborhood is contained in the union of finitely many $\delta$-neighborhoods of smooth codimension one manifolds of zero $\mm$-measure. From here the claim follows, as one can choose $\delta$ so that $\mm(N_{\delta})$ is arbitrarily small.

If $A\in \mathscr{G}\subset \bigvee_{N}^{N'}f^{-n}\parP$, then by applying \Cref{lem:lemanombres} we conclude that for all $n\geq 0$ it holds
\[
\overline{d}(\{f^j\parP\}_1^n,\{f^j\parP|A\}_1^n)\leq \overline{d}(\{f^j\parP\}_1^n,\{f^j\parP|E\}_1^n)+ \overline{d}(\{f^j\parP|E\}_1^n,\{f^j\parP|A\}_1^n)\leq 16(2\varepsilon)+2\varepsilon<\epsilon.
\]
Hence $\parP$ is VWB.
\end{proof}

\Cref{thm:VWB,thm:Ornstein} imply \hyperlink{theoremC}{Theorem C}.




\section{Applications and concluding remarks} 
\label{sec:applications_and_concluding_remarks}

In this last section we apply our results, particularly of \Cref{sec:a_characterization_of_equilibrium_states}, to the rank-one case. We finish with some questions.

\subsection{Hyperbolic flows} 
\label{sub:hyperbolic_flows}

Let $\Phi=(f^t)_{t\in\Real}$ be an Anosov flow with invariant splitting $TM=E^s\oplus E^c\oplus E^u$, where $E^c$ is generated by tangents to flow direction. If this flow is transitive there is a dichotomy: either
\begin{enumerate}
	\item $\Phi$ is a suspension of an Anosov diffeomorphism, and therefore $M$ is a fiber bundle over $S^1$, with fibers transverse to the flow line, or
	\item both $\Fs,\Fu$ are minimal foliations. 
\end{enumerate}
See \cite{AnosovFlow}. In what follows we will focus in this last case. 

For Anosov flows the powerful of symbolic representations is available \cite{SymbHyp,RatnerMarkov}. With this one can establish the following.

\begin{proposition}\label{pro:medidasinestablesrankone}
Given a H\"older potential $\varphi:M\to\Real$ there exists $\mu^u\in\mathrm{Meas}(\Fu)$ satisfying:
\begin{enumerate}
	\item[a)] $\mu^u_x$ has full support inside $\Wu{x}$, and is non-atomic.
	\item[b)] $\forall t\in\Real, x\in M$, $f^{-t}\mu^u_{f^t(x)}=e^{\Ptop(\varphi) t-\int_0^t \varphi(f^s(\cdot))ds}\mu^u_x$.
	\item[c)] The map $\mu^u:x\to \mu^u_x$ is weakly continuous in the following sense: given $y\in \Wu{x,\epsilon_0}$ and $\hs_{y,x}:\Wu{y}\to\Wu{x}$ the locally defined Poincaré map, it follows that for any  $A\subset \Wu{x}$ relatively open and pre-compact it holds
	\[
	\mu_{y}^u(\hs_{y,x}(A))\xrightarrow[y\to x]{}\mu^u_x(A).
	\]
	Moreover the convergence is uniform in $x$.
\end{enumerate}
\end{proposition}

\begin{proof}
The construction is due to N. Haydn \cite{localproductstructureHaydn}, based in the original construction of Gibbs measures from Sinai \cite{Sinai_1972}. Details are given in the Appendix of the paper of M. Pollicot and C. Waldken \cite{Pollicott2001}. Part $c)$ is not explicitly stated by Haydn who only states (a stronger version of) this fact for holonomies corresponding to the strong stable foliation, although it clearly follows from this and part $b)$.    
\end{proof}

Let $f=f_1$ the time-one map. Given $\varphi$ we denote $\tilde{\varphi}$ its average along flow lines $\tilde{\varphi}(x)=\int_0^1 \varphi(f^tx)dt$. Observe that by the above proposition the family of measures $\mu^u=\{\mu^u_x\}_{x\in M}$ satisfy the same conditions as the measures along unstables that we constructed in \Cref{sub:the_conditional_measures_potentials_constant_along_the_center_and_srbs}, therefore all our construction goes through and we can determine an equilibrium state $\muf[\tilde{\varphi}]$ for the system $(f,\tilde{\varphi})$ having conditionals along unstables given by the family $\mu^u$. We now relate $\muf[\varphi]$ and the equilibrium states for the flow.

Denote by $\PTM{\Phi}{M}=\cap_{t\in\Real} \PTM{f_t}{M}$ the set of invariant measures for the flow. The topological pressure of the system $(\Phi,\varphi)$ is given by the variational principle (\cite{WaltersPres}, Corollary $4.12$)
\begin{equation}
\Ptop(\Phi, \varphi)= \sup_{\mathclap{\nu\in \PTM{\Phi}{M}}}\{h_{\nu}(\Phi_1)+\int \varphi d\nu\}, 
\end{equation}
and the definition of equilibrium state for this system is clear. In what follows we will write $\Eq(f,\varphi)$ for the set of equilibrium measures of $(f,\varphi)$, and analogously $\Eq(\Phi,\varphi)$ denotes the set of equilibrium states for $(\Phi,\varphi)$. Clearly $\Eq(\Phi,\varphi)\subset \Eq(f,\varphi)$, and although in principle these sets can be different, it is known that 
\[
\Ptop(\Phi,\varphi)=\Ptop(f,\varphi).
\]
On the other hand, using the previous line it is direct to check that $\Ptop(\Phi,\varphi)=\Ptop(f,\tilde{\varphi})$, therefore it follows
\[
	\Eq(\Phi,\varphi)=\{\mu\in \PTM{\Phi}{M}:h_{\mu}(f)+\int \tilde{\varphi} d\mu\}\subset \Eq(f, \tilde{\varphi}).
\] 

\begin{proposition}\label{pro:unicamedidaequilibrio}
If $\Phi$ is a transitive Anosov flow, then the sets $\Eq(\Phi,\varphi)$ and $\Eq(f,\tilde{\varphi})$ coincide. Therefore $\Eq(f, \tilde{\varphi})=\{\muf[\tilde{\varphi}]\}$.
\end{proposition}

\begin{proof}
Let $A(t,x)=\int_0^t \varphi(\Phi_sx)ds$; then $A:M\times\Real\to\Real$ is an additive cocycle over $\Phi$ ($A(t+s,x)=A(t,\Phi_sx)+A(s,x))$ and $A(1,x)=\tilde{\varphi}$. For transitive hyperbolic flows it is known (cf. \cite{BowenRuelle})that $\Eq(\Phi,\varphi)=\{\mu\}$, and in particular $\mu$ is ergodic for the flow (in fact $(\Phi,\mu)$ is a Bernoulli flow). Fix any $\mm\in \Eq(f,\tilde{\varphi})$ and note that since $f$ and $f^s$ commute, $f^s\mm\in \Eq(f,\tilde{\varphi} \circ f^s)$.

Now the cocycle property tells us that 
\begin{align*}
&A(1+s,x)=A(1,f^s(x))+A(s,x)=A(s,f(x))+A(1,x)\\
&\Rightarrow \tilde{\varphi}\circ f^s(x)-\tilde{\varphi}(x)=A(1,f^s(x))-A(1,x)=A(s,f(x))- A(s,x)
\end{align*}  
and $\tilde{\varphi}\circ f^s$ is cohomologous to $\tilde{\varphi}$, i.e.\@ its difference is of the form $k(f(x))-k(x)$ for some $k$ (that in this case is H\"older). It is simple to check that cohomologous potentials have the same equilibrium states, therefore $f^s\mm\in \Eq(f,\tilde{\varphi})$. The measure $\tilde{\mu}=\int_0^1 f^s\mm$ is invariant under the flow, and a equilibrium state for $(f, \tilde{\varphi})$ hence $\tilde{\mu}=\mu$. Since $\mu$ is ergodic for the map $f=f^1$ ($(\Phi, \mu)$ is mixing thus clearly $(f,\mu)$ is mixing), we deduce that
\[
	f^s\mm=\mu\quad\forall s\in[0,1]
\]
and in particular $\mm=\mu$.
\end{proof}

The above together with \hyperlink{theoremB}{Theorem B} implies \hyperlink{corollaryB}{Corollary B}. In the article \cite{ContributionsErgodictheory} use \hyperlink{corollaryB}{Corollary B} to investigate further properties of the horocyclic flow, as uniqueness of quasi-invariant measures. 

Here is the last application of these ideas.

\begin{corollary}\label{cor:uniquequasiinvariantcs}
Suppose that $\zeta^{cs}\in\mathrm{Meas}(\Fcs)$ is such that for every $x\in M$, $f^{-1}\zeta^{cs}_{fx}=e^{\tilde{\varphi}-P'}\zeta^{cs}_{x}$, where $P'\in\Real$. Then $\zeta^{cs}=\mu^{cs}$ and $P'=\Ptop(\varphi)$.
\end{corollary}

\begin{proof}
The quasi-invariance condition in the hypotheses permit us to compute the Jacobian of $\zeta^{cs}$ with respect to the unstable holonomy, and this coincides with the Jacobian of $\mu^{cs}$. We can thus use $\zeta^{cs}$ instead of $\mu^{cs}$ and construct a probability $\mm$ that is an equilibrium state for $(f, \tilde{\varphi})$ (cf.\@ \Cref{sub:theequilibrium}), and therefore it is the unique equilibrium state for $(\Phi,\varphi)$. Using the local product structure one gets that $\zeta^{cs}=\mu^{cs}$.
\end{proof}

\subsection{Epilogue}\label{subs:epilogue}

We finish our article with some questions and considerations. In this part the standing hypotheses are $f:M\to M$ is a $\mathcal C^2$ center isometry with minimal stable and unstable foliations, $\varphi:M\to \Real$ a Hölder potential.

\begin{question}
Assume the existence of an equilibrium state $\muf$ for $(f,\varphi)$ and families of measures $\mu^u,\mu^s$ satisfying the quasi-invariance condition $3$ of \hyperlink{theoremA}{Theorem A}. Does it follows uniqueness of the equilibrium state? 
\end{question}

This question has eluded us for some time, and for the time being the only cases where uniqueness is known are potentials similar to the ones treated in this article. The proofs are obtained by methods that do not rely on this property. Maybe one can start from the following.

\begin{question}
Assume that the existence conclusions of \hyperlink{theoremA}{Theorem A} are valid, together with the ergodicity of $\muf, \tmuf$. Is the case that $\muf=\tmuf$?
\end{question}

Since these measures are ergodic, it would suffice to establish that they are not singular with respect to each other. Here is example of such situation.

\begin{claim}
Assume that $\varphi$ is constant on center leaves and suppose that the center disintegration of the measures $\mu^{cu}, \mu^{cs}$ is Lebesgue, meaning: for every $x$ the disintegration of $\mcux$ on $\Pcu{x,\clps}$ by local center manifolds is given by absolutely continuous measures with uniformly bounded (independently of $x$) Jacobians, and likewise for $\mcsx$ on $\Pcs{x,\clps}$. Then $\muf=\tmuf$.
\end{claim}
\begin{proof}
It suffices to establish that $\tmuf<<\muf$. During the proof we will denote $a\sim b$ to indicate that the quotient $\frac{a}{b}$ is bounded above and below, with constants independent of $\epsilon$. Using the \Cref{pro:estructuraproducto} and the remark after it we get	
\begin{align*}
&\sup_{0<\epsilon<\cpsm} \frac{\tmuf(B(x,\epsilon))}{\muf(B(x,\epsilon))}\leq D\cdot \sup_{0<\epsilon<\cpsm} \frac{\msx(\Ws{x,\epsilon})\cdot \mcux(\Pcu{x,\epsilon})}{\mcsx(\Pcs{x,\epsilon})\cdot \mux(\Wu{x,\epsilon})}
\end{align*}  
where $D$ does not depend on $x$. Recalling the definition of $\msx, \mux$ in the $c$-constant case \Cref{eq.defmuxcconstant} and since $\mcsx$ has Lebesgue disintegration,
\[
\frac{\msx(\Ws{x,\epsilon})}{\mcsx(\Pcs{x,\epsilon})}\sim \left(\frac{\ccen}{\epsilon}\right)^c
\]
and likewise,
\[
\frac{\mcux(\Pcu{x,\epsilon})}{\mux(\Wu{x,\epsilon}}\sim \left(\frac{\epsilon}{\ccen}\right)^c.
\]
We conclude that there exists some constant $\widetilde{D}>0$ such that for every $x\in M$ it holds
\begin{equation}
\sup_{0<\epsilon<\cpsm} \frac{\tmuf(B(x,\epsilon))}{\muf(B(x,\epsilon))}\leq \widetilde{D}.
\end{equation}
This implies that $\tmuf<<\muf$ and $\frac{\der\tmuf}{\der\muf}(x)\leq \widetilde{D}$ for $\muf\aep(x)$, as consequence of the standard Lebesgue differentiation theorem for Borel measures. To check the absolute continuity, assume that $\muf(A)=0$ and let $r>0$ be arbitrary. Take a sequence of dynamical boxes $\{B(x_k,\epsilon_k)\}_{k\in\Nat}$ satisfying  
	\begin{enumerate}
		\item $A\subset \bigcup_{k\geq 0}B(x_k,\epsilon_k)$,
		\item $\sum_{k=0}^{\oo} \muf(B(x_k,\epsilon_k)) <\frac{r}{\widetilde{D}}$.
	\end{enumerate}
Then
\[	
\tmuf(A)\leq \sum_{k=0}^{\oo} \tmuf(B(x_k,\epsilon_k))\leq \widetilde{D} \sum_{k=0}^{\oo} \muf(B(x_k,\epsilon_k))<r,
\]
and since $r>0$ is arbitrary, $\tmuf(A)=0$.  
\end{proof}

The conditional measures of $\mu^{cs}_x$ and $\mu^{cu}_x$ along centrals seem to contain some important information about the equilibrium state. We can thus ask.

\begin{question}
When do the center conditional measures of $\mu^{cs}_x$ and $\mu^{cu}_x$ are absolutely continuous?
\end{question}

\smallskip

One possible tool to answer the previous question is using the construction of \Cref{sec:measures_along_foliations} as in the following example.

\begin{example}\label{ex:desintegracioncentral}
For $x\in M$ denote by $\lambda_x^c$ the normalized Lebesgue measure on $\Wc{x,\clps}$. Assume that $E^c$ is differentiable, and that $\varphi$ is $c$-constant. 

\begin{claim}
 For every $x\in M$ the conditionals measures of $\mcux|\Pcu{x,\clps}$ with respect to the center foliation are given by $\{\lambda_y^c\}_{y\in\Wc{x}}$.
\end{claim}

In particular the disintegration is defined everywhere. Similarly for $\mcsx$.

To see this fix $x\in M$, $0<\ep\leq\clps$ and let $B=\Pcs{x,\ep}$. For $n\geq 0$ let $B_n=f^n(B)$ and note that since $f$ is an center isometry then  $\forall y\in B_n$ the map $f^{-n}|: \Wc{y,B_n}\to \Wc{f^{-n}y,B}$ is an isometry with respect to the induced metric. As $\Fc$ is differentiable, the disintegration of the Lebesgue measure $\nu|B_n=\Leb^{cu}|B_n$ along center plaques is given by $\{\lambda_y^c\}_{y\in B_n}$, therefore for $\upsilon_n=f^{-n}\nu|B_n$ its disintegration is given by $\{f^{-n}\lambda_y^c\}_{y\in B_n}=\{\lambda_z^c\}_{z\in B}$. This in turn implies that the center disintegration of $\nu^n=e^{\SB}f^{-n}\nu$ is given by
\[
  (\nu^n)^c_y=\frac{e^{\SB}}{\int e^{\SB} \der(\upsilon_n)^c_y}(\upsilon_n)^c_y=(\upsilon_n)^c_y=\lambda_y^c.
\]
From this it follows that every section $\nu\in\mathcal{X}$ \eqref{eq:combinacionesconvexas}, in particular $\mu^{cu}$, has the same disintegration, proving our claim.
\end{example}

\begin{question}
What are in general the disintegrations of $\mu^{cs}$ and $\mu^{cu}$ along center leaves?
\end{question}

\smallskip 
 
Similarly, we can inquire about the transverse measures of $\mcux, \mcsx$ on (local) center leaves. For the geometrical proof of the Kolmogorov property \Cref{thm:propiedadK} we noted that a sufficient condition would be establishing that these measures are Lebesgue. We suspect that this is always the case.

\begin{question}
Is it true that the transverse (normalized) measures of 
\begin{align*}
\mcux|\Wu{\Wc{x,\clps},\clps}\\
\mcsx|\Ws{\Wc{x,\clps},\clps}
\end{align*}
are equal $\lambda_x^c$?
\end{question}



\section{Appendix: Proof of Theorem D}\label{sec:Appendix}\hypertarget{sec:Appendix}{}

Let $A:\Tor^2\to\Tor^2$ be a linear Anosov map, and consider $f:\Tor^2\times \Tor^2\to \Tor^4$ the group extension of $A$ given by
\[
	f(x,\theta_1,\theta_2)=(A\cdot x,\theta_1+\alpha_1 k(x),\theta_2+\alpha_2 k(x))
\] 
where $k:\Tor^2\to\Tor$ is analytic, not cohomologous to constant, and $(\alpha_1,\alpha_1)$ is a Liouville vector (see below). Then $f$ is a conservative analytic center isometry with center bundle parallel to the vertical fibers $\{0\}\times \Real^2$. This example was considered by Dolgopyat \cite{dolgolivsic} in other context. Due to a result of Katok, Theorem 1 in \cite{Katok1980a}, this map is a $K$-automorphism, therefore in particular its unstable foliation $\Fu$ is ergodic.

\begin{lemma}
$\Fu$ is minimal. 
\end{lemma}
\begin{proof}
Let $e^u$ be a unit vector in $\Tor^2$ generating the unstable bundle of $A$ and, consider the linear flow $\psi^u_t:\Tor^2\to\Tor^2$, $\psi^u_t(x)=x+t\cdot e^u$. The unstable foliation of $A$ are the orbits of $\psi^u$, and it holds
\[
	A\psi^u_t(x)=\psi^u_{\lambda t}(Ax)
\]
 where $\lambda>1$ is the unstable eigenvalue of $A$. It is not hard to verify that the unstable foliation of $f$ is the orbit foliation of the flow $\hat{\psi}^u_t:\Tor^4\to\Tor^4$ given by
 \[
 \hat{\psi}^u_t(x,\theta_1,\theta_2)=(\psi^u_t(x), \theta_1+\alpha_1l_t(x),\theta+\alpha_2l_t(x))
 \]
 where $l_t(x)=\sum_{n=1}^{\oo}k(A^{-n}x)-k(A^{-n}\psi^u_t(x))=\sum_{n=1}^{\oo} k(A^{-n}x)-k(A^{-n}x+\frac{t}{\lambda^n}e^u)$; since $k$ is differentiable the previous series is convergent. In particular $\hat{\psi}^u$ is conservative.

The flow $\psi^u$ is uniquely ergodic, $\Leb$ being its invariant measure, therefore by a classical result due to Furstenberg, either $\hat{\psi}^u_t$ is uniquely ergodic (and therefore minimal, since it preserves $\mu=\Leb\times d\theta_1\times d\theta_2$), or any $\hat{\psi}^u_t$ ergodic invariant measure has atomic disintegration along the partition given by the vertical tori (i.e., the center foliation of $f$). See Theorem $4.1$ in \cite{Furstenberg1961}. The later possibility is prohibited by Katok's result, as $\hat{\psi}^u_t$ is ergodic with respect to Lebesgue.
\end{proof}

The vector $(\alpha_1,\alpha_2)$ is chosen so that $\frac{\alpha_1}{\alpha_2}\not\in\mathbb{Q}$, and that for every $n\in\Nat$ there exists $m_{n,1},m_{n,2}\in \Nat$ such that
\begin{itemize}
	\item \(|\al_1m_{n,1}+\al_2m_{n,2}|<\frac{1}{m_{n,2}^n}\);
	\item \(m_{n,2}\geq m_{n,1}>n\).
\end{itemize}

Denote $d_{n}(x,\theta_1,\theta_2)=\exp(2\pi i \prodi{(\theta_1,\theta_2)}{(m_{n,1},m_{n,2})})$; then $d_n:\Tor^4\to\Tor$ satisfies 
\begin{align*}
d_n(f(x,\theta_1,\theta_2))=\exp(2\pi i \prodi{(\theta_1+\alpha_1 k(x),\theta_2+\alpha_2 k(x))}{(m_{n,1},m_{n,2})})\\
=\exp(2\pi ik(x) \prodi{(\alpha_1,\alpha_2 )}{(m_{n,1},m_{n,2})})\cdot d_n(x,\theta_1,\theta_2).
\end{align*}

Define $\varphi:M\to \Real$ by
\begin{align*}
\varphi(x,\theta_1,\theta_2)&=\sum_{n\geq 0} d_{n}(f(x,\theta_1,\theta_2))-d_{n}(x,\theta_1,\theta_2)\\
&=\sum_{n} \left(\exp\Big(2\pi ik(x) \prodi{(\alpha_1,\alpha_2 )}{(m_{n,1},m_{n,2})}\Big)-1\right) d_n(x,\theta_1,\theta_2).
\end{align*}

Observe that $\varphi$ is the $\mathcal{C}^{\oo}$ limit of coboundaries, therefore $\Eq(f,\varphi)=\Eq(f,0)$. It is clear on the other hand that the Lebesgue measure in $\Tor^4$ maximizes entropy

Assume that there exists a family of measures $\zeta^{u}=\{\zeta^u_L: L\in\Fu\}$ satisfying the quasi-invariance condition $f^{-1}\zeta_{fL}^u=e^{\Re(\varphi) - \htop(f)}\zeta_{L}$, and depending continuously on the point. Then by \hyperlink{theoremB}{Theorem B} one would have 
\[
	\zeta_L=e^{\omega_L}\cdot\mathscr{m}^u_L
\]
where $\mathscr{m}^u_L$ is the corresponding Margulis measure on $L$, and $\omega_L:L\to\Real_{>0}$ is continuous. We would thus get
\[
	e^{\Re(\varphi)-\htop(f)+\omega_L}\mathscr{m}^u_L=f^{-1}(e^{\omega_{fL}}\cdot \mathscr{m}^u_{fL})=e^{\omega_{fL}\circ f-\htop(f)}\mathscr{m}^u_L,
\]
and hence $\Re(\varphi)=\omega_{fL}\circ f-\omega_L$ (because the Margulis measures have full support). If $k(0)=0$ then $L_0=\Wu{0}$ is fixed, and this gives, for $\omega=\omega_{L_0}$,
\[
	\Re(\varphi)=\omega\circ f-\omega\quad \text{on }L_0.
\]
Since the family $\{\omega_{\Wu{x}}\}_x$ depends continuously on $x$ ($\mathscr{m}^u$ are essentially the Lebesgue measure), $\omega$ extends continuously to $M$, thus implying that $\Re(\varphi)$ is cohomologous to the zero function, with continuous transfer function. Likewise for $\Im(\varphi)$: thus there would exist $\phi:M\to \Comp$ continuous so that $\varphi=\phi\circ f-\phi$.

We can use Fourier expansion to find $\phi$; for $x\in\Tor^2$ we consider $\varphi(x,\cdot,\cdot):\Tor^2\to \Real$. If $(n_1,n_2)\in \Z^2$ we denote by 
\[
	e_{(n_1,n_2)}(\theta_1,\theta_2)=\exp\left(2\pi i\prodi{(n_1,n_2)}{(\theta_1,\theta_2)}\right)
\]
the corresponding element of the Fourier basis. Note that $d_n(x,\theta_1,\theta_2)=e_{(m_{n,1},m_{n,2})}(\theta_1,\theta_2)$. We compute
\begin{align*}
&\varphi(x,\cdot,\cdot)=\sum_{n} \left(e_{(m_{n,1},m_{n,2})}(k(x)\alpha_1, k(x)\alpha_2)-1\right)\cdot e_{(m_{n,1},m_{n,2})}\\
&\phi(x,\cdot,\cdot)=\sum_{(n_1,n_2)} a_{(n_1,n_2)}(x)\cdot e_{(n_1,n_2)}\shortintertext{where $a_n:\Tor^2\to\Comp$ is continuous, and}
&\varphi(x,\cdot,\cdot)=\phi(f(x,\cdot,\cdot))-\phi(x,\cdot,\cdot)=\sum_{(n_1,n_2)}\left(a_{(n_1,n_2)}(Ax)\cdot e_{(n_1,n_2,)}((k(x)\alpha_1, k(x)\alpha_2))-a_{(n_1,n_2)}(x)\right)\cdot e_{(n_1,n_2)}\\
&\Rightarrow \begin{dcases}
(I)\quad a_{(n_1,n_2)}(Ax)\cdot e_{(n_1,n_2,)}((k(x)\alpha_1, k(x)\alpha_2))=a_{(n_1,n_2)}(x)\quad (n_1,n_2)\neq (m_{n,1},m_{n,2})\\
(II)\quad \left(a_{(m_{n,1},m_{n,2})}(Ax)-1\right)\cdot e_{(m_{n,1},m_{n,2})}((k(x)\alpha_1,k(x)\alpha_2))=a_{(m_{n,1},m_{n,2})}(x)-1
\end{dcases}
\end{align*}
From $(I)$ we get that if $(n_1,n_2)\neq (m_{n,1},m_{n,2})$ then $|a_{(n_1,n_2)}(Ax)|=|a_{(n_1,n_2)}(x)|$; since $A$ is topologically transitive, $|a_{(n_1,n_2)}(x)|=c_{(n_1,n_2)}\in \Real$ constant. It follows that $c_{(n_1,n_2)}\cdot e_{(n_1,n_2,)}((k(x)\alpha_1, k(x)\alpha_2))=c_{(n_1,n_2)}$ for every $x$, and this implies that $c_{(n_1,n_2)}=0$ (because $k(x)$ is not constant). We argue similarly for $(II)$: fixed pair $(m_{n,1},m_{n,2})$, the function $x\mapsto |a_{(m_{n,1},m_{n,2})}(x)-1|$ has to be constant, and therefore is the zero function. It follows that $a_{(m_{n,1},m_{n,2})}(x)=1$ for every $x$. Putting together $(I), (II)$ we get that 
\[
	\phi=\sum_n d_n=\sum_{\crampedclap{(m_{n,1},m_{n,2})}}\ \ e_{(m_{n,1},m_{n,2})}.
\] 

This is a contradiction, due to the Riemann-Lebesgue lemma. We have shown that our assumption on the existence of families of measures $\zeta^{u}$ (and the corresponding one for $\Im(\varphi)$) is incorrect, thus establishing \hyperlink{theoremD}{Theorem D}. 

\begin{remark}
We have used a complex valued potential to simplify the computations, but clearly the same argument can be adapted using one real valued potential, by employing the Fourier basis of $\Tor^2$ that consists of products of sines and cosines.
\end{remark}



\newpage

\printbibliography

\addcontentsline{toc}{section}{References}
\end{document}